\documentclass[11pt,a4paper]{article}

\usepackage[english]{babel}
\usepackage{scalefnt}
\usepackage{amsfonts}
\usepackage{yfonts}
\usepackage{bbm}
\usepackage{amssymb}
\usepackage{amsmath}
\usepackage{amssymb}
\usepackage{mathrsfs}
\usepackage{amsthm}
\usepackage{graphicx}
\usepackage{caption} 
\usepackage{wrapfig}
\usepackage{xcolor}
\usepackage{dsfont}
\usepackage{empheq}
\usepackage{thmtools}
\usepackage{thm-restate}
\usepackage{stmaryrd}

\usepackage{hyperref}
\hypersetup{
  unicode = true,                           
  pdftoolbar = true,                        
  pdfmenubar = true,                        
  pdffitwindow = true,                      
  pdfauthor = {Moatti},              
  pdftitle = {A structure preserving hybrid finite volume scheme for semiconductor models with magnetic field on general meshes},                        
  pdfnewwindow = true,                      
  colorlinks = true,                        
  linkcolor = blue,                         
  citecolor = magenta,                      
  filecolor = violet,                       
  urlcolor = violet                         
}

\usepackage{authblk}

\newcommand{\email}[1]{\href{mailto:#1}{#1}}

\usepackage{amsmath,amsfonts,amssymb,mathtools}
\usepackage{upgreek}
\usepackage{nicefrac}
\usepackage{cancel}
\usepackage{comment}
\usepackage{enumerate}
\usepackage{graphicx}
\usepackage{tikzscale}
\usepackage{xcolor}
\usepackage{paralist}
\usepackage{pgfplots}
\pgfplotsset{every axis/.append style={
                    label style={font=\Large},
                    tick label style={font=\Large},
                    legend style={font=\Large}
                    }}
\usepackage{subcaption}
\usepackage{etoolbox}
\usepackage[skins]{tcolorbox}
\makeatletter
\patchcmd{\ttlh@hang}{\parindent\z@}{\parindent\z@\leavevmode}{}{}
\patchcmd{\ttlh@hang}{\noindent}{}{}{}
\makeatother
\usepackage{multirow}
\usetikzlibrary{patterns}

\usepackage{tabularx}

\pgfplotsset{compat=1.15}

\usepackage{enumitem}

\newtheorem{theorem}{Theorem}
\newtheorem{prop}{Proposition}
\newtheorem{lemma}{Lemma}
\newtheorem{rem}{Remark}

\usepackage[scale = .8]{geometry}

\DeclareMathOperator{\e}{e}
\DeclareMathOperator{\divergence}{div}
\renewcommand{\C}{\mathcal{C}}
\newcommand{\R}{\mathbb{R}}

\newcommand{\M}{\mathcal{M}}
\newcommand{\E}{\mathcal{E}}
\newcommand{\D}{\mathcal{D}}
\newcommand{\Diss}{\mathbb{D}}
\newcommand{\N}{\mathbb{E}}
\newcommand{\F}{\mathcal{F}}

\renewcommand{\G}{\mathcal{G}}

\newcommand{\s}{\sigma}
\newcommand{\V}{\underline{V}}
\newcommand{\ww}{\underline{\mathbbm{w}}}
\newcommand{\vv}{\underline{\mathbbm{v}}}
\renewcommand{\u}{\underline{u}}
\renewcommand{\v}{\underline{v}}
\newcommand{\w}{\underline{w}}

\newcommand{\one}{\underline{1}}
\newcommand{\zero}{\underline{0}}

\newcommand{\Nd}{\underline{N}}
\newcommand{\Pd}{\underline{P}}
\newcommand{\phid}{\underline{\phi}}
\newcommand{\psid}{\underline{\psi}}

\newcommand{\SG} {Scharfetter--Gummel }
\newcommand{\Mat}{\mathbb{M}}

\newcommand{\dmax}{a}

\definecolor{forestgreen}{rgb}{0.13,0.54,0.13}

\numberwithin{equation}{section}

\title{A structure preserving hybrid finite volume scheme for semiconductor models with magnetic field on general meshes}

\author{Julien Moatti\footnote{\email{julien.moatti@inria.fr}}}
\affil{Inria, Univ.~Lille, CNRS, UMR 8524 - Laboratoire Paul Painlev\'e, F-59000 Lille, France}

\begin{document}
\maketitle
\begin{abstract}
We are interested in the discretisation of a drift-diffusion system in the framework of hybrid finite volume (HFV) methods on general polygonal/polyhedral meshes.
The system under study is composed of two anisotropic and nonlinear convection-diffusion equations with nonsymmetric tensors, coupled with a Poisson equation and describes in particular semiconductor devices immersed in a magnetic field.
We introduce a new scheme based on an entropy-dissipation relation and prove that the scheme admits solutions with values in admissible sets - especially, the computed densities remain positive. 
Moreover, we show that the discrete solutions to the scheme converge exponentially fast in time towards the associated discrete thermal equilibrium.
Several numerical tests confirm our theoretical results.
Up to our knowledge, this scheme is the first one able to discretise anisotropic drift-diffusion systems while preserving the bounds on the densities.
\end{abstract}



\tableofcontents
\section{Introduction}
We are interested in the numerical approximation of a generalised anisotropic Van Roosbroeck's drift-diffusion system, inspired by the model introduced in~\cite{GaGa:96} by Gajewski and G\"artner to describe semiconductors devices immersed in exterior magnetic fields. 
Such a model differs from the classical drift-diffusion system from its anisotropic nature, induced by the magnetic field.
From a numerical point of view, this anisotropy leads to several difficulties, including the preservation of bounds on the solution (see the numerical results of~\cite{GaGa:96}). 
The main originality of our work is to address this difficulty using a nonlinear Hybrid Finite Volume method~\cite{EGaHe:10,CHHLM:22}, specifically designed to handle anisotropy while preserving the bounds.
Hence, up to our knowledge, the scheme introduced here is the first bound-preserving scheme for anisotropic drift-diffusion systems. 
From a general perspective, the Van Roosbroeck's drift-diffusion system, initially introduced in~\cite{VanRo:50}, is one of the fundamental models for the description and the simulation of semiconductor devices. It is a macroscopic model taking into account the densities of the charge carriers (the electrons, negatively charged, and the holes, of positive charge) alongside with an electrostatic potential induced by the spatial inhomogeneity of the charges. Mathematically speaking, the system consists of two parabolic convection-diffusion equations coupled with a Poisson equation on the electrostatic potential.  
Among the other different generalisations of the initial model proposed in~\cite{VanRo:50}, one can mention the incorporation of nonlinear terms of recombination-generation (see for example~\cite{MRS:90}), as well as the use of non-Boltzmann statistics~\cite{GaGr:89,FKF:17} to describe some specific physical situations (high carrier densities) or devices (for example, organic semiconductors \cite{vMCo:08}), which leads to the modification of the diffusive term into a nonlinear one.
Another generalisation of the system, as discussed above, is presented in~\cite{GaGa:96}: the authors consider that the semiconductor is immersed into an exterior magnetic field. Such a situation leads to the introduction of anisotropy tensors (which are nonsymmetric) in the convection-diffusion equations, related to the magnetic field.

In the present work, we consider a general model that encompasses the different features mentioned above. More precisely, let $\Omega$ be an open, bounded, connected polytopal subsets of $\R^d$, with $d \in \{1,2,3\}$ whose boundary $\partial \Omega$ is divided into two disjoint subset $\partial \Omega = \overline{\Gamma^D} \cup \overline{\Gamma^N}$, with $|\Gamma^D| > 0$. We are interested in the following problem, where the unknowns $N$, $P$ and $\phi$ are functions from $\R_+ \times \Omega$ to $\R$: 
\begin{equation} \label{pb:DD:evol}	
	\left \lbrace
	\begin{aligned}
        \partial _ t N - \divergence ( N \Lambda_N  \nabla (h(N) - \phi )  ) &= - R(N,P) &&\text{ in } \R_+ \times \Omega \\
		\partial _ t P - \divergence (P \Lambda_P  \nabla (h(P)  +\phi )  ) &=- R(N,P) &&\text{ in } \R_+ \times \Omega \\
		- \divergence ( \Lambda_{\phi} \nabla \phi  ) &=  C + P - N  &&\text{ in } \R_+ \times \Omega \\
		N =  N^D, \ P =  P^D \text{ and } \phi &=  \phi^D &&\text{ on }  \R_+ \times \Gamma^D \\
		N\Lambda_N \nabla ( h(N) -  \phi ) \cdot n = P \Lambda_P \nabla( h(P) + \phi ) \cdot n 
		& = 	\Lambda_\phi \nabla \phi  \cdot n  = 0 &&\text{ on } \R_+ \times \Gamma^N\\		
	 	N(0, \cdot)  =  N^{in} \text{ and } \ P(0, \cdot) &=  P^{in}   &&\text{ in } \Omega, 
    \end{aligned}
	\right.
\end{equation}
where $n$ denotes the unit normal vector to $\partial \Omega$ pointing outward $\Omega$.

Let us give some insight and explanations about this system, and make precise assumptions on the data. 
First, the unknowns $N$ and $P$ refer respectively to the densities of electrons and holes, and $\phi$ refers to the electrostatic potential. The densities $N$ and $P$ take values in the set of admissible densities $I_h = ]0, \dmax[$, where $\dmax \in ]0, +\infty]$ is the upper bound on the density.
$I_h$ is in fact the definition domain of the function $h : I_h \to \R$, which depends on the statistics used to describe the relation between the densities and the chemical potential. We assume that $h$ is a $\C^1$ function, such that $h'>0$ on $I_h$, with limits $\displaystyle \lim_{0} h = -\infty$ and  $\displaystyle \lim_{\dmax} h = +\infty$.
We denote by $g = h^{-1}: \R \to I_h$ the inverse function of $h$. Note that $g$ is positive, and that $g'> 0$ on $\R$. Moreover, we assume that there exists $g_0 >0$ such that for any real $s$, $g'(s) \leq g_0 g(s)$. We refer the reader to \cite{GaGr:89,FKF:17,FPF:18} for more details about these statistics functions, but let us emphasise here three classical cases that fall under the scope of our assumptions:
\begin{itemize}[topsep=0pt, partopsep=0pt , itemsep=0pt,parsep=0pt]
	\item[(i)] the Boltzmann statistics, for which $g=\exp$, $\dmax=+\infty$ and $h=\log$;
	\item[(ii)] the Blakemore statistics, for which $g(s) = (\gamma + \e^{-s})^{-1}$, $\dmax =\frac{1}{\gamma}$ and $h(s) = \log(s / (1-\gamma s) )$, with $\gamma > 0$ (a relevant physical value is  $\gamma = 0.27$);
	\item[(iii)] the Fermi-Dirac statistics of order $\nicefrac12$, for which $g(s) = \frac{2}{\sqrt{\pi}} \int_0^{+\infty} \frac{\sqrt{z}}{1 + \e^{z-s}}dz$ and $\dmax=+\infty$.
\end{itemize}
The statistics (ii) can be interpreted as an approximation of (iii) in low density situations (see~\cite{Blakemore:82}), and similarly (ii) can be approximated by (i) if the carrier densities are small enough.
For the purpose of future analysis, we define $H : x \to \int_1^x h(s)ds$ and $G : x \mapsto \int_{-\infty}^x g(s)ds$. Notice that $H$ and $G$ are strictly convex, and that $G$ is positive.
\newline
The function $C \in L^\infty(\Omega)$ is the doping profile of the semiconductor, which characterises the device under study. In practice, $C$ is a discontinuous function.
\newline
The term $R(N,P)$ corresponds to the recombination-generation rate, which can be interpreted as a reaction term between electrons and holes. We assume that this term has the following form:
\begin{equation}
R(N,P) = r(N,P) \left ( \e^{h(N) + h(P) } -1 \right ), 
\end{equation}
where $r$ is a continuous non-negative function, which can be nonconstant in space. Pertinent choices of $r$ are for example (see~\cite{Marko:86,MRS:90,FRDKF:17}) the Auger recombination $r(N,P) = c_N N + c_P P$, the Shokley-Read-Hall (SRH) term $r(N,P) = (\tau_N N + \tau_P P + \tau_C)^{-1}$, or no recombination term $r=0$.
\newline
The tensor $\Lambda_\phi \in L^\infty( \Omega , \R^{d\times d})$ is the (rescaled) permittivity of the medium.
We assume that it is a symmetric and uniformly elliptic tensor, in the sense that there exists $\lambda_\sharp^\phi \geq \lambda_\flat^\phi > 0$ such that $\forall \xi \in \R^d,$ 
\[
	\lambda_\flat^\phi |\xi|^2\leq \xi \cdot \Lambda_\phi \xi \leq \lambda_\sharp^\phi   |\xi|^2 , 
\]
In practical situations, the permittivity is often assumed isotropic, leading to a tensor of the form $\Lambda_\phi = \lambda^2 \epsilon \,  I_d$, where $\lambda$ is the rescaled Debye length of the system, which accounts for the nondimensionalisation, and $\epsilon : \Omega \to R_+^*$ is a uniformly positive function corresponding to the dielectric permittivity of the material. Note that relevant values of the Debye length can be very small, inducing some stiff behaviours. Moreover, since the devices are often made of different materials, the function $\epsilon$ can be non-regular, and exhibits discontinuities at junctions between different materials.
\newline
The tensors $\Lambda_N$ and $\Lambda_P$ are diffusion tensors in $L^\infty( \Omega , \R^{d\times d})$, related to the exterior magnetic field. We refer to~\cite{GaGa:96} for a detailed description of the semiconductor models with magnetic field. A typical example is the case $d=2$, where the magnetic field $\vec{B}$ is orthogonal to the device: letting $b= |\vec{B}|$, the tensors write
\begin{equation*}
	\Lambda_N = \frac{\mu_N}{1+b^2}\begin{pmatrix}
	1 & b \\ 
	-b & 1
	\end{pmatrix} 
	\text{ and }
	 \Lambda_P = \frac{\mu_P}{1+b^2}\begin{pmatrix}
	1 & -b \\ 
	b & 1
	\end{pmatrix} , 
\end{equation*}
where $\mu_N: \Omega \to \R_+^\star$ and $\mu_P: \Omega \to \R_+^\star$ are the rescaled mobilities of the electrons and holes.
For our analysis, we will consider general tensors, and only assume that the tensors are uniformly elliptic and bounded, in the sense that there exist $\lambda_\sharp \geq \lambda_\flat > 0$ such that for any $\xi \in \R^d,$ 
\[
	\lambda_\flat  |\xi|^2 \leq \xi \cdot \Lambda_N \xi , \ \lambda_\flat  |\xi|^2 \leq \xi \cdot \Lambda_P \xi \text{ and }
	 |\Lambda_N \xi|\leq\lambda_\sharp   |\xi|  , |\Lambda_P \xi|\leq\lambda_\sharp   |\xi|.
\]
Note that these tensors are nonsymmetric in general (if $b \neq 0$).
\newline
Concerning the boundary data, we consider mixed Dirichlet-Neumann conditions. On $\Gamma^D$, which corresponds to the ohmic contacts, we impose that the values of the densities and the electrostatic potential are equal to $N^D$, $P^D$ and $\phi^D$ which are the traces of $H^1$ functions, denoted the same way. On $\Gamma^N$, which corresponds to insulated contacts, we impose a null flux. 
Concerning the initial conditions, we impose the initial densities $N^{in}$ and $P^{in}$, which are in $L^\infty(\Omega)$.
We assume that the initial and boundary data are uniformly bounded in $I_h$: there exist $0< m   < M < \dmax$ such that, almost everywhere on $\Omega$, 
\begin{equation} \label{eq:bounddata}
	m \leq  N^{in} , P^{in}, N^D , P^D \leq M.
\end{equation}
In the sequel, we will use the generic term of ``data'' or ``physical data'' to refer to the objects previously introduced, namely 
$h$, $r$, $\Lambda_\phi$, $\Lambda_N$, $\Lambda_P$, $C$, $N^D$, $P^D$, $\phi^D$, $N^{in}$ and $P^{in}$.

The existence and uniqueness of the global solutions to the drift-diffusion system~\eqref{pb:DD:evol} have been originally investigated in the case of Boltzmann statistics ($h = \log$), see for example~\cite{Mock:74,Gajew:85,MRS:90}. Concerning models with general statistics described above, we refer to~\cite{GaGr:89,GLN:21}.
One of the main properties of the solutions is the fact that the carrier densities $N$ and $P$ take values in $I_h$: in the sequel we will say that the densities are $I_h$-valued. Another natural concern lies in the long-time behaviour of the system, that is to say the asymptotics of the solutions when $t \to +\infty$: under some assumptions on the data, the solutions are shown to converge exponentially fast towards some steady-state called thermal equilibrium. 
The thermal equilibrium is a specific steady-state $(N^e,P^e, \phi^e)$ for which the electrons and holes
currents cancel. Under the hypothesis on the function $h$, according to \cite{MaUn:93}, there is no vacuum in this state, which hence satisfies
$\nabla \left (h(N^e) - \phi^e  \right ) = 0$ and $\nabla \left (h(P^e) + \phi^e  \right ) = 0$ on $\Omega$.
Therefore, there exist some constants $\alpha_N$ and $\alpha_P$ such that $h(N^e) - \phi^e = \alpha_N$ and $ h(P^e) + \phi^e = \alpha_P$ in $\Omega$ (recall that $\Omega$ is connected). In particular, this implies the following compatibility condition on the boundary data:
\begin{equation}\label{eq:compcond}
	h(N^D) - \phi^D = \alpha_N \text{ and } h(P^D) + \phi^D = \alpha_P \text{ on }  \Gamma^D.
\end{equation}
Moreover, at the thermal equilibrium, there is no more recombination nor generation, that is to say $R(N^e,P^e) =  0$. This implies, according to the definition of $R$, that $h(N^e) + h(P^e) = 0$ if $r\neq 0$. Combined with the previous constraint, this imposes another compatibility condition on the constants if $r \neq 0$: 
\begin{equation} \label{eq:alpha}
\alpha_N + \alpha_P = 0 .
\end{equation}
Under these conditions, we get that $N^e = g(\alpha_N + \phi^e )$ and $P^e = g(\alpha_P - \phi^e )$. Thus, the electrostatic potential at the equilibrium $\phi^e$ is a solution to the following nonlinear Poisson equation (called Poisson-Boltzmann in case of Boltzmann statistics, i.e. if $g = \exp$):
\begin{equation} \label{pb:PoissonEqT}
	\left \lbrace
	\begin{aligned}
		- \divergence ( \Lambda_{\phi} \nabla \phi^e  ) & =  C +  g(\alpha_P - \phi^e ) - g(\alpha_N + \phi^e )  &&\text{ in } \Omega, \\
		\phi^e &=  \phi^D &&\text{ on }  \Gamma^D, \\
		\Lambda_\phi \nabla \phi^e  \cdot n  &= 0 &&\text{ on }  \Gamma^N.
	\end{aligned}
	\right.
\end{equation}	
In the sequel, we will assume that the boundary data are compliant with the thermal equilibrium and satisfy the compatibility conditions~\eqref{eq:compcond} and~\eqref{eq:alpha} (if $r \neq 0$). Notice that~\eqref{eq:compcond} and~\eqref{eq:bounddata} imply that $\phi^D$ is bounded on $\Gamma^D$. Hence, we assume that its lifting on $\Omega$, also denoted by $\phi^D$, is in $H^1(\Omega) \cap L^\infty(\Omega)$.
The analysis of the long-time behaviour of the drift-diffusion system was carried out in~\cite{GaGa:96} in the case of the Boltzmann statistics with a magnetic field, and in~\cite{GaGr:89} for other statistics (without magnetic field). It relies on the use of the entropy method, which consists in using a Lyapunov functional -the entropy- decaying with time. Such a method was initially used in kinetic theory and then extended to the study of various dissipative systems, as explained in \cite{ACDDJ:04,Junge:16}. 
Let us sketch here the main lines of this method: letting 
\begin{multline*}
\N(t) = \int_\Omega H(N) - H(N^e) - h(N^e)(N-N^e) + \int_\Omega H(P) - H(P^e) - h(P^e)(P-P^e)  \\
	+ \frac{1}{2}\int_\Omega \nabla(\phi-\phi^e) \cdot \Lambda_\phi \nabla(\phi-\phi^e),
\end{multline*}
\begin{multline*}
	\Diss(t)= \int_\Omega N \nabla( h(N) - \phi ) \cdot \Lambda_N \nabla( h(N) - \phi ) +
		\int_\Omega P \nabla( h(P) + \phi ) \cdot \Lambda_P \nabla( h(P) + \phi ) \\+ 
		\int_\Omega  R(N,P)\left ( h(N) + h(P) \right ) , 
\end{multline*}
one has the following entropy-entropy dissipation relation: 
\begin{equation} \label{eq:entrostruct}
	\frac{d \N}{dt} (t)= - \Diss(t).
\end{equation}
From the assumptions on $R$ and $h$, one deduces that the entropy $\N$ and the dissipation $\Diss$ are non-negative quantities. Therefore, the entropy decreases, and this implies the convergence of $(N(t),P(t),\phi(t) )$ towards $(N^e,P^e,\phi^e)$ when $t \to \infty$. 

Note that the system \eqref{pb:DD:evol} can also be interpreted as a Poisson--Nernst--Planck (PNP) system, which describes the evolution of charged particles, typically ions in a solution. Especially, the results presented in this article can be easily generalised to systems with more than two charge carriers.

From a numerical point of view, approximating the solutions to system~\eqref{pb:DD:evol} proves to be a challenging task due to various factors, including the nonlinear coupling between the equations, the discontinuity of the doping profile $C$, the stiffness induced by small values of the Debye length $\lambda$, the anisotropy in the convection-diffusion equations or the non-symmetric nature of the tensors.
Moreover, it is essential to design numerical schemes that preserve the qualitative physical properties of the continuous model. Here, the fact that the densities take values in $I_h$ as well as the long-time behaviour are characteristic features of the continuous system that should be preserved at the discrete level.
Another challenge lies in the variety of meshes the scheme can handle. Indeed, semiconductor devices are subject to boundary layers phenomena, which often require local mesh refinement to be performed: matching simplicial meshes are strongly constrained, and not very suitable for this purpose (see~\cite[Chapter 5.4]{Marko:86}).
Among the various numerical methods that have been proposed for solving drift-diffusion systems, the seminal work of Scharfetter and Gummel~\cite{SchGu:69} in one dimension with Boltzmann statistics presents the keystone idea of numerous schemes: to use the exponential relation between chemical potentials and densities to ensure a preservation of the steady-state as well as a good discrete long-time behaviour. This idea was generalised for different statistics (of the form $h(s) = s^\alpha$, which do not match the assumptions presented above and allow vacuum) in~\cite{BeCh:12} by Bessemoulin-Chatard. This scheme was proved to have a good long-time behaviour in \cite{BCCH:17,BCCH:19}, following ideas from Chainais-Hilairet and Filbet introduced in~\cite{CHFi:07} to analyse the long-time behaviour of an upwind TPFA scheme. The analysis strongly relies on the adaptation of the entropy method at the discrete level.
The schemes discussed above are essentially part of the two-points flux approximation (TPFA)~\cite{EGH:00}, which suffers from limitations: the mesh has to satisfy some orthogonality conditions, and the diffusion tensor has to be isotropic. In order to overcome these strong constraints, various finite volume methods have been introduced in the framework of the Poisson equation (see \cite{Droni:14}), using auxiliary unknowns. Regrettably, these schemes do not enjoy monotonicity properties, and especially they can have negative solutions. In \cite{CaGui:16,CaGui:17}, Canc\`es and Guichard proposed a solution to compensate this limitation, introducing nonlinear schemes based on the entropic structure of the problem.
The specific question of approximating drift-diffusion systems on general meshes was already investigated in the past. 
One can cite the discrete duality finite volume (DDFV, see \cite{Herm:00,DoOm:05} for a presentation and analysis of these schemes applied to Poisson equations) scheme of Chainais-Hillairet \cite{Chainais:09}, which can handle general polygonal meshes with $d=2$ in the framework of Boltzmann statistics.
More recently, Su and Tang designed and analysed a scheme for Poisson--Nernst--Planck systems on general meshes in~\cite{SuTa:22}. This scheme is based on two discretisation methods: a virtual element method \cite{BdVBCMMR:13} 
for the Poisson equation alongside with a positive nonlinear finite volume method~\cite{BlLa:14,CaHe:16} for the convection-diffusion ones. The scheme is shown to admit solutions with positive densities and to have an entropic structure. However, it is restricted to the Boltzmann statistics with pure homogeneous Neumann boundary conditions and the Debye length is assumed to be $1$.
Apart from these finite volume schemes, many finite elements schemes were also designed for this type of drift-diffusion systems. Among them, one can cite the mixed and hybrid exponential fitting schemes~\cite{BrMaPi:89,BMaPi:89} of Brezzi, Marini and Pietra, which are proposed as finite elements generalisations of the \SG scheme~\cite{SchGu:69}, for linear diffusion on simplicial meshes. These schemes were then generalised to other statistics, especially in dimension 2, by J\"ungel and Pietra in~\cite{JuPi:97}. For other references about these schemes and their extensions, we refer the reader to~\cite{BMMPSW:05}.
Up to our knowledge, the only existing scheme which can handle the presence of a magnetic field in the model (anisotropic tensors $\Lambda_N$ and $\Lambda_P$) is the one introduced and analysed by Gajewski and G\"artner in~\cite{GaGa:96}.
It can be seen as a modified \SG scheme, and is hence restricted to Boltzmann statistics ($h = \log$). Moreover, the scheme is restricted to triangular meshes, and bound to some strong constraints between the mesh geometry and the magnetic field intensity. Last, the scheme does not preserve the positivity of the densities in the presence of a strong magnetic field.  


In this article, we are concerned with the design and the analysis of a numerical scheme for~\eqref{pb:DD:evol} that preserves the continuous features of the system (long-time behaviour and $I_h$-valuation of the densities). One of our main concerns is the ability of the scheme to handle the large variety of possible data: statistics function $h$, recombination-generation term $R$, physical data (doping $C$, Debye length $\lambda$, initial and boundary data, magnetic field) and discretisation data (mesh, time step).
Note that among the different motivations, for the support of general meshes is the desire to make local mesh adaptation much simpler in order to capture the boundary layers.
Such an approach based on a local adaptation has already been investigated in 1D~\cite{CaHu:19}, but should become much more complicated or even impossible to use in higher dimension on constrained meshes. Hence, the scheme presented here could be a possible way to use local mesh refinement in 2D or 3D.
To devise such a scheme, we use the hybrid finite volume (HFV) method, introduced in~\cite{EGaHe:10} by Eymard et al. in the framework of stationary diffusion problems. 
This method entails several interesting features: it can handle anisotropic and heterogeneous diffusion tensors alongside with very general polytopal meshes, and it benefits from a unified 1D/2D/3D formulation.
The scheme introduced here and its analysis are based on the entropy structure \eqref{eq:entrostruct} of the system, following the ideas of \cite{BCCHV:14,BCCH:17} (in the framework of drift-diffusion systems with TPFA schemes) and \cite{CaGui:17,CCHKr:18,CHHLM:22} (in the framework of advection-diffusion equations, with positivity-preserving schemes supporting anisotropy and general meshes). 
In particular, the nonlinearity induced by the function $h$ is discretised along the principles of the nonlinear HFV scheme introduced and analysed in~\cite{CHHLM:22} for advection-diffusion equations.
Using this approach, we get a unified scheme for the system~\eqref{pb:DD:evol}, robust with respect to the various data of the problem.
The main result of this paper, stated in Theorem~\ref{th:DD}, is the existence of solutions (with $I_h$-valued discrete densities) to the nonlinear scheme. In Theorem~\ref{th:longtime}, we establish the exponential decay of the discrete solutions towards the associated thermal equilibrium. These theoretical results are validated by various numerical simulations.

The article is organised as follows. In Section~\ref{sec:HFV}, we introduce the HFV framework (mesh, discrete unknowns, discrete operators), present the schemes for the equations~\eqref{pb:DD:evol} and a generalisation of~\eqref{pb:PoissonEqT} (namely, the semi-linear Poisson equation \eqref{pb:NLPoisson}), and state our main result. 
In Section~\ref{sec:sta}, we show that the stationary scheme is well posed, and discuss an important consequence, namely the correspondence between discrete densities and discrete quasi-Fermi potentials. 
In Section~\ref{sec:DD}, we analyse the scheme for the transient problem, showing that it has an entropy structure and admits solutions with $I_h$-valued densities whose  long-time behaviour mimics that of the continuous solutions.
In Section~\ref{sec:num}, we discuss the implementation of the schemes, and give some numerical evidences of our theoretical results. 
Finally, in Appendix~\ref{ap:lemma}, we state and prove a technical result that is instrumental to show the existence of solutions.

\section{Discrete setting and schemes} \label{sec:HFV}

The aim of this section is to recall the HFV framework for diffusive problems, and introduce the schemes for the discretisation of the drift-diffusion system~\eqref{pb:DD:evol} and the nonlinear Poisson equation~\eqref{pb:NLPoisson}. Here, for the purpose of analysis, we present the schemes from the finite element viewpoint. 
We refer the reader to Section~\ref{sec:flux} for a presentation of the schemes within the framework of finite volume methods (especially for the purpose of implementation).
\newline 
For a detailed presentation of the HFV method in the framework of a steady variable diffusion problem with symmetric tensor, we refer the reader to~\cite{EGaHe:10}.

\subsection{Mesh, discrete unknowns and boundary data} \label{sec:discre}

The definitions and notation we adopt for the discretisation are essentially the same as in \cite{EGaHe:10}.
A discretisation of the (open, bounded, connected) polytopal set $\Omega \subset \R ^d $, $d\in \{1, 2,3\}$, is defined as a triplet $\mathcal{D} = ( \M, \E, \mathcal{P})$, where:
\begin{itemize}
		\item $\M$ (the mesh) is a partition of $\Omega$, i.e., a finite family of nonempty disjoint (open, connected) polytopal subsets $K$ of $\Omega$ (the mesh {cells}) such that (i) for all $K \in \M$, $|K|>0$, and (ii) $\overline{\Omega} = \bigcup _ {K \in \M } \overline{K} $.
	
		\item $\E$ (the set of faces) is a partition of the mesh skeleton $\bigcup_{K\in\M}\partial K$, i.e., a finite family of nonempty disjoint (open, connected) subsets $\s$ of $\overline{\Omega}$ (the mesh faces, or mesh edges if $d = 2$) such that for all $\s\in\E$, $|\s|>0$ and there exists $\mathcal{H}_\s$ affine hyperplane of $\R^d$ such that $\s\subset\mathcal{H}_\s$, and $\bigcup_{K\in\M}\partial K=\bigcup_{\s\in\E}\overline{\s}$. We assume that, for all $K \in \M$, there exists $\E_K \subset \E$ (the set of faces of $K$) such that $\partial K = \bigcup_{\s \in \E_K} \overline{\s}$. For $\s \in \E$, we let $\M_\s  =  \lbrace K \in  \M \mid \s \in \E_K \rbrace $ be the set of cells whose $\s$ is a face. Then, for all $\s \in \E$, either $\M_\s=\{K\}$ for a cell $K\in\M$, in which case $\s$ is a boundary face ($\s \subset \partial \Omega$) and we note $\s \in \E_{ext}$, or $\M_\s = \lbrace K,L \rbrace$ for two cells $K,L\in\M$, in which case $\s$ is an interface and we note $\s = K|L \in \E_{int} $.
		
		\item $\mathcal{P}$ (the set of cell centres) is a finite family $\{x_K\}_{K \in \M}$ of points of $\Omega$ such that, for all $K \in \M$, (i) $x_K \in K$, and (ii) $K$ is star-shaped with respect to $x_K$. Moreover, we assume that the Euclidean (orthogonal) distance $d_{K,\s}$ between $x_K$ and the affine hyperplane $\mathcal{H}_\s$ containing $\s$ is positive (equivalently, the cell $K$ is strictly star-shaped with respect to $x_K$).
\end{itemize}
For a given discretisation $\D$, we denote by $h_\D>0$ the size of the discretisation (the meshsize), defined by $h_\D  = \underset{K \in \M}{\sup} h_K $ where, for all $K \in \M$, $h_K = \underset{x,y \in \overline{K}}{\sup} |x-y|$ is the diameter of the cell $K$.
For all $\s\in\E$, we let ${\overline{x}_\s} \in \s$ be the barycentre of $\s$.
Finally, for all $K\in\M$, and all $\s \in \E_K$, we let $n_{K,\s} \in \R^d$ be the unit normal vector to $\s$ pointing outward $K$, and $P_{K, \s}$ be the (open) pyramid of base $\s$ and apex $x_K$ (notice that, when $d=2$, $P_{K,\s}$ is always a triangle). Since $|\s|$ and $d_{K,\s}$ are positive, we have $|P_{K,\s}|=\frac{|\s|d_{K,\s}}{d}>0$.
We depict on Figure \ref{Maillage} an example of discretisation. Notice that the mesh cells are not assumed to be convex, neither $x_K$ is  assumed to be the barycentre of $K\in\M$. Moreover, hanging nodes are seamlessly handled with our assumptions, so that meshes with non-conforming cells are allowed (see the orange cross in Figure \ref{Maillage}; the cell $K$ therein is treated as an hexagon).
\begin{figure}[h!]
\begin{center}
{\scalefont{0.9}
\def\svgwidth{0.7\textwidth}
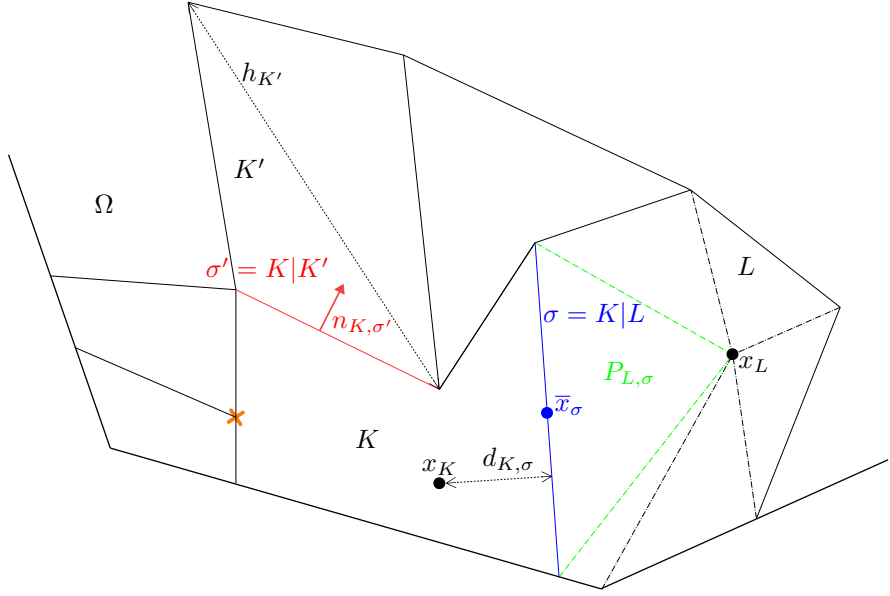
}
\end{center}
\caption{Two-dimensional discretisation and corresponding notations.}
\label{Maillage}
\end{figure}
\newline
We consider the following measure of regularity for the discretisation (which is the same as in \cite[Eq.~(2.1)]{CHHLM:22}): 
\begin{equation} \label{def:regmesh}
  \theta_\D  =  \max \left (  \underset{K \in \M, \s \in \E_K}{\max} \frac{h_K}{d_{K,\s}} , \underset{\s \in \E, K \in \M_\s}{\max} \frac{h_K^{d-1}}{|\s|} \right ).
\end{equation}

We now introduce the set of (hybrid, cell- and face-based) discrete unknowns: 
\begin{equation*}
		\V _\D   =  \big\lbrace \v_{\D} =\big( (v_K )_{K \in \M } ,  (v_\s)_{\s \in \E} \big) \mid v_K \in \R\;\forall K\in\M, v_\s \in \R\;\forall\s\in\E \big\rbrace.
\end{equation*}
Given a mesh cell $K\in\M$, we let $\V_K = \R\times\R^{|\E_K|}$ be the restriction of $\V_{\D}$ to $K$, and $\v_K=\big(v_K,(v_{\s})_{\s\in\E_K}\big)\in\V_K$ be the restriction of a generic element $\v_{\D}\in\V_{\D}$ to $K$.
Also, for $\v_{\D}\in\V_{\D}$, we let $v_\M : \Omega \to \R$ and $v_\E : \bigcup _{K\in\M} \partial K \to \R$ be the piecewise constant functions such that
\[
	v_{\M\mid K} = v_K \text{ for all } K\in\M, \quad\text{ and }  \quad v_{\E\mid\s} = v_\s \text{ for all } \s \in \E.
\]
For further use, we let $\one_{\D}$ (respectively $\zero_\D$) denote the element of $\V_{\D}$ with all coordinates equal to $1$ (respectively  $0$).
Also, given a function $f: \R \to \R$, and with a slight abuse in notation, we denote by $f(\v_{\D})$ the element of $\V_{\D}$ whose coordinates are the $(f(v_K))_{K\in \M}$ and the $(f(v_\s))_{\s \in \E}$. Given $\u_\D$ and $\v_\D$ two elements of $\V_\D$, we say that $\u_\D \leq \v_\D$ (respectively $<$, $\geq$, $>$) if and only if for any $K \in \M$ and $\s \in \E$, $u_K \leq v_K$ and $u_\s \leq v_\s$ (respectively $<$, $\geq$, $>$).
In particular, a vector of unknowns $\v_\D \in \V_\D$ is $I_h$-valued if and only if
$\zero_\D < \v_\D < \dmax \one_\D$. Note that $I_h$-valued vectors have therefore positive coordinates.

We assume that the discretisation $\D$ is compliant with the partition $\partial \Omega = \overline{\Gamma^D} \cup \overline{\Gamma ^N}$ of the boundary of the domain, in the sense that the set $\E_{ext}$ can be split into two disjoint subsets $\E_{ext}^D = \left \{ \s \in \E_{ext} \mid \s \subset \Gamma^D \right \}$ and $\E_{ext}^N = \left \{ \s \in \E_{ext} \mid \s \subset \Gamma^N \right \}$ such that $\E_{ext} = \E_{ext}^D \cup \E_{ext}^N$. Notice that, since $|\Gamma^D|>0$, one has $|\E_{ext}^D|\geq 1$. We define the following subspace of $\V_{\D}$, enforcing strongly a homogeneous Dirichlet boundary condition on $\Gamma^D$:
\[
	\V_{\D,0}  =  \left \lbrace \v_{\D} \in \V_{\D} \mid \forall \s \in \E_{ext}^D, \ v_\s = 0 \right \rbrace.
\]
In view of the upcoming analysis, we define a discrete counterpart of the $H^1$ semi-norm. Locally to any cell $K\in\M$, we let, for any $\v_K\in\V_K$, $|\v_K|_{1,K}^2 = \sum _{\s \in \E _K} \frac{|\s|}{d_{K,\s}}(v_K - v_\s)^2$. At the global level, for any $\v_{\D} \in \V_\D$, we let
\begin{equation*}
|\v_{\D}|_{1, \D}  =  \sqrt{\sum_ {K \in \M} |\v_K|_{1,K}^2 }.
\end{equation*}
Notice that $|\cdot|_{1,\D}$ does not define a norm on $\V_\D$, but if $|\v_{\D}|_{1, \D}= 0$, then there is $c \in \R$ such that $\v_{\D} = c \,\one_{\D}$ ($\v_{\D}$ is constant).
Thus, $|\cdot|_{1,\D}$ defines a norm on the space $\V_{\D,0}$.

We finally introduce discrete hybrid (Dirichlet) boundary data. The definitions below are motivated by the need to preserve the compatibility condition~\eqref{eq:compcond} at the discrete level. First, we define the discrete liftings of the densities $\Nd_\D^D$ and $\Pd_\D^D$ as the elements of $\V_\D$ such that for any $K \in \M$ and any $\s \in \E$, 
\begin{equation}\label{def:discreteN}
	N_\s^D = g \left ( \frac{1}{|\s|} \int_\s h(N^D) \right ) \text{ and }
	N_K^D = g \left ( \frac{1}{|K|} \int_K h(N^D) \right ), 
\end{equation}
\begin{equation}\label{def:discreteP}
	P_\s^D = g \left ( \frac{1}{|\s|} \int_\s h(P^D) \right ) \text{ and }
	P_K^D = g \left ( \frac{1}{|K|} \int_K h(P^D) \right ) .
\end{equation}
We also define the hybrid interpolate $\phid_\D^D \in \V_\D$ of the lifting $\phi^D$, defined by 
$\phi^D_K = \frac{1}{|K|}\int_K \phi^D$ for any $K \in \M$ and $\phi^D_\s = \frac{1}{|\s|}\int_\s \phi^D$ for all $\s \in \E$. One has (see~\cite[Proposition B.7]{DEGGH:18}) the following stability results: 
\begin{equation}\label{eq:boundlifting}
	|\phid_\D^D|_{1,\D} \leq C_{l,\Gamma^D} \|\phi^D\|_{H^{\nicefrac12} (\Gamma^D)} \text{ and }
		- \|\phi^D \|_{L^\infty(\Omega)} \one_\D \leq \phid_\D^D \leq \|\phi^D \|_{L^\infty(\Omega)} \one_\D, 
\end{equation}	
where $C_{l,\Gamma^D}$ is a positive constant only depending on $\theta_\D$, $\Omega$ and $\Gamma^D$. By~\eqref{eq:compcond} and \eqref{def:discreteN}-\eqref{def:discreteP}, the following compatibility condition, discrete counterpart of~\eqref{eq:compcond}, holds:
\begin{equation}\label{eq:compconddiscret}
	\forall \s \in \E_{ext}^D, \ h(N_\s^D) - \phi_\s^D = \alpha_N
	\text{ and } 
	 h(P_\s^D) + \phi_\s^D = \alpha_P.
\end{equation}
\begin{rem}[Discrete liftings]\label{rem:lifting}
The definition of the discrete boundary densities~\eqref{def:discreteN} and~\eqref{def:discreteP} is driven by the desire to obtain the discrete compatibility condition~\eqref{eq:compconddiscret}. Such definition boils down to discretise quantities carrying a physical meaning, homogeneous to the quasi-Fermi potentials. Moreover, this definition ensures that $\Nd_\D^D$ and $\Pd_\D^D$ are $I_h$-valued.
Notice that in practice, if $N^D$ is regular enough, one can choose to approximate $ \frac{1}{|\s|} \int_\s h(N^D)$ by $h(N^D(\overline{x}_\s))$ (respectively, $ \frac{1}{|K|} \int_K h(N^D)$ by $h(N^D(x_K))$ if $x_K$ is chosen to be the barycenter of $K$), and therefore find a classical expression for the discrete liftings, namely $N_\s^D = N^D(\overline{x}_\s)$ and $N_K^D = N^D(x_K)$.
The same holds true for $\Pd_\D ^D$.
\end{rem}
\subsection{Foundations of the hybrid finite volume method}
The HFV method hinges on the definition of a discrete gradient operator $\nabla_{\D}$, that maps any element $\v_{\D}\in\V_{\D}$ to a piecewise constant $\R^d$-valued function on the pyramidal submesh of $\M$ formed by all the $P_{K,\s}$'s, for $K\in\M$ and $\s\in\E_K$. More precisely, for all $K\in\M$, and all $\s\in\E_K$,
\begin{equation*}
  {\nabla_{\D}\v_{\D}}_{\mid K} = \nabla_K\v_K\quad\text{with}\quad{\nabla_K \v_K}_{\mid P_{K,\s}}  =  \nabla_{K,\s} \v_K= G_K \v_K + S_{K,\s} \v_K\in\R^d,
\end{equation*}
where $G_K\v_K$ is the consistent part of the gradient given by
$$G_K\v_K=\frac{1}{|K|}\sum_{\s'\in\E_K}|\s'|(v_{\s'}-v_K)n_{K,\s'}=\frac{1}{|K|}\sum_{{\s'}\in\E_K}|\s'|v_{\s'}n_{K,\s'},$$
and $S_{K,\s}\v_K$ is a stabilisation, given, for some parameter $\eta>0$, by
\begin{equation} \label{def:stabilisation}
	S_{K,\s}\v_K=\frac{\eta}{d_{K,\s}}\big(v_\s-v_K-G_K\v_K\cdot(\overline{x}_\s-x_K)\big)n_{K,\s}.
\end{equation}
Let us consider a generic tensor $\Lambda \in L^\infty(\Omega, \R^{d\times d})$, such that for any $\xi \in \R^d$, $\lambda_\flat |\xi|^2 \leq  \xi \cdot \Lambda \xi$ and $ |\Lambda \xi| \leq \lambda_\sharp |\xi|$.
Locally to any cell $K\in\M$, we introduce the discrete bilinear form $a^{\Lambda}_K:\V_K\times\V_K\to\R$ such that, for all $\u_K,\v_K\in\V_K$,
\begin{equation} \label{def:aKL}
  {a^{\Lambda}_K}(\u_K,\v_K) = \sum_{\s\in\E_K}|P_{K,\s}|\nabla_{K,\s}\v_K\cdot\Lambda_{K,\s}\nabla_{K,\s}\u_K= \int_K  \nabla_K \v_K \cdot \Lambda\nabla_K \u_K
\end{equation}
where we set {$\Lambda_{K,\s} = \frac{1}{|P_{K,\s}|}\int_{P_{K,\s}}\Lambda$}.
At the global level, we let $a^{\Lambda}_{\D}:\V_{\D}\times\V_{\D}\to\R$ be the discrete bilinear form such that, for all $\u_{\D},\v_{\D}\in\V_{\D}$,
\begin{equation*}
  a^{\Lambda}_\D (\u_{\D},\v_{\D}) = \sum_{K\in\M}a^{\Lambda}_K(\u_K,\v_K)= \int_\Omega \nabla_\D \v_\D \cdot \Lambda\nabla_\D \u_{\D}.
\end{equation*}

As for the continuous case, the well-posedness of HFV methods for diffusion problems relies on a coercivity argument.
Let $K\in\M$, and reason locally. By definition~\eqref{def:aKL} of the local discrete bilinear form $a_K^{\Lambda}$, and from the bounds on the diffusion coefficient, we have $\lambda_\flat \|\nabla_K \v_K\|_{L^2(K;\R^d)}^2 \leq a_K^{\Lambda}(\v_K,\v_K) \leq \lambda_\sharp \|\nabla_K \v_K\|_{L^2(K;\R^d)}^2$ for all $\v_K\in\V_K$.
Furthermore, the following comparison result holds (cf.~\cite[Lemma 13.11]{DEGGH:18}): there exist {$\alpha_\flat,\alpha_\sharp$} with $0<\alpha_\flat\leq\alpha_\sharp<\infty$, only depending on $\Omega$, $d$, and $\theta_\D$ such that 
$\alpha_\flat |\v_K|_{1,K}^2\leq \|\nabla_K \v_K\|_{L^2(K;\R^d)} ^2\leq  \alpha_\sharp |\v_K|_{1,K}^2$ for all $\v_K\in\V_K$.
Combining both estimates, we infer a local coercivity and boundedness result:
\begin{equation} \label{def:localcoercivity}
	\forall \v_K\in \V_K, \qquad\lambda_\flat\alpha_\flat |\v_K|_{1,K}^2
	\leq a^{\Lambda}_K(\v_K,\v_K)
	\leq  \lambda_\sharp\alpha_\sharp |\v_K|_{1,K}^2.
\end{equation}
Summing over $K \in \M$, we get the following global estimates:
\begin{equation} \label{def:globalcoercivity}
	\forall \v_{\D} \in \V_{\D}, \qquad  \lambda_\flat\alpha_\flat | \v_{\D} |_{1,\D}^2 \leq a^{\Lambda}_\D(\v_{\D},\v_{\D}) \leq \lambda_\sharp \alpha_\sharp | \v_{\D} |_{1,\D}^2.
\end{equation}
Last, we recall the hybrid discrete Poincar\'e inequality (cf.~\cite[Lemma B.32, $p=2$] {DEGGH:18}): there exists $C_{P,\Gamma^D}>0$, only depending on $\Omega$, $d$, $\Gamma^D$, and $\theta_\D$ such that 
\begin{equation} \label{eq:poinca}
	\forall \v_{\D} \in \V_{\D,0}, \qquad \|v_\M \|_{L^2(\Omega)} \leq {C_{P,\Gamma^D}} | \v_{\D} |_{1,\D}.
\end{equation}
In the sequel, we will omit the tensor in superscript and use the following notations instead:
\begin{equation}
	a_\D^\phi = a_\D^{\Lambda_\phi}, \ a_\D^P = a_\D^{\Lambda_P} \text{ and }a_\D^N = a_\D^{\Lambda_N}.
\end{equation} 
One can note that, since $\Lambda_\phi$ is symmetric, one has 
\begin{equation} \label{eq:contbilin}
	\forall(\u_\D, \v_\D) \in \V_{\D,0}^2, \
		 |a_\D^\phi(u_\D, \v_\D)| \leq \alpha_\sharp \lambda^\phi_\sharp |\u_\D|_{1,\D} |\v_\D|_{1,\D}.
\end{equation}

For further use, we introduce the linear form $L_\D^\phi : \V_{\D} \to \R$ such that 
\begin{equation}
	\forall \v_\D \in \V_\D, \, L_\D^\phi(\v_D) = \int_\Omega C  v_\M  - a_\D^\phi(\phid_\D^D, \v_\D).
\end{equation}
Note that, according to~\eqref{eq:contbilin}, the Poincar\'e inequality \eqref{eq:poinca} and the estimates~\eqref{eq:boundlifting}, one has
\begin{equation} \label{def:contL}
	\forall \v_\D \in \V_{\D,0}, \; | L_\D^\phi(\v_D) | \leq 
	\|C\|_{L^2(\Omega)}  \| v_\M \|_{L^2(\Omega)} + \alpha_\sharp \lambda_\sharp |\phid_\D^D|_{1,\D} |\v_\D|_{1,\D}
	\leq L_\sharp^\phi \ |\v_\D|_{1,\D},
\end{equation}
where there exists $C_L > 0$, only depending on $\Omega$, $\Gamma^D$, $d$, $\Lambda_\phi$ and $\theta_\D$ such that 
\[L_\sharp^\phi \leq C_L \left (\|C\|_{L^2(\Omega)} + \| \phi^D \|_{H^{\nicefrac12}(\Gamma^D)} \right).\]
\subsection{Description of the schemes and main theorem}
In this section, we introduce an HFV scheme for Problem \eqref{pb:DD:evol}.
The scheme is designed so as to preserve the entropic structure presented in the introduction.
In the sequel, we consider a fixed spatial discretisation $\D$ of $\Omega$ which satisfies the assumptions of Section \ref{sec:discre}, and a fixed time step $\Delta t >  0$ for the time dicretisation. For any $n \in \mathbb{N}$, we let $t^n = n \Delta t $.

Our HFV scheme relies on the hybrid nonlinear discretisation introduced in~\cite{CHHLM:22} for linear advection-diffusion equations, which was in turn inspired from the VAG and DDFV schemes of~\cite{CaGui:17,CCHKr:18}. We refer to~\cite{CHHLM:22} for a detailed description of the discretisation (with $h=\log$).
The common idea of these schemes is to discretise the flux in a nonlinear way: formally, given a function $u$ with values in $I_h$ and a function $\phi$, the problem $\divergence (- u \Lambda \nabla ( h(u) + \phi)) = 0 $ can be expressed in term of the flux
\[
	J = -u \Lambda \nabla(h(u) + \phi) = -u \Lambda \nabla w, 
\]
where $w = h(u) + \phi$ is a quasi-Fermi potential.
Therefore, locally to any cell $K\in\M$, we define an approximation of
\begin{equation*}
 (u,w,v) \mapsto -\int_K J\cdot\nabla v= \int_K u \,\Lambda\nabla w  \cdot  \nabla v
\end{equation*}
under the form
\[
	T_K^\Lambda (\u_K,\w_K,\v_K) = \int _K r_K(\u_K) \,\Lambda\nabla_K \w_K \cdot \nabla_K \v_K
\]
for all $I_h$-valued $\u_K\in\V_K$ and all $(\w_K,\v_K)\in\V_K^2$, where $r_K : \V_K \to I_h$, called local reconstruction operator, maps $I_h$-valued elements of $\V_K$ to $I_h$.
Since $r_K(\u_K)$ is a (positive) constant on $K$, we have
\begin{equation} \label{def:TK}
  T_K^\Lambda(\u_K,\w_K,\v_K)=r_K(\u_K) a^\Lambda_K\big(\w_K,\v_K\big),
\end{equation}
where $a^\Lambda_K$ is defined by~\eqref{def:aKL}. 
As already pointed out in previous works on similar nonlinear schemes \cite{Cances:18,CCHKr:18,CCHHK:20,CHHLM:22}, the definition of the local reconstruction operator is crucial to guarantee the existence of solutions as well as the long-time behaviour of the discrete solutions.
Especially, we use a reconstruction which embeds information from both the local cell and face unknowns. It is a natural generalisation of the one introduced in~\cite{CHHLM:22}.
For $\u_K$ an $I_h$-valued element of $\V_K$, we let
\begin{equation} \label{eq:rK}
  r_K (\u_K) =  \frac{1}{|\E_K|}\sum_{\s\in\E_K} m(u_K,u_\s)
\end{equation}
where $m : I_h^2 \to I_h $ is a continuous (mean-)function satisfying 
\begin{equation}\label{hyp:m}
 \forall (x,y) \in I_h^2, \, m_h(x,y) \leq m(x,y) \leq \max(x,y), 
\end{equation}
and the function $m_h$ is defined by 
\begin{equation}\label{def:mh}
	m_h(x,y) = \frac{[xh(x) - H(x)] - \left [ yh(y) - H(y)\right ]}{h(x) - h(y)} \text{ if } x\neq y
	\ \text{ and } m_h(x,x) = x.
\end{equation} 
This function $m_h$ was used in~\cite[Definition 3.3]{FiHer:17} to analyse an entropic scheme for nonlinear advection-diffusion equations.
One can note that in the case of the Boltzmann statistics ($h = \log$), $m_h(x,y)$ is the classical logarithmic average of $x$ and $y$, which coincides with the assumptions of~\cite{CHHLM:22}. 
Moreover, for any $(x,y) \in \R^2$, one has 
$m_h\left ( g(x),g(y) \right ) =  \frac{xg(x)- H(g(x)) - \left [ yg(y) - H(g(y))\right ]}{x -y }$.
Since the derivative of $x \mapsto xg(x)- H(g(x))$ is $g$, the following simple expression holds: 
\begin{equation} \label{eq:mhg}
	\forall (x,y) \in \R^2, \, m_h \left ( g(x) , g(y)  \right )
		= \frac{G(x) - G(y)}{x-y}.
\end{equation}
Note that, for all $(x,y) \in I_h^2$, one has 
$ m_h(x,y)  \leq \frac{x+y}{2} \leq \max(x,y)$,
and each expression of the previous sequence is a mean function $m$ satisfying \eqref{hyp:m}.
Heuristically, $r_K(\u_K)$ computes an average of the unknowns attached to the cell $K$, especially it contains information about all the local face unknowns.
The property~\eqref{hyp:m} will be instrumental to prove Lemma~\ref{lem:bounds} and the existence of solutions to the scheme.
Finally, we let $T_\D^\Lambda$ be such that, for all $I_h$-valued $\u_\D\in\V_\D$, and all $(\w_\D,\v_\D) \in \V_\D^2$,
\begin{equation} \label{sch2}
  T_\D^\Lambda(\u_\D,\w_\D,\v_\D) = \sum_{K\in\M}T_K^\Lambda(\u_K,\w_K,\v_K),
\end{equation}
where the local contributions $T_K^\Lambda$ are defined by~\eqref{def:TK}. In the sequel, we will use the following notation 
\[ 	
	 T_\D^{N}= T_\D^{\Lambda_N} \text{ and } 	T_\D^{P} = T_\D^{\Lambda_P} =  \text{ (\emph{resp.}, }
		 T_K^{N} = T_K^{\Lambda_N} \text{ and } 	  T_K^{P}=T_K^{\Lambda_P}
	\text{)}.
\]

Using a backward Euler discretisation in time, and the nonlinear HFV discretisation previously described in space, we consider the following scheme for the drift-diffusion system~\eqref{pb:DD:evol}: 
find $\left ( (\Nd_\D^n, \Pd_\D^n, \phid_\D^n) \right ) _{n \in \mathbb{N}} \in {\V_\D^3 }^\mathbb{N}$ such that 
\begin{subequations}\label{sch:DD}
        \begin{empheq}[left = \empheqlbrace]{align}
        	\w_\D^{N,n+1}&= h \left( \Nd_\D^{n+1}\right) - \phid_\D ^{n+1} - \alpha_N \one_\D \in \V_{\D,0}, \label{sch:DD:w-}\\
        	\w_\D^{P,n+1}&= h \left( \Pd_\D^{n+1}\right) + \phid_\D ^{n+1} - \alpha_P \one_\D \in \V_{\D,0}, \label{sch:DD:w+}\\
        \psid_\D ^{n+1}&= \phid_\D ^{n+1} - \phid_\D^D \in \V_{\D,0}, \label{sch:DD:psi}\\
	     \int_\Omega \frac{N_\M^{n+1} - N_\M^n}{\Delta t} v_\M +  T_\D^N (\Nd_\D^{n+1}, \w_\D^{N,n+1}, \v_\D)
	     	&= -\int_\Omega R(N_\M^{n+1},P_\M^{n+1}) v_\M  \quad\,\forall\v_{\D}\in\V_{\D,0},\label{sch:DD:N}  \\
	     \int_\Omega \frac{P_\M^{n+1} - P_\M^n}{\Delta t} v_\M +  T_\D^P (\Pd_\D^{n+1}, \w_\D^{P,n+1}, \v_\D)
	     	&= -\int_\Omega R(N_\M^{n+1},P_\M^{n+1}) v_\M  \quad\,\forall\v_{\D}\in\V_{\D,0},\label{sch:DD:P}  \\
	     a_\D^\phi(\psid_\D^{n+1}, \v_\D) - L_\D^\phi(\v_\D) &= \int_\Omega \left (P_\M^{n+1} - N_\M^{n+1} \right ) v_\M  
	     	\;\;\quad\:\forall\v_{\D}\in\V_{\D,0},\label{sch:DD:phi}  \\
            N^0_K = \frac{1}{|K|}\int_K N^{in } \text{ and } P^0_K &= \frac{1}{|K|}\int_K P^{in } \qquad\forall K\in\M. \label{sch:DD:ini} 
        \end{empheq}
\end{subequations}
Notice that the equations \eqref{sch:DD:w-} and \eqref{sch:DD:w+} imply that, for any $n \geq 1$, 
\begin{equation} \label{def:densities}
 \Nd_\D^{n} = g \left ( \w_\D^{N,n} + \phid_\D ^{n} + \alpha_N \one_\D\right ) 
 	\text{ and }
 \Pd_\D^{n} = g \left ( \w_\D^{P,n} - \phid_\D ^{n} + \alpha_P \one_\D\right ),
\end{equation}
therefore, since $g(\R) = I_h$, the discrete carrier densities $\Nd_\D^{n}$ and $\Pd_\D^{n}$ of a solution to \eqref{sch:DD} are $I_h$-valued.
\begin{rem}[Discrete potentials and boundary values]\label{rem:boundvalues}
The definitions of the discrete quasi-Fermi potentials~\eqref{sch:DD:w-} and~\eqref{sch:DD:w+} are used to enforce
the boundary conditions on the discrete unknowns. Indeed, since $\w_\D^{N,n} \in \V_{\D,0}$, from \eqref{def:densities} we get that for any $\s\in\E^D_{ext}$, $N^n_\s =   g \left ( 0 + \phi_\s ^D + \alpha_N \right )  = N^D_\s$ according to the compatibility condition~\eqref{eq:compconddiscret}. A similar result holds for $\Pd_\D^n$.
In fact, by definition of the trilinear forms, one can rewrite in a more natural way the equations~\eqref{sch:DD:N} and~\eqref{sch:DD:P} noticing that
\[ T_\D^N (\Nd_\D^{n+1}, \w_\D^{N,n+1} , \v_\D^N) = T_\D^N (\Nd_\D^{n+1}, h\left ( \Nd_\D^{n+1}\right) - \phid_\D ^{n+1} , \v_\D^N), \text{ and } \]
\[  T_\D^P (\Pd_\D^{n+1}, \w_\D^{P,n+1}, \v_\D^P) = T_\D^P (\Pd_\D^{n+1}, h \left( \Pd_\D^{n+1}\right) + \phid_\D ^{n+1}, \v_\D^P). \]
When considering general boundary conditions (i.e., if $\alpha_N$ and $\alpha_P$ are no longer constant on $\Gamma^D$, and therefore with liftings no longer constant in $\Omega$), the scheme will be the same, except for the equations~\eqref{sch:DD:N} and~\eqref{sch:DD:P} where arguments of the trilinear forms will be 
\[ 
	T_\D^N (\Nd_\D^{n+1}, \w_\D^{N,n+1} + \underline{\alpha_N}_\D, \v_\D^N) \text{ and }T_\D^P (\Pd_\D^{n+1}, \w_\D^{P,n+1}+				\underline{\alpha_P}_\D, \v_\D^P).
\]
The analysis of such a scheme is more sophisticated than the one presented here, and shall be the subject of a future work.
\end{rem}
Plugging the expression of the discrete carrier densities~\eqref{def:densities} into Equation~\eqref{sch:DD:phi}, we get an elliptic semi-linear equation on $\psid_\D^{n+1}$. This equation will be of great importance in the analysis. At the continuous level, it corresponds to the following semi-linear Poisson equation: 
\begin{equation} \label{pb:NLPoisson}
	\left\{
	\begin{aligned}
		- \divergence \left ( \Lambda_\phi \nabla \phi \right ) &= C + g(z^P - \phi) - g(z^N + \phi)  &&\text{ in } \Omega, \\
		  \phi&= \phi^D &&\text{ on } \Gamma^D, \\
		 \Lambda_\phi \nabla \phi  \cdot n &= 0 &&\text{ on } \Gamma^N, 
	\end{aligned}
	\right.
\end{equation} 
where $z^P$ and $z^N$ are given functions in $L^\infty(\Omega)$.
\begin{rem}[Electrostatic potential at thermal equilibrium]\label{rem:poteq}
If $z^P$ and $z^N$ are constant with respective values $\alpha_P$ and $\alpha_N$ in $\Omega$, then the solution to \eqref{pb:NLPoisson} is the electrostatic potential at thermal equilibrium $\phi^e$ solution to \eqref{pb:PoissonEqT}.
\end{rem} 
We introduce an HFV scheme to approximate \eqref{pb:NLPoisson}: 
find $\phid_\D = \psid_\D + \phid_\D^D \in \V_\D$ where $\psid_\D \in \V_{\D,0}$ is such that
\begin{equation}\label{sch:NLPoisson}
	\forall \v_\D \in \V_{\D,0}, \ a_\D^\phi(\psid_\D, \v_\D) = L_\D^\phi(\v_\D) 
		+ \int_\Omega \left [ g(z^P_\M - \phi_\M^D - \psi_\M )  - g(z^N_\M + \phi_\M^D + \psi_\M ) \right ] v_\M,
\end{equation}
where $z^P_\M$ and $z^N_\M$ are the piecewise constant functions such that, for any cell $K \in \M$, $z^P_\M = \frac{1}{|K|}\int_K z^P$ and $z^N_\M = \frac{1}{|K|}\int_K z^N$on $K$.
We will prove in Theorem~\ref{th:discreteNLPoisson} that this scheme is well-posed. Especially, in the light of Remark~\ref{rem:poteq}, one can define the discrete thermal equilibrium $(\Nd_\D^e, \Pd_\D^e, \phid_\D^e) \in \V_\D ^3$ as the following triplet of unknowns: 
\begin{equation} \label{def:discreteequ}
	\left\{
	\begin{aligned}
			\phid_\D^e &= \phid_\D^D + \psid_\D^e, \text{ where } \psid_\D^e \in \V_{\D,0} \text{ is the solution to~\eqref{sch:NLPoisson} with} (z^N,z^P) = (\alpha_N,\alpha_P),  \\
			\Nd_\D^e &= g(\alpha_N \one_\D + \phid_\D^e),  \\
			\Pd_\D^e &= g(\alpha_P \one_\D - \phid_\D^e).
	\end{aligned}
	\right.
\end{equation}
Note that, by definition, the discrete densities at thermal equilibrium $\Nd_\D^e$ and $\Pd_\D^e$ are $I_h$-valued. 
\newline
We remind the reader that alternative formulations (in the spirit of finite volume methods) of the schemes~\eqref{sch:DD} and~\eqref{sch:NLPoisson} are presented in Section~\ref{sec:flux}. 

We are now in position to state the main result of this paper, whose proof is the subject of Section~\ref{sec:DD}.
\begin{theorem}[Existence of discrete $I_h$-valued solutions to~\eqref{sch:DD}] \label{th:DD}
Let $N^{in}$ and $P^{in}$ two positive functions in $L^\infty(\Omega)$ satisfying the condition~\eqref{eq:bounddata}.
There exists at least one solution $\left ( (\Nd_\D^n, \Pd_\D^n, \phid_\D^n) \right ) _{n \in \mathbb{N}} \in {\V_\D^3 }^\mathbb{N}$ to the coupled scheme~\eqref{sch:DD}. It satisfies the following entropy-dissipation relation:
\begin{equation} \label{eq:diss}
	\forall n \in \mathbb{N}, \quad \frac{\N^{n+1} - \N ^n }{\Delta t} + \Diss^{n+1} \leq 0,
\end{equation}
where $\N ^n$ and $\Diss^{n}$ are, respectively, the discrete relative entropy and dissipation at time $t^n = n \Delta t$, defined by 
\begin{multline} \label{def:entroth}
	\N^n = \int_\Omega H(N_\M^n) - H(N^e_\M) - h(N_\M^e) \left (N_\M^n - N_\M^e \right )    \\
		+ \int_\Omega H(P_\M^n) - H(P^e_\M) - h(P_\M^e) \left (P_\M^n - P_\M^e \right )   
	 	+ \frac{1}{2}a^\phi_\D(\phid_\D^n - \phid_\D^e,\phid_\D^n - \phid_\D^e), 
\end{multline} 
\begin{equation} \label{def:dissth}
	\text{ and }\Diss^n  =  T_\D^N (\Nd_\D^n, \w_\D^{N,n}, \w_\D^{N,n}) + T_\D^P (\Pd_\D^n, \w_\D^{P,n}, \w_\D^{P,n})
		+ \int_\Omega  R(N_\M^n,P_\M^n)\left ( h(N_\M^n) + h(P_\M^n) \right ) .
\end{equation}
Moreover, there exist $ 0 < M_\flat  \leq M_\sharp < \dmax$, depending on the physical data, $\Delta t$, and $\D$ such that
\begin{equation} \label{eq:unifpos}
	\forall n \geq 1, \qquad M_\flat  \one_\D \leq \Nd^n_\D \leq M_\sharp \one_\D  
		\text{ and }  M_\flat  \one_\D \leq \Pd^n_\D \leq M_\sharp \one_\D.
\end{equation}
\end{theorem}
\begin{rem}[$L^\infty$ bounds]\label{rem:boundnonuniform}
Notice that the bounds $M_\flat$ and $M_\sharp$ on the discrete densities depend on the time step $\Delta t$ and on the spatial discretisation $\D$ (and therefore, on the meshsize). Thus, our result is weaker than the one proved in~\cite[Theorem 1]{BCCH:19} in the framework of TPFA schemes. In fact, we believe that the discrete Moser iteration process used in~\cite{BCCH:19} to get uniform bounds cannot be adapted in the hybrid framework, because the method is intrinsically non-monotone. 
\end{rem}
\section{Analysis of the stationary scheme} \label{sec:sta}
In this section, we analyse the scheme~\eqref{sch:NLPoisson} for the approximation of the semi-linear elliptic equation~\eqref{pb:NLPoisson}. First, we show that the scheme is well-posed. Then, we use this result to show that there is a correspondence between the discrete densities and the discrete quasi-Fermi potentials, which will be very useful for the following. 
\subsection{Well-posedness } \label{sec:sta:wellposed}
In the next theorem, we show that the scheme~\eqref{sch:NLPoisson} admits a unique solution. The proof relies on an energy minimisation approach.
\begin{theorem}[Well-posedness for \eqref{sch:NLPoisson}] \label{th:discreteNLPoisson}
Let $z^N$ and $z^P$ be two functions in $L^\infty(\Omega )$.
There exists a unique $\psid_\D \in \V_{\D,0}$ solution to \eqref{sch:NLPoisson}, and there exists a positive constant $C_\text{Poisson}$, only depending on $\Lambda_\phi$, $\Omega$, $\Gamma^D$, $d$ and $\theta_\D$ such that the following stability estimate holds:
\begin{equation} \label{eq:bound:NLPoisson}
	|\psid_\D|_{1,\D}^2 \leq C_\text{Poisson} \left [
		G \left (\|z^P\|_{L^\infty(\Omega)} + \|z^N\|_{L^\infty(\Omega)}   + \|\phi^D\|_{L^\infty(\Omega)} \right )
		+\|C\|_{L^2(\Omega)}^2 + \| \phi^D \|_{H^{\nicefrac12}(\Gamma^D)}^2		
		 \right ].
\end{equation}
Moreover, the application $(z^N, z^P) \mapsto \psid_\D$ is continuous from $L^\infty(\Omega) \times L^\infty(\Omega)$ to $\V_{\D,0}$.
\end{theorem}
\begin{proof}
 We define the discrete energy functional $J_\D : \V_{\D,0} \to \R$ by 
\begin{equation}
	J_\D(\u_\D) = \frac{1}{2} a_\D^\phi(\u_\D,\u_\D) - L_\D^\phi(\u_\D) + \int_\Omega G(z^P_\M - \phi_\M^D - u_\M ) + G(z^N_\M + \phi_\M^D  +u_\M ).
\end{equation}
Since $G$ is $\C^2$ on $\R$, $J_\D$ is $\C^2$ on $\V_{\D,0}$, and one can compute its differential at a point $\u_\D\in \V_{\D,0}$: 
\[
	\forall \v_\D\in\V_{\D,0}, \ d {J_\D} (\u_\D) \cdot \v_\D  = 
		a_\D^\phi(\u_\D,\v_D)- L_\D^\phi(\v_\D) - \int_\Omega  \left [g(z^P_\M - \phi_\M^D - u_\M ) - g(z^N_\M + \phi_\M^D  +u_\M ) \right ]v_\M.
\]
Hence, a solution $\psid_\D \in \V_{\D,0}$ to \eqref{sch:NLPoisson} is an element of $\V_{\D,0}$ such that $d {J_\D} (\psid_\D) = 0$.
Now, one can notice that $J_\D$ is strongly convex on $\V_{\D,0}$. Indeed, since $a^\phi_\D$ is coercive and $L^\phi_\D$ is a continuous linear form, $\u_\D \mapsto \frac{1}{2} a_\D^\phi(\u_\D,\u_\D) - L_\D^\phi(\u_\D)$ is strongly convex. Moreover, $G$ is a convex function (since $G'=g$ is strictly increasing), so $\u_\D \mapsto \int_\Omega G(z^P_\M - \phi_\M^D - u_\M ) + G(z^N_\M + \phi_\M^D  +u_\M )$ is also convex.
Therefore, $J_\D$ has a unique minimum $\psid_\D \in \V_{\D,0}$, which is a solution to \eqref{sch:NLPoisson}. 
On the other hand, by convexity of $J_\D$, any critical point of $J_\D$ in $\V_{\D,0}$ is a minimum, therefore the solution to~\eqref{sch:NLPoisson} is unique. The continuity of $(z^N_\M, z^P_\M) \mapsto \psid_\D$ follows from the regularity of $J_\D$ and the implicit function theorem.
To obtain \eqref{eq:bound:NLPoisson}, one can use the monotonicity of $G$ and \eqref{eq:boundlifting} to get
\begin{align*}
	J_\D (\psid_\D ) \leq J_\D( \underline{0}_\D ) &= \int_\Omega G(z^P_\M - \phi_\M^D) + G(z^N_\M + \phi_\M^D) \\
		&\leq  \int_\Omega G \left (\|z^P\|_{L^\infty(\Omega)} + \|\phi^D\|_{L^\infty(\Omega)} \right ) 
		+  G \left (\|z^N\|_{L^\infty(\Omega)} + \|\phi^D\|_{L^\infty(\Omega)} \right) \\
		& \leq 2  |\Omega |  G \left ( \|z^N\|_{L^\infty(\Omega)}+ \|z^P\|_{L^\infty(\Omega)} +\|\phi^D\|_{L^\infty(\Omega)}\right ). 
\end{align*}
Furthermore, using the fact that $G$ is positive alongside with \eqref{def:globalcoercivity} and \eqref{def:contL},  and using Young's inequality on $ L_\sharp^\phi |\psid_\D|_{1,\D}$, one has the following lower bound:
\[
	J_\D(\psid_\D) \geq  \frac{1}{2} a_\D^\phi(\psid_\D,\psid_\D ) - L_\D^\phi(\psid_\D) 
		\geq	 \frac{\alpha_\flat \lambda_\flat}{2} |\psid_\D|_{1,\D}^2 - L_\sharp^\phi |\psid_\D|_{1,\D} 
		\geq  \frac{\alpha_\flat \lambda_\flat}{4} |\psid_\D|_{1,\D}^2  - \frac{1}{\alpha_\flat \lambda_\flat}{L_\sharp^\phi}^2.
\] 
We conclude by combining the previous estimates and using the bound on $L_\sharp^\phi$.
\end{proof}
Notice that this result ensures the existence and uniqueness of the discrete thermal equilibrium $(\Nd_\D^e, \Pd_\D^e, \phid_\D^e)$ described in equation~\eqref{def:discreteequ}.
\begin{rem}[Uniform bounds] The stability result~\eqref{eq:bound:NLPoisson} gives a uniform bound (with respect to the meshsize $h_\D$) on $\psid_\D$ in energy norm $| \cdot |_{1,\D}$. There is however no clear $L^\infty$ bound (uniform in $\D$) on the potential. It is rather different from some previous results in the TPFA framework (see~\cite{CHFi:07}), where $L^\infty$ bounds were obtained alongside with the $H^1$ estimate.
\end{rem}
\subsection{Correspondence between discrete densities and discrete potentials} \label{sec:assoc}
In this section, we discuss the link between discrete densities and discrete potentials.

For a given couple $(\Nd_\D,\Pd_\D) \in \V_\D^2$ of $I_h$-valued vectors, one associates a unique discrete electrostatic potential $\phid_\D = \psid_\D +  \phid_\D^D \in \V_\D$ such that $\psid_\D \in \V_{\D,0}$ is the solution to
\begin{equation} \label{def:assocpot}
	\forall\v_{\D} \in\V_{\D,0}, \, a_\D^\phi(\psid_\D, \v_\D) = L_\D^\phi(\v_\D) + \int_\Omega \left (P_\M - N_\M \right ) v_\M.
\end{equation}
Since $a_\D^\phi$ is coercive on $\V_{\D,0}$ and $L_\D^\phi$ is a linear form, \eqref{def:assocpot} is well-posed and $\psid_\D$ is uniquely defined, so is $\phid_\D$. Then, one associates to $(\Nd_\D,\Pd_\D)$ a unique couple of discrete quasi-Fermi potentials $(\w_\D^N, \w_\D^P)\in \V_\D^2 $ defined by 
\begin{equation} \label{def:w}
	\w_\D^N= h \left( \Nd_\D\right) - \phid_\D- \alpha_N \one_\D
	\text{ and }
	\w_\D^P= h \left( \Pd_\D\right) + \phid_\D  - \alpha_P \one_\D.
\end{equation}
We call these vectors the discrete electrostatic ($\phid_\D$) and  quasi-Fermi potentials ($(\w_\D^N, \w_\D^P)$) associated to the discrete densities $(\Nd_\D,\Pd_\D)$. 
\newline
Note that, as in Remark~\ref{rem:boundvalues}, the discrete quasi-Fermi potentials associated to $(\Nd_\D,\Pd_\D)$ are in $\V_{\D,0}$ if and only if the densities satisfy the following discrete boundary condition: 
\begin{equation} \label{cond:boundarydensities}
	\forall \s \in \E_{ext}^D, \ N_\s = N_\s^D \text{ and } P_\s = P_\s^D.
\end{equation}

Conversely, for a given couple $(\w_\D^N, \w_\D^P) \in \V_{\D}^2$, one associates a unique electrostatic potential $\phid_\D = \psid_\D + \phid_\D^D \in \V_\D$ such that $\psid_\D \in \V_{\D,0}$ is the solution to
\begin{multline} \label{def:assocpot:w}
	\forall\v_{\D} \in\V_{\D,0}, \, a_\D^\phi(\psid_\D, \v_\D) = L_\D^\phi(\v_\D) +  \\
	 \int_\Omega \left (g (w^P_\M + \alpha_P  - \phi^D_\M - \psi_\M ) - g (w^N_\M + \alpha_N  + \phi^D_\M + \psi_\M ) \right ) v_\M.
\end{multline}
Indeed, since the functions $ w^P_\M + \alpha_P$ and $w^N_\M + \alpha_N$ are in $L^\infty(\Omega)$, according to Theorem~\ref{th:discreteNLPoisson} applied to $(z^N,z^P )= (w^P_\M + \alpha_P,w^N_\M + \alpha_N)$, there exists a unique solution to problem~\eqref{def:assocpot:w}. Then, one associates to $(\w_\D^N, \w_\D^P)$ a couple of $I_h$-valued discrete carrier densities $(\Nd_\D,\Pd_\D) \in \V_{\D}$ by
\[
	\Nd_D = g( \w_\D^N+ \phid_\D + \alpha_N \one_\D) \text{ and } \Pd_D = g( \w_\D^P - \phid_\D + \alpha_P \one_\D).
\]
Notice that with this process, according to~\eqref{def:assocpot:w} and the definition of $\Nd_\D$ and $\Pd_\D$, one has that
\[
	\forall\v_{\D} \in\V_{\D,0}, \, a_\D^\phi(\psid_\D, \v_\D ) = L_\D^\phi(\v_\D) + \\
	 \int_\Omega \left (P_\M -N_\M  \right ) v_\M .
\]
Therefore, it means that $(\w_\D^N, \w_\D^P)$ and $\phid_\D$ are the discrete potentials associated to the discrete densities $(\Nd_\D,\Pd_\D)$ in the sense of~\eqref{def:assocpot}-\eqref{def:w}.
\newline
Moreover, $(\w_\D^N, \w_\D^P) \in \V_{\D,0}^2$ if and only if $(\Nd_\D,\Pd_\D)$ satisfy the discrete boundary condition~\eqref{cond:boundarydensities}.  

Thus, there is a bijective correspondence between ($I_h$-valued) discrete densities and discrete quasi-Fermi potentials (in $\V_{\D})$.
\begin{rem}[Thermal equilibrium and quasi-Fermi potentials] \label{rem:weq}
From the definition~\eqref{def:discreteequ} of the discrete thermal equilibrium, one can notice that $(\Nd_\D^e, \Pd_\D^e, \phid_\D^e)$ are the discrete densities and electrostatic potential associated to the quasi-Fermi potentials at thermal equilibrium, which are the null vectors $\left (\w_\D^{N,e}, \w_\D^{P,e} \right )  = (\zero_\D, \zero_\D)$.
\end{rem}
\section{Analysis of the transient scheme} \label{sec:DD}
In this section, we perform the analysis of the scheme~\eqref{sch:DD} for the drift-diffusion system. 
The first two subsections are dedicated to the proof of Theorem~\ref{th:DD}.
We also state and prove a qualitative property about the discrete solutions, namely the long-time behaviour in Theorem~\ref{th:longtime}, in line with the work of~\cite{BCCH:17}.
\subsection{Discrete entropy structure} \label{sec:structure}
The scheme~\eqref{sch:DD} is designed so as to mimic the continuous entropic structure presented in introduction.
In this part, we are interested in the entropic property of the scheme~\eqref{sch:DD}.

First, we define a discrete counterpart of the entropy and entropy dissipation.
Given $(\Nd_\D,\Pd_\D)$ a couple of $I_h$-valued vectors, one can define the relative discrete entropy as  
\begin{multline} \label{def:entro}
	\N(\Nd_\D,\Pd_\D) = \int_\Omega H(N_\M) - H(N^e_\M) - h(N_\M^e) \left (N_\M - N_\M^e \right )    \\
		+ \int_\Omega H(P_\M) - H(P^e_\M) - h(P_\M^e) \left (P_\M - P_\M^e \right )   
	 	+ \frac{1}{2}a^\phi_\D(\phid_\D - \phid_\D^e,\phid_\D - \phid_\D^e),
\end{multline} 
where $(\Nd^e_\D, \Pd^e_\D, \phid_\D^e)$ are defined in~\eqref{def:discreteequ}, and $\phid_\D$ is the discrete electrostatic potential associated to $(\Nd_\D,\Pd_\D)$ in the sense of Section~\ref{sec:assoc}.
Notice that, by convexity of $H$, $H(N_\M) - H(N^e_\M) - h(N_\M^e) \left (N_\M - N_\M^e \right )$ and the analogous term in $P_\M$ are non-negative. Therefore, the discrete entropy is non-negative.
Similarly, one defines the discrete dissipation by 
\begin{multline} \label{def:diss}
	\Diss(\Nd_\D,\Pd_\D) =  T_\D^N (\Nd_\D, \w_\D^N, \w_\D^N) + T_\D^P (\Pd_\D, \w_\D^P, \w_\D^P)
		+ \int_\Omega  R(N_\M,P_\M)\left ( h(N_\M) + h(P_\M) \right ) .
\end{multline}
By definition of $T_\D^N$ and $T_\D^P$, the first two terms are non-negative. One can also notice that the integrand of the third term writes 
\[
	R(N_\M,P_\M)\left ( h(N_\M) + h(P_\M) \right ) = r(x,N_\M, P_\M) 
		\left  ( \exp  ( h(N_\M) + h(P_\M)  ) - 1 \right )  \left ( h(N_\M) + h(P_\M) \right ), 
\]
so since $r$ is non-negative and $x(\e^x-1)\geq 0$ for any real $x$, the discrete dissipation is also non-negative.
Note that, given a couple $(\w_\D^N, \w_\D^P) \in \V_{\D}^2$ of discrete quasi-Fermi potentials, one can also define the corresponding relative entropy and dissipation by $\N(\w_\D^N, \w_\D^P)= \N(\Nd_\D,\Pd_\D)$ and $\Diss(\w_\D^N, \w_\D^P)= \Diss(\Nd_\D,\Pd_\D)$, where $(\Nd_\D,\Pd_\D)$ are the discrete densities associated to $(\w_\D^N, \w_\D^P)$.

At the continuous level, the analysis of the drift-diffusion system relies on the entropic structure. At the discrete level, we  use analogous results, based on the following lemma, which will be used in the proofs of Theorems~\ref{th:DD} and~\ref{th:longtime}. 
\begin{lemma}[Discrete differentiation of the entropy] \label{lem:test}
Let $\Nd_\D^n$, $\Pd_\D^n$, $\Nd_\D$ and $\Pd_\D$ be four $I_h$-valued elements of $\V_\D$, and $(\w_\D^N, \w_\D^P)\in \V_\D^2 $ be the discrete quasi-Fermi potentials associated to the discrete densities $(\Nd_\D , \Pd_\D)$.
Then, the following inequality holds: 
\begin{equation} \label{eq:test}
	\N(\Nd_\D,\Pd_\D) -\N(\Nd_\D^n,\Pd_\D^n) \leq 
		\int_\Omega \left (N_\M - N_\M ^n \right ) w_\M^N  +
		\int_\Omega \left ( P_\M - P_\M ^n \right )  w_\M^P .
\end{equation}
\end{lemma}
\begin{proof}
By definition of the discrete entropy, one has 
\begin{multline} \label{eq:difentro}
	\N(\Nd_\D,\Pd_\D) -\N(\Nd_\D^n,\Pd_\D^n) = 
		\int_\Omega H(N_\M) - H(N^n_\M) - h(N_\M^e) \left (N_\M - N_\M^n \right ) \\
		+\int_\Omega H(P_\M) - H(P^n_\M) - h(P_\M^e) \left (P_\M - P_\M^n \right ) \\
		+\frac{1}{2} \left ( a^\phi_\D(\phid_\D - \phid_\D^e,\phid_\D - \phid_\D^e) -
			a^\phi_\D(\phid_\D^n - \phid_\D^e,\phid_\D^n - \phid_\D^e) \right ).
\end{multline}
Let us first consider the first term (in $N$) of \eqref{eq:difentro}: by convexity of $H$ and definition of $\Nd_\D^e$, we get 
\begin{multline*}
	\int_\Omega H(N_\M) - H(N^n_\M) - h(N_\M^e) \left (N_\M - N_\M^n \right ) \leq \int_\Omega (h(N_\M) - h(N_\M^e) )\left (N_\M - N_\M^n \right ) \\
	= \int_\Omega (h(N_\M) - \phi^e_\M - \alpha_N )\left (N_\M - N_\M^n \right ).
\end{multline*}
For the second term of \eqref{eq:difentro}, one has an analogous result, namely 	
\[
	\int_\Omega H(P_\M) - H(P^n_\M) - h(P_\M^e) \left (P_\M - P_\M^n \right ) 
		\leq \int_\Omega (h(P_\M) + \phi^e_\M - \alpha_P )\left (P_\M - P_\M^n \right ).
\]
Now, notice that since $a^\phi_\D$ is a symmetric and positive semi-definite bilinear form on $\V_{\D}$, the quadratic form $\u_\D \mapsto a^\phi_\D(\u_\D,\u_\D)$ is convex, and therefore for any $(\u_\D, \v_\D) \in \V_\D^2$, one has
\[
	a^\phi_\D(\u_\D, \u_\D) - a^\phi_\D(\v_\D, \v_\D ) \leq  2 a^\phi_\D (\u_\D - \v_\D, \u_\D ).
\]
Using this relation with $\u_\D = \phid_\D - \phid_\D^e$ and $\v_\D = \phid_\D^n - \phid_\D^e \in \V_{\D}$, we obtain the following estimate on the third term of \eqref{eq:difentro}:
\[
	\frac{1}{2} \left ( a^\phi_\D(\phid_\D - \phid_\D^e,\phid_\D - \phid_\D^e) -
			a^\phi_\D(\phid_\D^n - \phid_\D^e,\phid_\D^n - \phid_\D^e) \right )
			\leq a^\phi_\D(\phid_\D -\phid_\D^n, \phid_\D -\phid_\D^e).
\]
We recall that the electrostatic potentials are defined by \eqref{def:assocpot}, and that $\phid_\D -\phid_\D^n = \psid_\D - \psid_\D^n$. Hence, letting $\v_\D = \phid_\D -\phid_\D^e \in \V_{\D,0}$ and using the linearity of $a^\phi_\D$ with respect to the first variable, one gets 
\begin{align*}
	a^\phi_\D(\phid_\D -\phid_\D^n, \v_\D) &=  a^\phi_\D(\psid_\D, \v_\D)  - a^\phi_\D(\psid_\D^n, \v_\D)  \\
		&= L^\phi_\D(\v_\D) + \int_\Omega (P_\M - N_\M)v_\M - L^\phi_\D(\v_\D) - \int_\Omega (P^n_\M - N^n_\M)v_\M  \\
		&= \int_\Omega (P_\M- P^n_\M ) v_\M -  \int_\Omega (N_\M-N^n_\M)  v_\M.
\end{align*}
Combining the previous estimates, we obtain 
\begin{multline*}
	\N(\Nd_\D,\Pd_\D) -\N(\Nd_\D^n,\Pd_\D^n) \leq 
		\int_\Omega (h(N_\M) - \phi^e_\M - \alpha_N  - v_\M)\left (N_\M - N_\M^n \right )  \\
			+ \int_\Omega (h(P_\M) + \phi^e_\M - \alpha_P + v_\M)\left (P_\M - P_\M^n \right ).
\end{multline*}
But, by definition of $\v_\D$, we have  
$\displaystyle 
	h(N_\M) - \phi^e_\M - \alpha_N  - v_\M = h(N_\M) - \phi_\M - \alpha_N  = w ^N_\M
$ and $h(P_\M) + \phi^e_\M - \alpha_P + v_\M = w ^P_\M$, therefore \eqref{eq:test} holds.
\end{proof}
One can now state the following a priori result about the dissipation of entropy, which ensures that~\eqref{eq:diss} holds.
\begin{prop}[Entropy dissipation] \label{prop:entrodiss}
Assume that $\left ( (\Nd_\D^n, \Pd_\D^n, \phid_\D^n) \right ) _{n \in \mathbb{N}} \in {\V_\D^3 }^\mathbb{N}$ is a solution to 
\eqref{sch:DD}. Then, the following entropy-entropy dissipation relation holds: 
\[
	\forall n \geq 0, \quad \frac{\N(\Nd_\D^{n+1},\Pd_\D^{n+1}) -\N(\Nd_\D^n,\Pd_\D^n)}{\Delta t} 
		\leq -\Diss(\Nd_\D^{n+1},\Pd_\D^{n+1}).
\]
\end{prop}
\begin{proof}
Since the discrete quasi-Fermi potentials $\w_\D^{N,n+1}$ and $\w_\D^{P,n+1}$ are in $\V_{\D,0}$ by~\eqref{sch:DD:w-} and~\eqref{sch:DD:w+}, one can use them as test functions in \eqref{sch:DD:N} and \eqref{sch:DD:P}: summing the two identities, and noticing that $w_\M^{N,n+1} +w_\M^{P,n+1} = h(N_\M^{n+1}) + h(P_\M^{n+1})$ because of \eqref{eq:alpha}, one gets that 
\[
	\int_\Omega \frac{N_\M^{n+1} - N_\M^n}{\Delta t} w_\M^{N,n+1} + 
		\int_\Omega \frac{P_\M^{n+1} - P_\M^n}{\Delta t} w_\M^{P,n+1} 
		= -\Diss(\Nd_\D^{n+1},\Pd_\D^{n+1}).
\]
One concludes using Lemma \ref{lem:test}.
\end{proof}
From this entropic structure follows an important preservation property of the scheme~\eqref{sch:DD}.
\begin{prop}[Thermodynamic consistency] \label{prop:thermoconsis} 
Let $(\Nd_\D^\infty, \Pd_\D^\infty, \phid_\D^\infty) \in \V_\D^3$ be a discrete stationary state of the scheme~\eqref{sch:DD}, in the sense that there exist 
$\left ( (\Nd_\D^n, \Pd_\D^n, \phid_\D^n) \right ) _{n \in \mathbb{N}} \in {\V_\D^3 }^\mathbb{N}$ a solution to 
\eqref{sch:DD} and $n_\infty \in \mathbb{N}$ such that, 
 \[ 
   \forall n \geq n_\infty, \, (\Nd_\D^n, \Pd_\D^n, \phid_\D^n) = (\Nd_\D^\infty, \Pd_\D^\infty, \phid_\D^\infty).
\]
Then the discrete stationary state coincides with the discrete thermal equilibrium: 
 \begin{equation} \label{eq:preserveq}
 	(\Nd_\D^\infty, \Pd_\D^\infty, \phid_\D^\infty) = (\Nd_\D^e, \Pd_\D^e, \phid_\D^e) .
 \end{equation}
\end{prop}
\begin{proof}
Let $n \geq n_\infty$, we have 
\[
	\N(\Nd_\D^{n+1},\Pd_\D^{n+1}) -\N(\Nd_\D^n,\Pd_\D^n) 
		= \N(\Nd_\D^\infty,\Pd_\D^\infty) -\N(\Nd_\D^\infty,\Pd_\D^\infty) 
		= 0.
\]
According to Proposition~\ref{prop:entrodiss}, one has by positivity of the dissipation 
\[
	\Diss(\Nd_\D^\infty,\Pd_\D^\infty) = \Diss(\Nd_\D^{n+1},\Pd_\D^{n+1}) =0. 
\]
But, since all the terms of the dissipation are non-negative (see \eqref{def:diss}), it means that 
\[
	T_\D^N (\Nd_\D^\infty, \w_\D^{N,\infty}, \w_\D^{N, \infty}) 
		= T_\D^P (\Pd_\D^\infty, \w_\D^{P,\infty}, \w_\D^{P,\infty}) = 0, 
\]
where $(\w_\D^{N,\infty},\w_\D^{P,\infty}) \in \V_{\D,0}^2$ are the quasi-Fermi potentials associated to the discrete stationary densities $(\Nd_\D^\infty, \Pd_\D^\infty)$. Now, let $K \in \M$. Since all the coordinates of $\Nd_K^\infty$ are positive, $r_K(\Nd_K^\infty)$ is also positive, and by definition~\eqref{def:TK} of $T_K^N$, we get that 
\[
	a_K^N(\w_K^{N,\infty}, \w_K^{N,\infty}) = 0.
\]
Using the local coercivity estimates~\eqref{def:localcoercivity}, we infer that $|\w_K^{N,\infty}|_{1,K}^2 = 0$. This holds for any cell $K \in \M$, so we have $|\w_\D^{N,\infty}|_{1,\D} = 0$, but $\w_\D^{N,\infty} \in \V_{\D,0}$, therefore $\w_\D^{N,\infty} = \zero_\D$. A similar result holds for $\w_\D^{P,\infty}$. Thus, one has 
$	\w_\D^{N,\infty} = \w_\D^{P,\infty} = \zero_\D$
and, by Remark~\ref{rem:poteq}, it means that $(\w_\D^{N,\infty},\w_\D^{P,\infty})$ are equal to the quasi-Fermi potentials at thermal equilibrium. By correspondence between densities and potential discussed in Section~\ref{sec:assoc}, we get~\eqref{eq:preserveq}.
\end{proof}

\begin{rem}[Preservation of the thermal equilibrium]
The statement of Proposition~\ref{prop:thermoconsis} means that the scheme~\eqref{sch:DD} preserves the thermal equilibrium: the only admissible discrete stationary state is the discrete thermal equilibrium.
This remarkable feature is sometimes called thermodynamic consistency in the context of TPFA schemes (see \cite{FPF:18}). In such a framework, this result is expressed in the following way: if the numerical fluxes of the convection-diffusion equations vanish, then the discrete quasi-Fermi potentials are constant (equal to zero). For our HFV scheme, this statement still holds, using the numerical fluxes~\eqref{TKL-flux} defined in Section~\ref{sec:flux}. In fact, both statements in terms of fluxes and stationary states are equivalent. This notion of thermodynamic consistency can indeed be reformulated using the entropic property of the schemes: if the discrete dissipation associated to a discrete solution vanishes, then the solution is the discrete thermal equilibrium.
\end{rem}

\subsection{Existence of solutions}
The goal of this section is to prove the existence result and the estimates~\eqref{eq:unifpos} of Theorem~\ref{th:DD}. 
The proof proposed here is an adaptation of the one introduced in \cite{CHHLM:22} for similar schemes in the context of linear advection-diffusion equations. According to the discussion in Section~\ref{sec:assoc}, it is equivalent to seek discrete solutions in terms of  densities or in terms of quasi-Fermi potentials. Therefore, we will seek discrete quasi-Fermi potentials in the whole space $\V_{\D,0} \times \V_{\D,0}$ instead of discrete densities in the set of $I_h$-valued elements of $\V_{\D} \times \V_{\D}$ (which is not a vector space). 

In the sequel, we consider the vector space $\V_{\D,0}^2$, and denote by $\ww_\D =(\w_\D^N, \w_\D^P) \in \V_{\D,0} ^2$ a typical element. We endow $\V_{\D,0}^2$ with the following natural inner product and norm: 
\[
	\langle \ww_\D, \vv_\D \rangle  = 
	\sum_{K \in \M} \left (w^N_K  v^N_K  + w^P_K v^P_K \right ) 
		+ \sum_{\s \in \E_{int} \cup \E_{ext}^N} \left (w^N_\s  v^N_\s  + w^P_\s v^P_\s \right ) 
		\text{ and } \|\ww_\D\| = \sqrt{\langle \ww_\D,\ww_\D \rangle}.
\]
One can also endow $\V_{\D,0}^2$ with the $l^\infty$ norm $\| . \|_\infty$ defined by 
\[
	\|\ww_\D\|_\infty = \displaystyle\max_{K \in \M, \ \s \in \E_{int}\cup \E_{ext}^N} 
		\left ( |w^N_K|, |w^P_K|, |w^N_\s|, |w^P_\s| \right ) .
\]
Since $\V_{\D,0}^2$ is a finite-dimensional vector space of dimension $ 2\left (|\M|+|\E_{int}|+ |\E_{ext}^N| \right )$  
there exists a constant $c_{\dim}$ only depending on $\dim(\V_{\D,0})$ such that 
\begin{equation} \label{eq:eqnorm}
	\forall \ww_\D \in \V_{\D,0}^2, \  c_{\dim} \| \ww_\D \| \leq\|\ww_\D \|_\infty.
\end{equation} 

We can now establish the existence result by induction on $n$. Let $\Nd_\D^n$ and $\Pd_\D^n$ be two $I_h$-valued vectors, we want to show that there exists a solution to~\eqref{sch:DD:w-}-\eqref{sch:DD:phi}.
Given $\ww_\D =(\w_\D^N, \w_\D^P) \in \V_{\D,0} ^2$, (and $(\Nd_\D , \Pd_\D, \psid_\D)$ the associated discrete densities and electrostatic potential), the application  
\begin{multline*}
	\vv_\D \mapsto 
		\int_\Omega \frac{N_\M - N_\M^n}{\Delta t} v_\M^N +  T_\D^N (\Nd_\D, {\w_\D^N}, \v_\D^N)
	     	 + \int_\Omega R(N_\M,P_\M) v_\M^N  \\
	    + \int_\Omega \frac{P_\M - P_\M^n}{\Delta t} v_\M^P +  T_\D^P (\Pd_\D, {\w_\D^P}, \v_\D^P)
	     	+ \int_\Omega R(N_\M,P_\M) v_\M^P
\end{multline*}
is a continuous linear form on $\V_{\D,0}^2$.
Hence, there exists a unique $\G(\ww_\D) \in  \V_{\D,0}^2$ such that 
\begin{multline}
	\forall \vv_\D \in \V_{\D,0}^2, \ \langle \G(\ww_\D) , \vv_\D \rangle  = 
		\int_\Omega \frac{N_\M - N_\M^n}{\Delta t} v_\M^N +  T_\D^N (\Nd_\D, {\w_\D^N}, \v_\D^N)
	     	 + \int_\Omega R(N_\M,P_\M) v_\M^N  \\
	    + \int_\Omega \frac{P_\M - P_\M^n}{\Delta t} v_\M^P +  T_\D^P (\Pd_\D, {\w_\D^P}, \v_\D^P)
	     	+ \int_\Omega R(N_\M,P_\M) v_\M^P, 
\end{multline}
and it is straightforward to see that $\G(\ww_\D) = (\zero_\D, \zero_\D)$ if and only if $(\Nd_\D,  \Pd_\D, \phid_\D)$ is a solution to~\eqref{sch:DD:w-}-\eqref{sch:DD:phi}. 
Moreover, by the continuity result of Theorem~\ref{th:discreteNLPoisson}, the vector field $\G : \V_{\D,0}^2 \to \V_{\D,0}^2 $ is continuous.
As in~\cite{CHHLM:22}, our proof relies on two key results.
The first one, that can be found, e.g., in~\cite[Section 9.1]{Evans:10}, is a corollary of Brouwer's fixed-point theorem.
\begin{lemma} \label{lem:exist}
  Let $k$ be a positive integer and let $P : \R ^k \to \R^k $ be a continuous vector field. Assume that there is $r > 0$ such that 
  \begin{equation*}
    P(x) \cdot x \geq 0 \qquad\text{ if \;}  |x| = r. 
  \end{equation*}
  Then, there exists a point $x_0 \in \R^k$ such that $P(x_0) =0$ and $|x_0| \leq r$.
\end{lemma}
\noindent
The second lemma, whose proof is postponed until Appendix~\ref{ap:lemma}, establishes that bounds on the discrete entropy and dissipation terms imply bounds on the discrete quasi-Fermi potentials.
\begin{restatable}{lemma}{lem} \label{lem:bounds}
Let $(\w_\D^N, \w_\D^P) \in \V_{\D,0}^2$, and assume that there exists $B_\sharp \geq 0$ such that
\begin{equation} \label{Entropycontrol}
	 \N(\w_\D^N, \w_\D^P)  \leq B_\sharp \qquad\text{ and } \qquad \Diss(\w_\D^N, \w_\D^P)  \leq B_\sharp.
\end{equation}
Then, there exists $C_\sharp > 0$, depending on the data, $B_\sharp$ and $\D$ such that 
\[
	-C_\sharp \one_\D \leq \w_\D^N \leq C_\sharp \one_\D, \  -C_\sharp \one_\D \leq \w_\D^P \leq C_\sharp \one_\D
		\text{ and } -C_\sharp \one_\D \leq \phid_\D \leq C_\sharp \one_\D, 
\]	
where $\phid_\D$ is the electrostatic potential associated to $(\w_\D^N, \w_\D^P)$.
\end{restatable}
We are now in position to prove the existence of solutions to the scheme and the estimates on these solutions.
\begin{proof}[Proof of existence - Theorem~\ref{th:DD}]
In order to use the result of Lemma~\ref{lem:exist}, notice that by definition of $\G$, and because $w_\M^N + w_\M^P = h(N_\M) + h(P_\M)$, we have
\[		
	\forall \ww_\D \in \V_{\D,0}^2, \ 
		\langle \G(\ww_\D) , \ww_\D \rangle  =
 		\int_\Omega \frac{N_\M - N_\M^n}{\Delta t} w_\M^N + \int_\Omega \frac{P_\M - P_\M^n}{\Delta t} w_\M^P  + 
 		\Diss(\w_\D^N, \w_\D^P),.
\]
Since $(\w_\D^N, \w_\D^P)$ are the quasi-Fermi potentials associated to $(\Nd_\D, \Pd_\D)$, the result of Lemma~\ref{lem:test} ensures that 
\begin{equation}\label{eq:Gw.w}
	\forall \ww_\D \in \V_{\D,0}^2, \  \langle \G(\ww_\D) , \ww_\D \rangle  \geq 
 		\frac{\N(\w_\D^N, \w_\D^P) -\N(\Nd_\D^n,\Pd_\D^n) }{\Delta t } + \Diss(\w_\D^N, \w_\D^P).
\end{equation}
Let $B^n = \N(\Nd_\D^n,\Pd_\D^n)$. According to Lemma~\ref{lem:bounds}, there exists $C^n>0$ depending on the data, $\D$ and $\Delta t$ such that 
\begin{equation}\label{eq:implybound}
\text{ if } \N(\w_\D^N, \w_\D^P)+ \Delta t \, \Diss(\w_\D^N, \w_\D^P) \leq B^n \text{, then } \| \ww_\D \|_\infty \leq C^n.
\end{equation} 
Now, letting $r^n =2C^n/c_{\dim}$, one can notice that for any $\ww_\D \in \V_{\D,0}^2$ such that $\|\ww_\D \| = r^n$, we have $C^n< 2C^n = c_{\dim} \|\ww_\D \|\leq \|\ww_\D \|_\infty$. Therefore, by~\eqref{eq:Gw.w} and contraposition of~\eqref{eq:implybound}, if $\|\ww_\D \| = r^n$, one has 
\[
	\Delta t \langle \G(\ww_\D) , \ww_\D \rangle  \geq
		\N(\w_\D^N, \w_\D^P)+ \Delta t\,  \Diss(\w_\D^N, \w_\D^P) - B^n > B^n-B^n \geq 0.
\]
 Thus, we can use the result of Lemma~\ref{lem:exist} applied to the vector field $\G$, and there exists at least one solution $(\Nd_\D^{n+1},  \Pd_\D^{n+1}, \phid_\D^{n+1})$ to~\eqref{sch:DD:w-}-\eqref{sch:DD:phi}.

To prove the bounds~\eqref{eq:unifpos}, it suffices to note that the dissipation relation~\eqref{eq:diss} implies that 
\[	
	\N^{n+1} \leq \N^n \leq \N^0 \text{ and } \Diss^{n+1} \leq \frac{\N^0}{\Delta t }.
\]
Hence, by Lemma~\ref{lem:bounds}, there exists $C_\sharp$ depending on the data, $\D$, $\Delta t$ and $\N^0$ (but not on $n$) such that 
\[
	 -C_\sharp \one_\D \leq \w_\D^{N,n+1} \leq C_\sharp \one_\D, \; -C_\sharp \one_\D \leq \w_\D^{P,n+1} \leq C_\sharp \one_\D \text{ and } -C_\sharp \one_\D \leq \phid_\D^{n+1}  \leq C_\sharp \one_\D .
\]
 Therefore, by \eqref{def:densities}, $g(-2C_\sharp + \alpha_N) \one_\D \leq \Nd_\D^{n+1} \leq g(2C_\sharp + \alpha_N)\one_\D$, and an analogous estimate holds for $\Pd_\D^{n+1}$. 
\end{proof}
\subsection{Long-time behaviour}
In this section, we analyse the long-time behaviour of the scheme~\eqref{sch:DD}. The analysis proposed here relies on an adaptation of the arguments presented in~\cite{BCCH:17} in the framework of TPFA schemes.
\begin{theorem}[Discrete long-time behaviour] \label{th:longtime}
Let $\left ( (\Nd_\D^n, \Pd_\D^n, \phid_\D^n) \right ) _{n \in \mathbb{N}} \in {\V_\D^3 }^\mathbb{N}$ be a solution to the coupled scheme~\eqref{sch:DD}. 
There exists $\nu_{\D,\Delta t}>0$ depending on the data, $\D$ and $\Delta t$ such that 
\begin{equation} \label{eq:entroexp}
	\forall n \in \mathbb{N}, \, \N^{n+1} \leq (1+ \nu_{\D,\Delta t} \Delta t) ^{-1} \N^n.
\end{equation}
Moreover, the discrete solution converges geometrically fast towards the discrete thermal equilibrium: there exists a positive constant $c_{\D,\Delta t}$ depending on the data, $\D$ and $\Delta t$ such that
\begin{equation} \label{eq:expnorm}
	\forall n \in \mathbb{N}, \,  
		\|N^n_\M -N^e_\M  \|_{L^2(\Omega)}^2 +
	 	\| P^n_\M -P^e_\M\|_{L^2(\Omega)}^2 + 
	 	\| \phi^n_\M -\phi^e_\M\|_{L^2(\Omega)}^2 			
	 	\leq c_{\D,\Delta t} \N^0 (1+ \nu_{\D,\Delta t} \Delta t) ^{-n}.
\end{equation}
\end{theorem}
\begin{proof}
In this proof, we will denote by $c$ a generic constant depending on the data, $\D$ and $\Delta t$ (but not on $n$).
First, recall that the solutions satisfy the bounds \eqref{eq:unifpos}. One can notice (using a Taylor expansion of $H$) that there exists two positive constants $\tilde{c_1}$ and $\tilde{c_2}$ such that for any $(x,y) \in [M_\flat, M_\sharp]$, 
\begin{equation}\label{eq:ineg}
\tilde{c_1} (x-y)^2 \leq H(x) - H(y) - h(y)(x-y) \leq \tilde{c_2} (x-y)^2.
\end{equation}
Therefore, letting 
$\displaystyle \hat{\N}^n = \|N^n_\M - N^e_\M \|_{L^2(\Omega)}^2 +\|P^n_\M - P^e_\M \|_{L^2(\Omega)}^2 + \frac{1}{2}a^\phi_\D(\phid_\D - \phid_\D^e,\phid_\D - \phid_\D^e) $ and
using \eqref{eq:unifpos} and \eqref{eq:ineg}, we deduce that there exist $c_1$ and $c_2$ positive such that 
\begin{equation} \label{eq:entrovshat}
	c_1 \hat{\N}^n  \leq \N^n \leq c_2 \hat{\N}^n.
\end{equation} 
Since $\N(\Nd_\D^e,\Pd_\D^e) = 0$, by Lemma~\ref{lem:test} (used with $(\Nd_\D,\Pd_\D) = (\Nd_\D^n,\Pd_\D^n)$
and $(\Nd_\D^n,\Pd_\D^n) = (\Nd_\D^e,\Pd_\D^e)$), one gets 
\[ 
	c_1 \hat{\N}^n\leq \N^n \leq \int_\Omega \left (N_\M^n - N_\M^e \right ) w_\M^{N,n} + \int_\Omega  \left (P_\M^n - P_\M^e \right ) w_\M^{P,n}.
\]
Then, we use Young inequality (with scaling $c_1$) to obtain
\[ 
	c_1 \hat{\N}^n \leq \frac{c_1}{2} \left ( \| N_\M^n - N_\M^e \|_{L^2(\Omega)}^2 + \| P_\M^n - P_\M^e \|_{L^2(\Omega)}^2  \right )
		 +  \frac{1}{2c_1} \left ( \|w_\M^{N,n} \|_{L^2(\Omega)}^2 + \|w_\M^{P,n}  \|_{L^2(\Omega)}^2  \right ).
\]
Notice that $\| N_\M^n - N_\M^e \|_{L^2(\Omega)}^2 + \| P_\M^n - P_\M^e \|_{L^2(\Omega)}^2 \leq \hat{\N}^n$, therefore we have 
\[
	\frac{c_1}{2} \hat{\N}^n \leq \frac{1}{2c_1} \left ( \|w_\M^{N,n} \|_{L^2(\Omega)}^2 + \|w_\M^{P,n}  \|_{L^2(\Omega)}^2  \right ).
\]
Combining this estimate with~\eqref{eq:entrovshat}, we deduce that 
\begin{equation} \label{eq:entrohat}
	\frac{c_1}{ c_2} \N^n \leq c_1\hat{\N}^n \leq  \frac{1}{c_1} \left ( \|w_\M^{N,n} \|_{L^2(\Omega)}^2 + \|w_\M^{P,n}  \|_{L^2(\Omega)}^2  \right ).
\end{equation}
On the other hand, one has $\Diss^n \geq T_\D^N (\Nd_\D^n, \w_\D^{N,n}, \w_\D^{N,n}) + T_\D^P (\Pd_\D^n, \w_\D^{P,n}, \w_\D^{P,n})$, and, for any $K\in\M$, using~\eqref{def:TK} alongside with the positivity of $r_K(\Nd_K)$ and the local coercivity~\eqref{def:localcoercivity} of $a^N_K$, one gets that 
\[
	T_K^N (\Nd_K^n, \w_K^{N,n}, \w_K^{N,n})\geq r_K(\Nd_K^n) a_K^N(\w_K^{N,n}, \w_K^{N,n}) 
	\geq \alpha_\flat \lambda_\flat   r_K(\Nd_K^n) |\w_K^{N,n}|_{1,K}^2 .
\]
Moreover, by definition of $r_K$, properties on $m$ and $m_h$ and the bounds \eqref{eq:unifpos} on $\Nd_\D^n$, we deduce that 
\[
	r_K(\Nd_K^n) = \sum_{\s \in \E_K} \frac{m(N_K^n, N_\s^n)}{|\E_K|} 
		\geq  \sum_{\s \in \E_K} \frac{m_h(N_K^n, N_\s^n)}{|\E_K|} 
		\geq  \sum_{\s \in \E_K} \frac{m_h(M_\flat,  M_\flat)}{|\E_K|} \geq M_\flat.
\]
The same holds for the term $T_\D^P (\Pd_\D^n, \w_\D^{P,n}, \w_\D^{P,n})$, therefore by the discrete Poincar\'e inequality \eqref{eq:poinca} one has 
\begin{equation} \label{eq:dissw}
	\Diss^n 
		\geq  \alpha_\flat \lambda_\flat M_\flat \left ( |\w_\D^{N,n}|_{1,\D}^2 + |\w_\D^{P,n}|_{1,\D}^2 \right )
		\geq  \frac{\alpha_\flat \lambda_\flat M_\flat}{C_{P,\Gamma^D}^2} \left ( \|w_\M^{N,n}\|_{L^2(\Omega)}^2 + \|w_\M^{P,n}\|_{L^2(\Omega)}^2  \right ).
\end{equation}
Combining \eqref{eq:entrohat} and \eqref{eq:dissw}, we get the following control on the entropy:
\[
	\nu_{\D,\Delta t} \N^n \leq \Diss^n \text{ with } \nu_{\D,\Delta t} = \frac{\alpha_\flat \lambda_\flat M_\flat c_1^2}{c_2 C_{P,\Gamma^D}^2}.
\]
From this inequality and the entropy dissipation relation~\eqref{eq:diss}, we deduce~\eqref{eq:entroexp}. To get~\eqref{eq:expnorm}, notice that $c_1 \hat{\N}^n  \leq \N^n$, so using the discrete Poincar\'e inequality alongside with the coercivity of $a_\D^\phi$, we get
\[
	\|N^n_\M -N^e_\M  \|_{L^2(\Omega)}^2 + \| P^n_\M -P^e_\M\|_{L^2(\Omega)}^2 + \frac{\alpha_\flat \lambda_\flat^\phi }{2 C_{P,\Gamma^D}^2}\| \phi^n_\M -\phi^e_\M\|_{L^2(\Omega)}^2 \leq \hat{\N}^n \leq \frac{1}{c_1} \N^n.
\] 
We conclude by using~\eqref{eq:entroexp}.
\end{proof}
\begin{rem}[Non-uniformity with respect to the mesh] \label{rem:nonuniexp}
The result of Theorem~\ref{th:longtime} states a geometric convergence of the solutions to the scheme towards the discrete thermal equilibrium.
However, the constants involved in~\eqref{eq:entroexp} and~\eqref{eq:expnorm} may depend strongly on the mesh (and in particular on $h_\D$), which yelds a weaker result than the one of~\cite{BCCH:17,BCCH:19} for TPFA schemes. Such a dependency is a reminiscence of the non-monotonicity of the HFV schemes (see Remark~\ref{rem:boundnonuniform}).
In practice, the long-time behaviour of the scheme seems not to depend on the meshsize: see Section~\ref{sec:num:longtime}.
\end{rem}

\section{Numerical results} \label{sec:num}
In this section, we give some numerical evidences of the good behaviour of the scheme~\eqref{sch:DD}. 
We use test-cases inspired by the 2D PN-junction studied in~\cite{BCCH:17}, whose geometry is described in Figure~\ref{fig:diode}.
The domain $\Omega$ is the unit square ${]0,1[}^2$.
For the boundary conditions, we split $\Gamma^D = \Gamma^D_0 \cup \Gamma^D_1 $ with $\Gamma^D_0 = [0,1] \times \{0\}$ and $\Gamma^D_1 = [0, 0.25] \times \{1\}$.
For $i \in \{0,1\}$, we let
\[ 
	N^D = N^D_i , \  P^D = P^D_i \text{ and } \phi^D = \frac{h(N^D_i) - h(P^D_i)}{2}  \text { on } \Gamma^D_i.
\]
To be consistent with the compatibility condition~\eqref{eq:compcond} we assume that there exists a constant $\alpha_0$ such that $\displaystyle 	h(N^D) +  h(P^D) = \alpha_0 $, therefore for given $N^D$ and $\alpha_0$ we set $P^D = g \left(\alpha_0 - h(N^D) \right)$ on $\Gamma^D$. 
Thus, we get $\alpha_N = \alpha_P = \frac{\alpha_0}{2}$. If $r \neq 0$, we finally impose that $\alpha_0 = 0$ to satisfy~\eqref{eq:alpha}. 
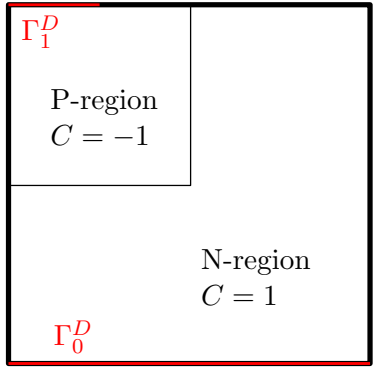
\begin{figure}[ht] 
\begin{center}
{\scalefont{1}
\def\svgwidth{5cm}
\begingroup%
  \makeatletter%
  \providecommand\color[2][]{%
    \errmessage{(Inkscape) Color is used for the text in Inkscape, but the package 'color.sty' is not loaded}%
    \renewcommand\color[2][]{}%
  }%
  \providecommand\transparent[1]{%
    \errmessage{(Inkscape) Transparency is used (non-zero) for the text in Inkscape, but the package 'transparent.sty' is not loaded}%
    \renewcommand\transparent[1]{}%
  }%
  \providecommand\rotatebox[2]{#2}%
  \newcommand*\fsize{\dimexpr\f@size pt\relax}%
  \newcommand*\lineheight[1]{\fontsize{\fsize}{#1\fsize}\selectfont}%
  \ifx\svgwidth\undefined%
    \setlength{\unitlength}{419.52755906bp}%
    \ifx\svgscale\undefined%
      \relax%
    \else%
      \setlength{\unitlength}{\unitlength * \real{\svgscale}}%
    \fi%
  \else%
    \setlength{\unitlength}{\svgwidth}%
  \fi%
  \global\let\svgwidth\undefined%
  \global\let\svgscale\undefined%
  \makeatother%
  \begin{picture}(1,1)%
    \lineheight{1}%
    \setlength\tabcolsep{0pt}%
    \put(0.52925868,0.27331711){\color[rgb]{0,0,0}\makebox(0,0)[lt]{\lineheight{1.25}\smash{\begin{tabular}[t]{l}N-region \\ $C = 1$\end{tabular}}}}%
    \put(0.13280309,0.69886021){\color[rgb]{0,0,0}\makebox(0,0)[lt]{\lineheight{1.25}\smash{\begin{tabular}[t]{l}P-region \\ $C = -1$\end{tabular}}}}%
    \put(0.14142531,0.0714536){\color[rgb]{0,0,0}\makebox(0,0)[lt]{\lineheight{1.25}\smash{\begin{tabular}[t]{l}$\color{red}{\Gamma^D_0}$\end{tabular}}}}%
    \put(0,0){\includegraphics[width=\unitlength,page=1]{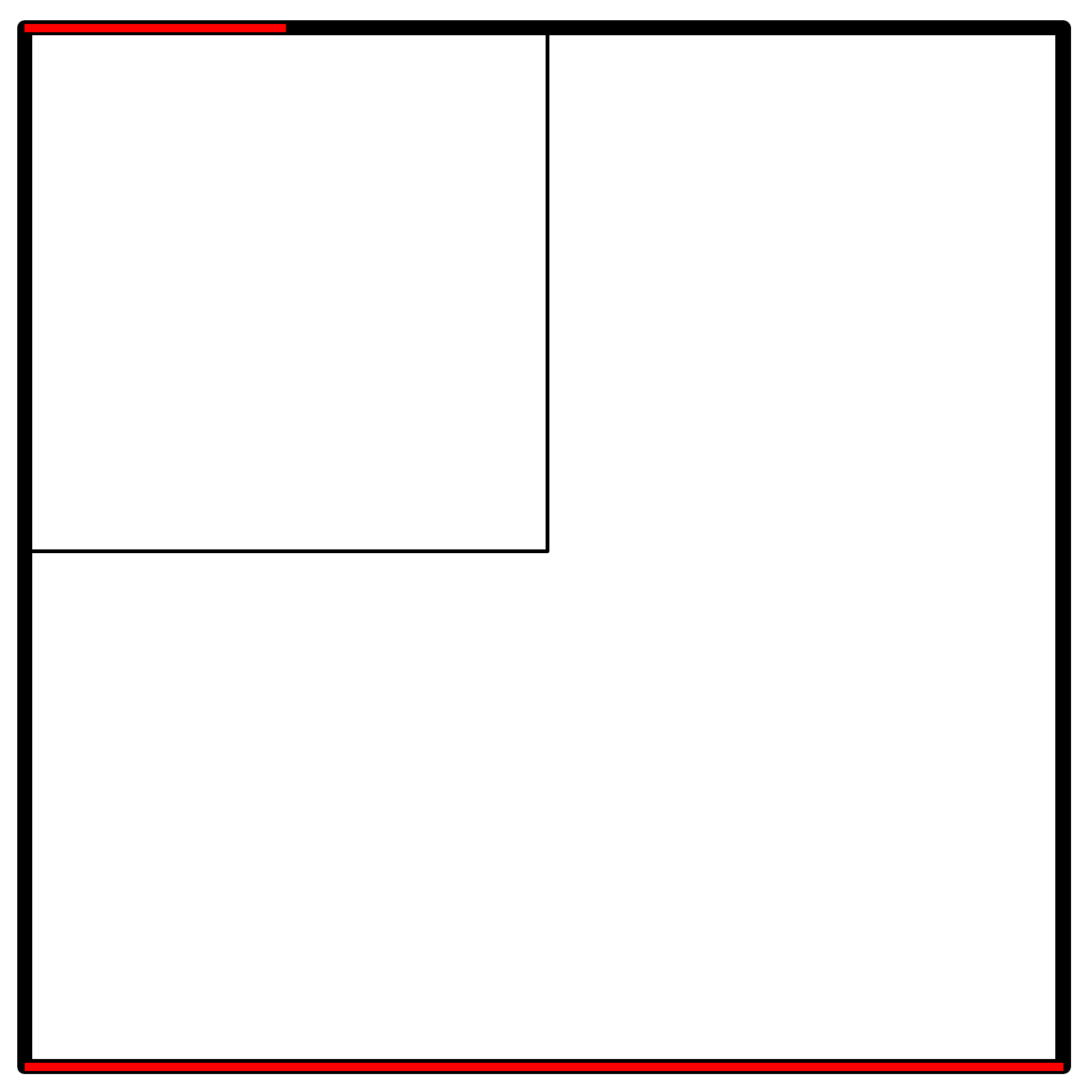}}%
    \put(0.05636858,0.88043292){\color[rgb]{0,0,0}\makebox(0,0)[lt]{\lineheight{1.25}\smash{\begin{tabular}[t]{l}$\color{red}{\Gamma^D_1}$\end{tabular}}}}%
  \end{picture}%
\endgroup%

}
\end{center}
\caption{The geometry of the PN diode.}
\label{fig:diode}
\end{figure}
We use the following initial condition: 
\[
	N_0(x,y) = N^D_1 + (N^D_0 - N^D_1) (1 - \sqrt{y}) \text{ and }
	P_0(x,y) = P^D_1 + (P^D_0 - P^D_1) (1 - \sqrt{y}).
\]
Finally, we use a piecewise constant doping profile $C$, equal to $-1$ in the P-region and $1$ in the N-region (see Figure~\ref{fig:diode}).
%
Concerning the tensors, we assume that the permittivity is isotropic, of the form $\Lambda_\phi = \lambda^2 I_2$, where $\lambda>0$ is the rescaled Debye length. We also assume that the magnetic field is constant, of magnitude $b\geq0$, and that the rescaled mobilities are equal to $1$ ($\mu_N = \mu_P = 1$), therefore the tensors for the convection-diffusion equations are
\begin{equation} \label{def:tensors}
	\Lambda_N = \frac{1}{1+b^2}\begin{pmatrix}
	1 & b \\ 
	-b & 1
	\end{pmatrix} 
	\text{ and }
	 \Lambda_P = \frac{1}{1+b^2}\begin{pmatrix}
	1 & -b \\ 
	b & 1
	\end{pmatrix} .
\end{equation}
\subsection{Implementation}
In this section, we discuss some practical details concerning the implementation of the schemes introduced in this paper.
The mesh families used for the numerical tests are classical Cartesian and triangular families - see for example \cite{HerHu:08}- and a tilted hexagonal-dominant mesh family, depicted on Figure \ref{fig:meshes}.
\begin{figure}[h]
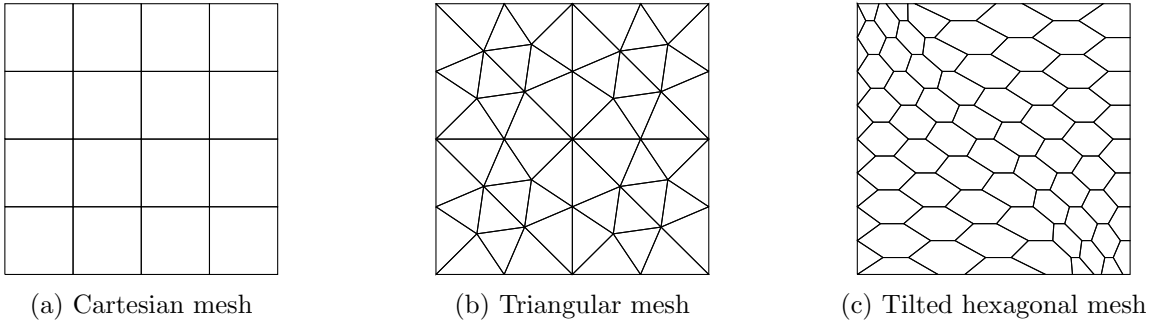

\begin{subfigure}{.33\textwidth}
  \centering
  \includegraphics[width=.65\textwidth]{meshes/mesh2_1.tikz}
  \caption{Cartesian mesh}
  \label{fig:cartesian}
\end{subfigure}
\begin{subfigure}{.33\textwidth}
  \centering
  \includegraphics[width=.65\textwidth]{meshes/mesh1_1.tikz}
  \caption{Triangular mesh}
  \label{fig:triangular}
\end{subfigure}%
\begin{subfigure}{.33\textwidth}
  \centering
  \includegraphics[width=.65\textwidth]{meshes/pi6_tiltedhexagonal_2.tikz}
  \caption{Tilted hexagonal mesh}
  \label{fig:tiltedhexa}
\end{subfigure}%
\caption{\textbf{Implementation.} Coarsest meshes of each family used in the numerical tests.}
\label{fig:meshes}
\end{figure}
These meshes have convex cells, hence we always choose $x_K$ to be the barycentre of $K$. 
Moreover, we use a fixed stabilisation parameter $\eta = 1.5$ (see \eqref{def:stabilisation}).
In the simulations showed below, we use the arithmetic mean as an $m$ function for the reconstruction defined in \eqref{eq:rK}, therefore for $\u_\D \in \V_\D$ and $K \in \M$, 
\[
	r_K(\u_K) = \frac{1}{2} \left ( u_K + \frac{1}{|\E_K|}\sum_{\s \in \E_K}  u_\s \right ).
\]
For a discussion on other choices of reconstructions we refer to \cite[Section 6.2]{CCHHK:20}.
\subsubsection{Finite volume formulation} \label{sec:flux}
The hybrid schemes described in this paper define finite volume methods, in the sense that they can be equivalently rewritten under a conservative forms, with local mass balances, flux equilibration at interfaces, and boundary conditions. For more details about this formulation in the framework of a linear Poisson equation, we refer the reader to~\cite{EGaHe:10}. Let us detail succinctly this formulation for our schemes. Let $\Lambda$ be a generic uniformly elliptic tensor.
For all $K\in\M$, and all $\s\in\E_K$, in the framework of HFV methods, the normal diffusive flux $ - \int_{\s} \Lambda \nabla u \cdot n_{K,\s}$ is approximated by the following numerical flux:
\begin{equation} \label{diff.flux}
	{F_{K,\s}^{\Lambda} (\u_K)} = \sum_{\s ' \in \E_K } A_K ^{\s \s'} (u_K - u_{\s'} ),
\end{equation}
where the $A_K^{\s\s'}$ are defined by 
\begin{equation} \label{def:flux:A}
		A_K ^{\s\s'}  =  \sum_{\s'' \in \E_K } |P_{K,\s''}|\,y_K ^{\s'' \s } \cdot \Lambda _{K, \s''} y_K^{\s''\s'},
\end{equation} 
and the $y_K^{\s\s'}\in\R^d$ only depend on the geometry of the discretisation $\D$ (see, for example, \cite[Eq.~(2.22)]{EGaHe:10} for an exact definition with $\eta = \sqrt{d}$).
For all $K\in\M$, one can express the local discrete bilinear form $a_K^{\Lambda}$ in terms of the local fluxes $\big(F^{\Lambda}_{K,\s}\big)_{\s\in\E_K}$: for all $(\u_K,\v_K)\in\V_K^2$, 
\begin{equation*} \label{aKL-flux}
	a_K^{\Lambda}(\u_K, \v_K) = \sum_{\s \in \E_K } F^{\Lambda}_{K,\s}(\u_K)(v_K - v_\s).
\end{equation*}
With these considerations, the scheme~\eqref{sch:NLPoisson} for the Poisson equation~\eqref{pb:NLPoisson} writes under the following form:
\begin{subequations}\label{sch:NLPoisson:flux}
        \begin{empheq}[left = \empheqlbrace]{align}
			\sum _{\s \in \E _K} F^{\Lambda_\phi}_{K,\s}(\phid_K) 
				&= |K| \left (C_K + g(z^P_K - \phi_K)  - g(z^N_K + \phi_K)  \right )   &&\forall K \in \M, \\
			F^{\Lambda_\phi}_{K,\s}(\phid_K) 
				&= -F^{\Lambda_\phi}_{L,\s}(\phid_L) &&\forall \s = K|L \in \E_{int}, \\
			\phi_\s &= \phi^D_\s  && \forall \s \in \E_{ext}^D, \\
			F^{\Lambda_\phi}_{K,\s}(\phid_K) &= 0  && \forall \s \in \E_{ext}^N \text{ with } \M_\s = \{K\}, 
        \end{empheq}
\end{subequations}
where the fluxes are defined by~\eqref{diff.flux}, and $C_K = \frac{1}{|K|} \int_K C$. The first equation corresponds to local balance, the second imposes the local conservativity of the fluxes at interfaces and the last one enforces the boundary conditions.

Following an analogous approach, we define the nonlinear flux for the advection-diffusion: for any $I_h$-valued $\u_K \in \V_K$ and $\phid_K \in \V_K$, we let
\begin{equation} \label{diff.fluxnl}
	\F^\Lambda_{K,\s}(\u_K,\phid_K) 
	= r_K (\u_K) F^\Lambda_{K,\s}\left (h(\u_K) + \phid_K \right).
\end{equation} 
Therefore, letting $\w_K = h(\u_K) + \phid_K - \alpha \one_K$ (with $\alpha \in \R$), one can write the local trilinear form $T^\Lambda_K$ in terms of the nonlinear flux: for any $\v_K \in \V_K$, 
\begin{equation*} \label{TKL-flux}
	T_K^{\Lambda}(\u_K, \w_K ,\v_K) = \sum_{\s \in \E_K } \F^{\Lambda}_{K,\s}(\u_K,\phid_K)(v_K - v_\s).
\end{equation*}
%
The scheme~\eqref{sch:DD} for the drift-diffusion system then writes under the following form:
\begin{subequations}\label{sch:DD:flux}
        \begin{empheq}{align}
        		&\forall K \in \M, &&|K|\frac{N^{n+1}_K - N^{n}_K}{\Delta t} + \sum _{\s \in \E _K} \F^{\Lambda_N}_{K,\s}(\Nd^{n+1}_K, -\phid^{n+1}_K)  = -|K| R(N^{n+1}_K,P^{n+1}_K),  \\
			&\forall K \in \M,&& |K|\frac{P^{n+1}_K - P^{n}_K}{\Delta t} + \sum _{\s \in \E _K} \F^{\Lambda_P}_{K,\s}(\Pd^{n+1}_K, \phid^{n+1}_K) = -|K| R(N^{n+1}_K,P^{n+1}_K),  \\
			&\forall K \in \M, && \sum _{\s \in \E _K} F^{\Lambda_\phi}_{K,\s}(\phid^{n+1}_K) 
				 = |K| \left (C_K+  P_K^{n+1} -N_K^{n+1}   \right ) ,  \\
			&\forall \s = K|L \in \E_{int},&& \F^{\Lambda_N}_{K,\s}(\Nd^{n+1}_K, -\phid^{n+1}_K) 
				 + \F^{\Lambda_N}_{L,\s}(\Nd^{n+1}_L, -\phid^{n+1}_L) =0 , \\
			&\forall \s = K|L \in \E_{int}, &&\F^{\Lambda_P}_{K,\s}(\Pd^{n+1}_K, \phid^{n+1}_K) 
				 + \F^{\Lambda_P}_{L,\s}(\Pd^{n+1}_L, \phid^{n+1}_L) =0,\\
			&\forall \s = K|L \in \E_{int},&& F^{\Lambda_\phi}_{K,\s}(\phid^{n+1}_K) 
				 +F^{\Lambda_\phi}_{L,\s}(\phid^{n+1}_L)=0 ,\\
			& \forall \s \in \E_{ext}^D, &&N^{n+1}_\s = N^D_\s, \  P^{n+1}_\s 
				 = P^D_\s \text{ and } \ \phi^{n+1}_\s = \phi^D_\s , \\
			& \forall \s \in \E_{ext}^N, 
				&&\F^{\Lambda_N}_{K,\s}(\Nd^{n+1}_K, -\phid^{n+1}_K) 
				= \F^{\Lambda_P}_{K,\s}(\Pd^{n+1}_K, \phid^{n+1}_K)  = F^{\Lambda_\phi}_{K,\s}(\phid^{n+1}_K) = 0, \label{sch:flux:neu}
        \end{empheq}
\end{subequations}
where the nonlinear fluxes are defined by~\eqref{diff.fluxnl}, the initial data $(N^0_K, P^0_K)_{K \in \M}$ are defined as in~\eqref{sch:DD:ini} and the cell $K$ in~\eqref{sch:flux:neu} is such that $\M_\s = \{K\}$.

Note that these formulations yield nonlinear systems. Thus, we can introduce natural functions $\G^{sta}: \V_\D \to \V_\D$ and $\G^{n,\Delta t} : \V_\D^3 \to \V_\D^3$ such that \eqref{sch:NLPoisson:flux} rewrites $\G^{sta}(\phid_\D) =\zero_\D$ and \eqref{sch:DD:flux} writes $\G^{n,\Delta t}(\Nd^{n+1}_\D,\Pd^{n+1}_\D, \phid^{n+1}_\D) = (\zero_\D,\zero_\D,\zero_\D)$. 
The two functions are regular on their domains.

\subsubsection{Newton's method and static condensation} \label{sec:Newton}
The implementation of the nonlinear schemes relies on their finite volume formulation. To fix ideas, we consider the case of the transient scheme~\eqref{sch:DD:flux}. Given $(\Nd^{n}_\D, \Pd_\D^n) \in \V_\D^2$ $I_h$-valued, we want to solve the nonlinear system $\G^{n,\delta t}(\Nd^{n+1}_\D,\Pd^{n+1}_\D, \phid^{n+1}_\D) =(\zero_\D,\zero_\D,\zero_\D)$ (with a time step $\delta t$ instead of $\Delta t$).
The resolution of this system relies on a Newton method with time step adaptation.

First, one initialises the method with initial guess $(\tilde{\Nd}^{n}_\D, \tilde{\Pd}_\D^n) \in \V_\D^2$, where the coordinates of $\tilde{\Nd}^{n}_\D$ (respectively $\tilde{\Pd}_\D^n$) are the projections of the coordinates of $\Nd^{n}_\D$ (resp. $\Pd_\D^n$) on $[\epsilon, \dmax -\epsilon]$ (if $\dmax = +\infty$, we project on $[\epsilon, +\infty[$) in order to avoid potential problems due to the singularity of $h$ near $0$ and $\dmax$. \newline
The computation of the residues follows the process described below: 
let us denote by $R_\M \in \R^{3|\M|}$ and $R_\E \in \R^{3|\E|}$ the residue vectors $(r_K^N, r_K^P, r^\phi_K)_{K\in\M}$ and $(r_\s^N, r_\s^P, r^\phi_\s)_{\s \in \E}$. They are solution to the following linear block system: 
\begin{equation} \label{sysmatrix}
	\begin{pmatrix}
	\Mat_\M & \Mat_{\M,\E} \\ 
	\Mat_{\E,\M} & \Mat_\E
	\end{pmatrix} 
	\begin{pmatrix} R_\M \\  R_\E \end{pmatrix} 
	= \begin{pmatrix} S_\M \\  S_\E \end{pmatrix},  
\end{equation}
where $\Mat_\M \in \R^{3|\M| \times 3|\M| }$, $\Mat_{\M,\E} \in \R^{3|\M| \times 3|\E|}$, $\Mat_{\E,\M} \in \R^{3|\E| \times 3|\M|}$, $\Mat_\E \in \R^{3|\E|\times 3|\E|}$, and  $S_\M$ and $S_\E$ are vectors of size $3|\M|$ and $3|\E|$ issued from the previous iteration.
By construction, the matrix $\Mat_\M$ is block diagonal with $3 \times 3$ invertible diagonal blocks.
Therefore, this matrix can be inverted at a very low computational cost, inverting only small matrices.
Thus, we can eliminate the cell unknowns, noticing that
\begin{equation} \label{bulkunk}
	R_\M = \Mat_\M ^{-1} \left ( S_\M - \Mat_{\M,\E} R_\E \right ).
\end{equation}
Using this relation, one shows that $R_\E$ is the solution to the following linear system:
\begin{equation}\label{squelunk}
	\left ( \Mat_\E - \Mat_{\E,\M}\Mat_\M ^{-1} \Mat_{\M,\E} \right ) R_\E 
		= S_\E - \Mat_{\E,\M} \Mat_\M ^{-1} S_\M, 
\end{equation}
where $\displaystyle \Mat_\D = \Mat_\E - \Mat_{\E,\M}\Mat_\M ^{-1} \Mat_{\M,\E}$, the Schur complement of the block $\Mat_\M$, is an invertible matrix of size $3|\E| \times 3|\E|$. 
In practice, we solve the linear system \eqref{squelunk} using an LU factorization algorithm, and we use the solution $R_\E$ in order to compute $R_\M$ from \eqref{bulkunk}. This technique, called static condensation, allows one to replace a system of size $3(|\E|+|\M|)$ by a system of size $3|\E|$ without any additional fill-in.
As a stopping criterion for the Newton iterations, we compare the $l^\infty$ relative norm of the residue with a threshold $tol$.
If the method does not converge after $i_{max}$ iterations, we divide the time step by $2$ and restart the resolution. 
When the method converges, one can proceed to the approximation of $(\Nd^{n+2}_\D,\Pd^{n+2}_\D, \phid^{n+2}_\D)$, with an initial time step of $\min(\Delta t,  1.4 \times  \delta t)$. 
\newline
The initial time step (used to compute $(\Nd^{1}_\D,\Pd^{1}_\D, \phid^{1}_\D)$) is $\Delta t$. In practice, we use $\epsilon = 10^{-9}$, $i_{max} = 50$ and $tol = 10^{-10}$.

For the discrete thermal equilibrium, we solve the nonlinear system $\G^{sta}(\phid_\D) =\zero_\D$ (with $(z^P,z^N) = (\alpha_N, \alpha_P)$) using a similar approach, based on a Newton's method alongside with a continuation method.
Moreover, note that the counterpart in this case of the matrix $\Mat_\M$ is diagonal with non-zero entries, so its inversion is straightforward.
\subsection{Proof of concept} \newcounter{numtest} \addtocounter{numtest}{1}
In this section, we are interested in qualitative and quantitative properties of the discrete solutions, and we present the profiles of some computed discrete solutions.

\begin{figure}[!htb]
\begin{subfigure}{0.32\textwidth}\caption{$P_\M^0$} 
  \includegraphics[width=\linewidth]{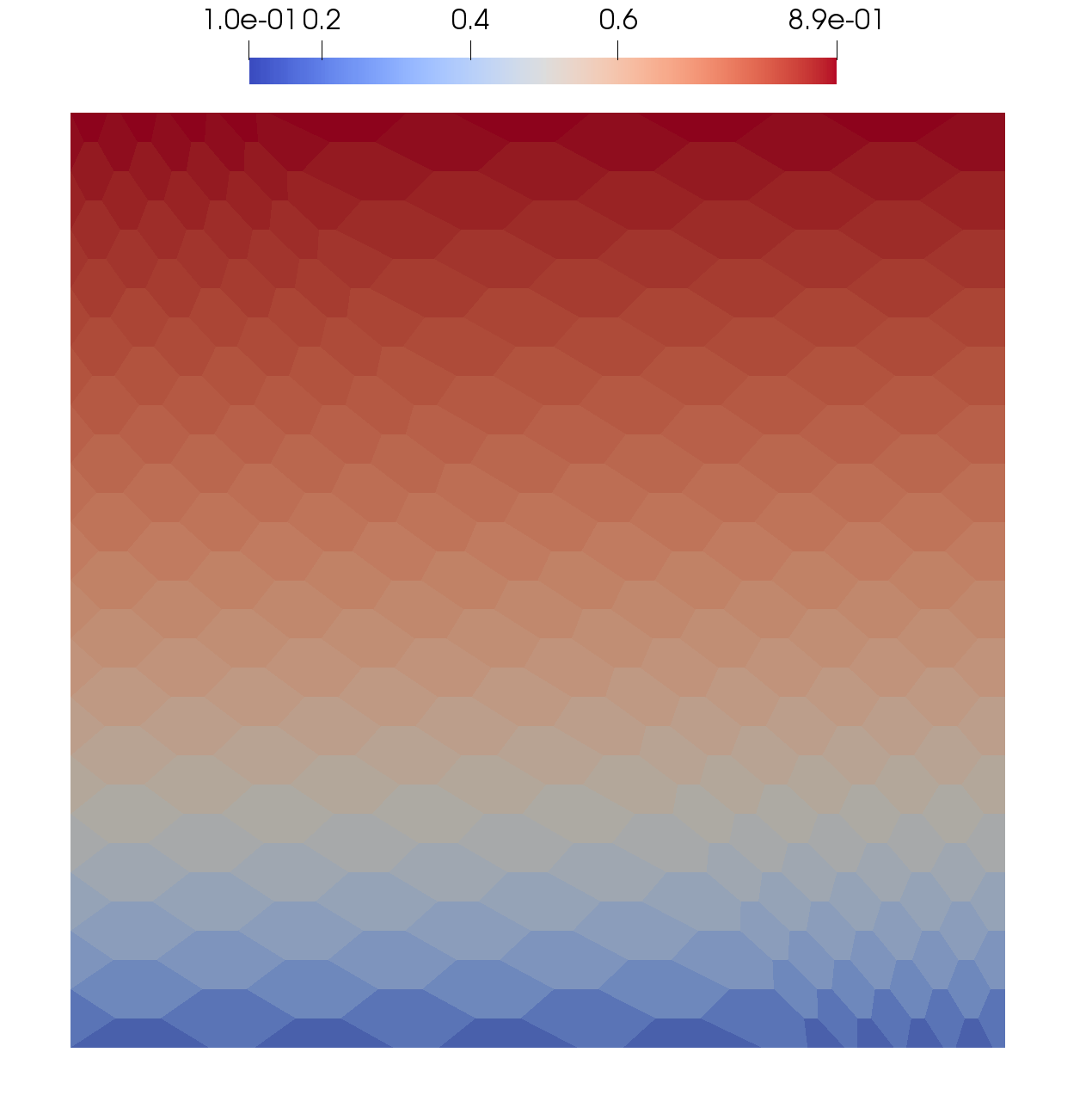}
\end{subfigure}\hfill
\begin{subfigure}{0.32\textwidth}  \caption{$P_\M^n$ with $t^n = 0.1$} 
  \includegraphics[width=\linewidth]{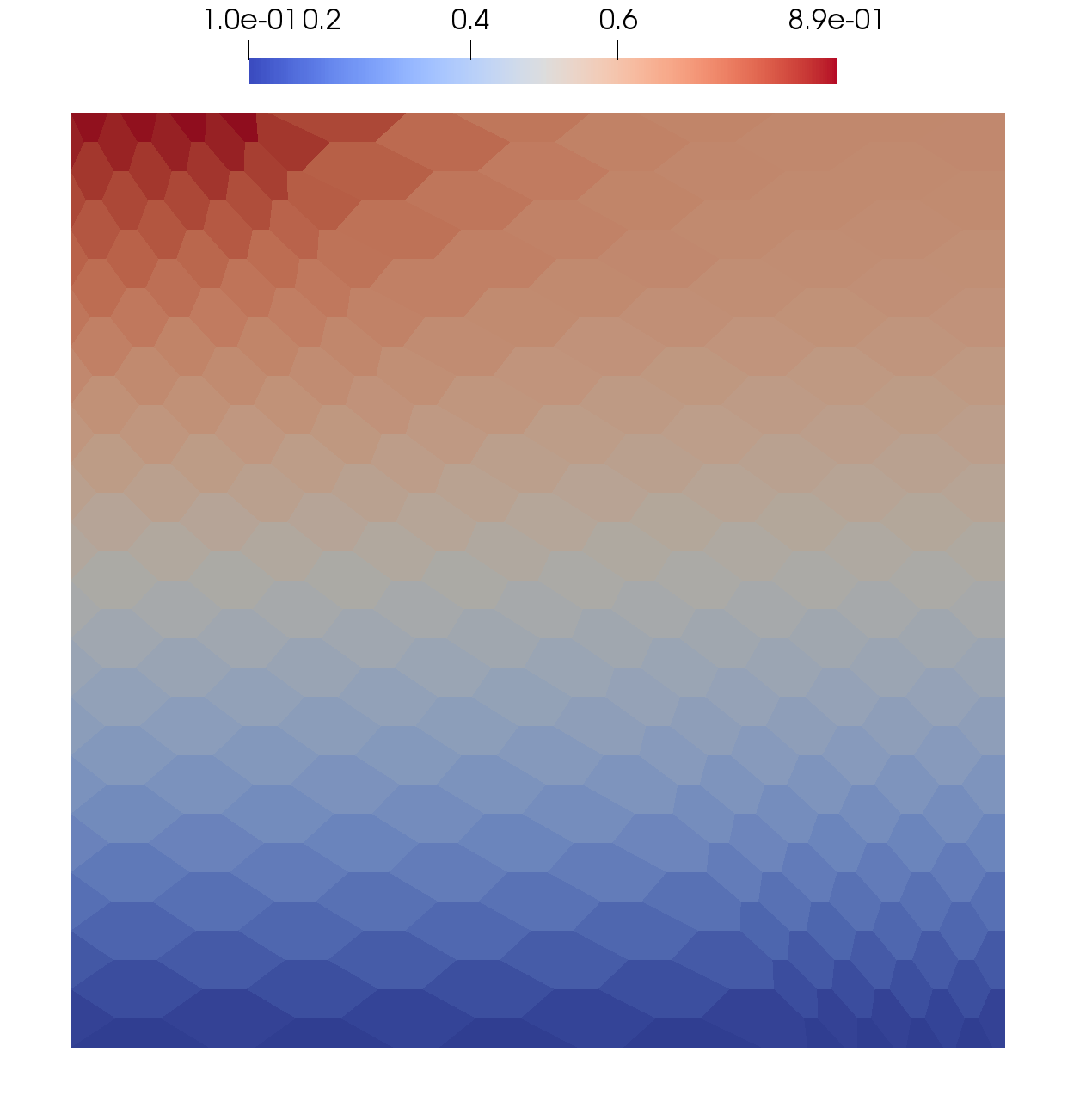}
\end{subfigure}\hfill
\begin{subfigure}{0.32\textwidth}  \caption{$P_\M^n$ with $t^n = 0.5$} 
  \includegraphics[width=\linewidth]{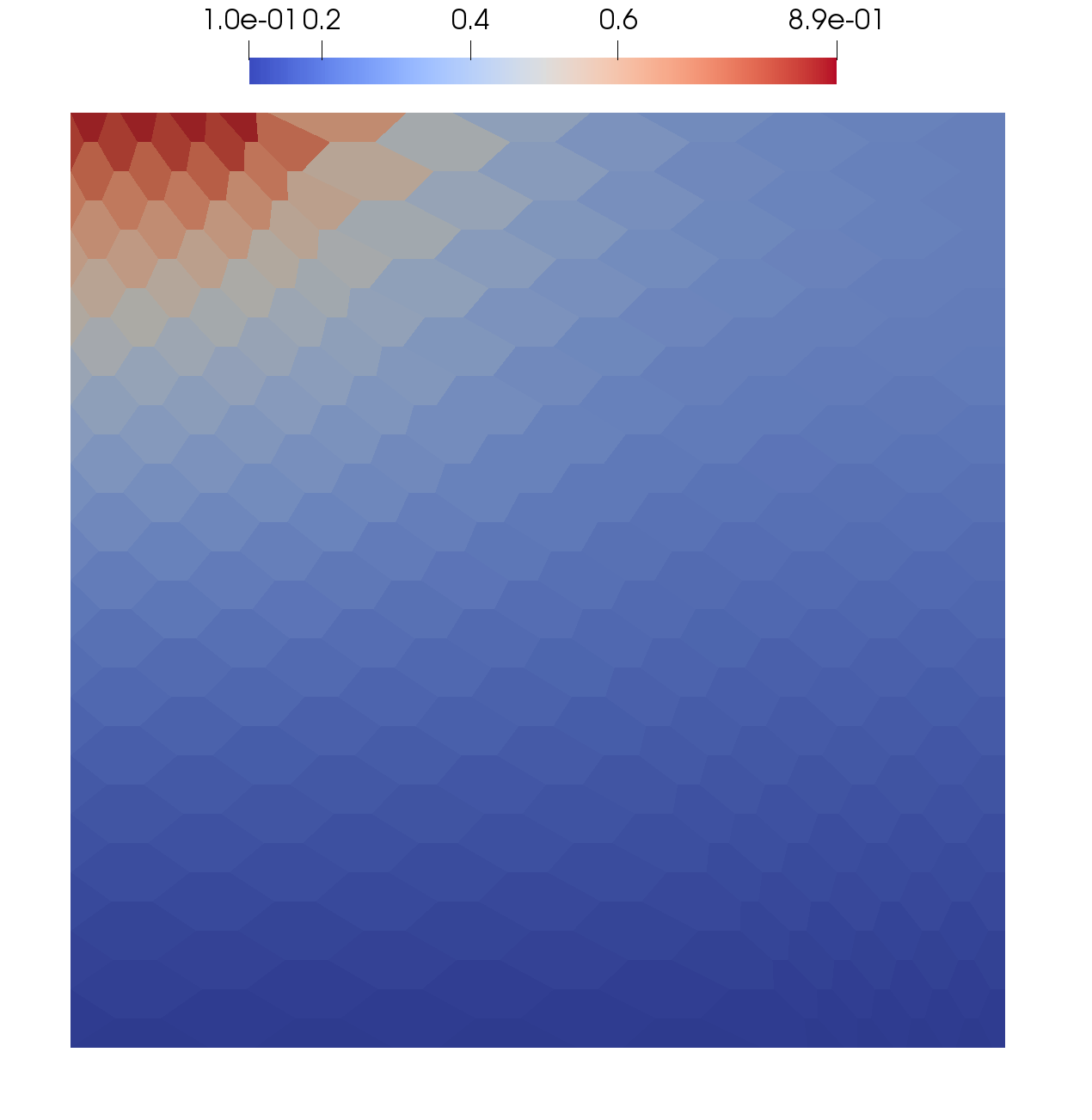}
\end{subfigure}
\caption{\textbf{Test-case \thenumtest}. Evolution of the discrete density of holes}
\label{fig:profile:Boltzmann}
\end{figure}
For the test-case \thenumtest, we use a set of data introduced in~\cite{CHFi:07}, with Boltzmann statistics ($h = \log$), no recombination ($r=0$), no magnetic field ($b=0$), $\lambda = 1$, boundary values $N^D_0 = 0.9$, $N^D_1 = 0.1$ and $\alpha_0 = \log \left ( N^D_0 \times N^D_1 \right) $. We perform a simulation on a tilted hexagonal-dominant mesh constituted of 280 cells, with $\Delta t = 0.1$.  
On Figure~\ref{fig:profile:Boltzmann}, we show the profiles of $P_\M^n$ for different values of $n$. The evolution is in accordance with results obtained in the TPFA context (see \cite{CHFi:07}, and notice that we use a different initial condition). Note that the discrete density remains positive, as expected. In fact, even on coarser meshes, the positivity of the densities is preserved (both for the cell and edge values). 

\addtocounter{numtest}{1}
Now, we want to assert the robustness of the scheme with respect to the $h$ function and the anisotropy.%
\begin{figure}[!ht]
\begin{subfigure}{0.32\textwidth}\caption{$P_\M^0$}  \label{subfig:profil:0}
  \includegraphics[width=\linewidth]{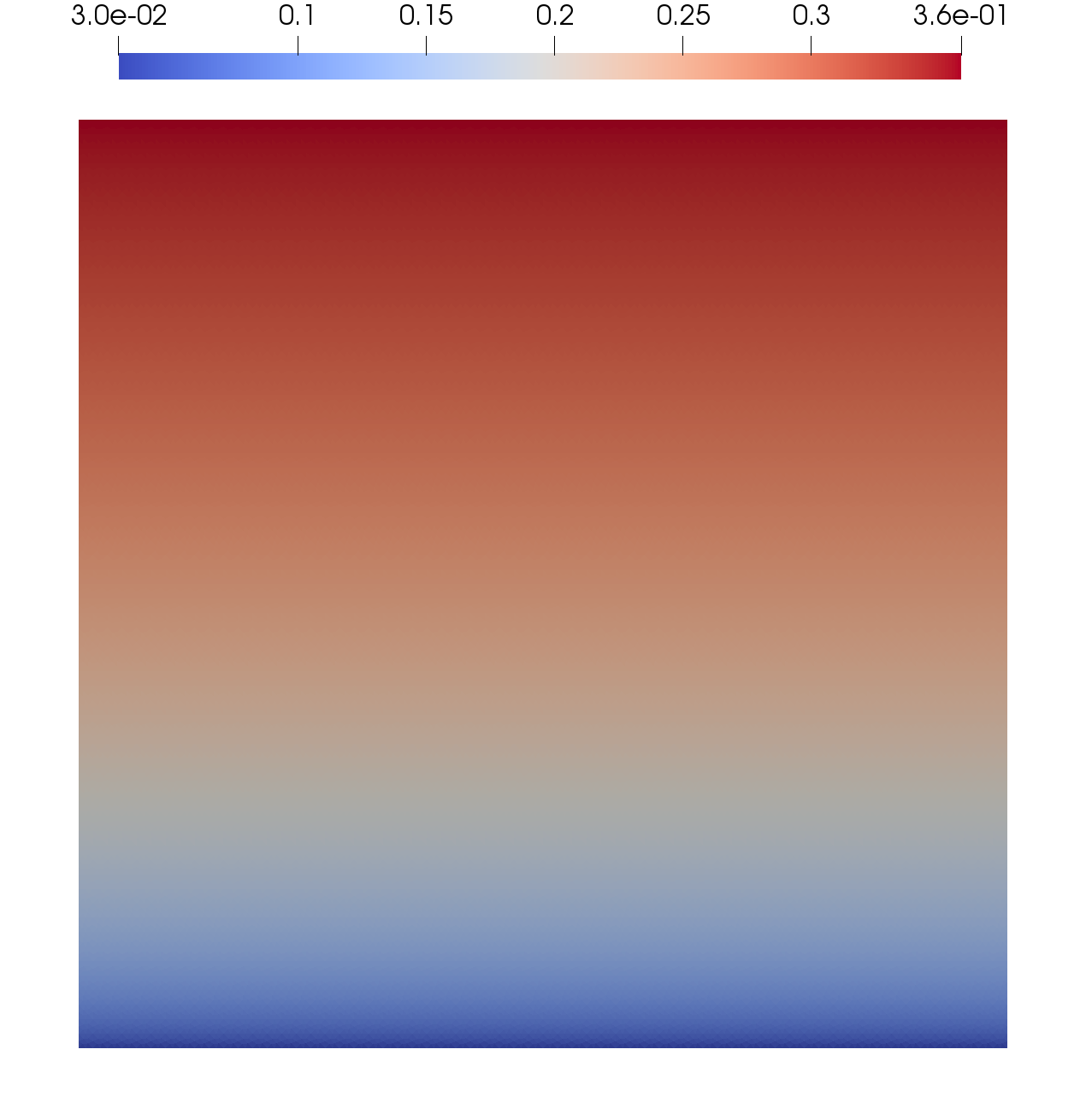}
\end{subfigure}\hfill
\begin{subfigure}{0.32\textwidth}  \caption{$P_\M^n$ with $t^n = 0.001$}  \label{subfig:profil:1}
  \includegraphics[width=\linewidth]{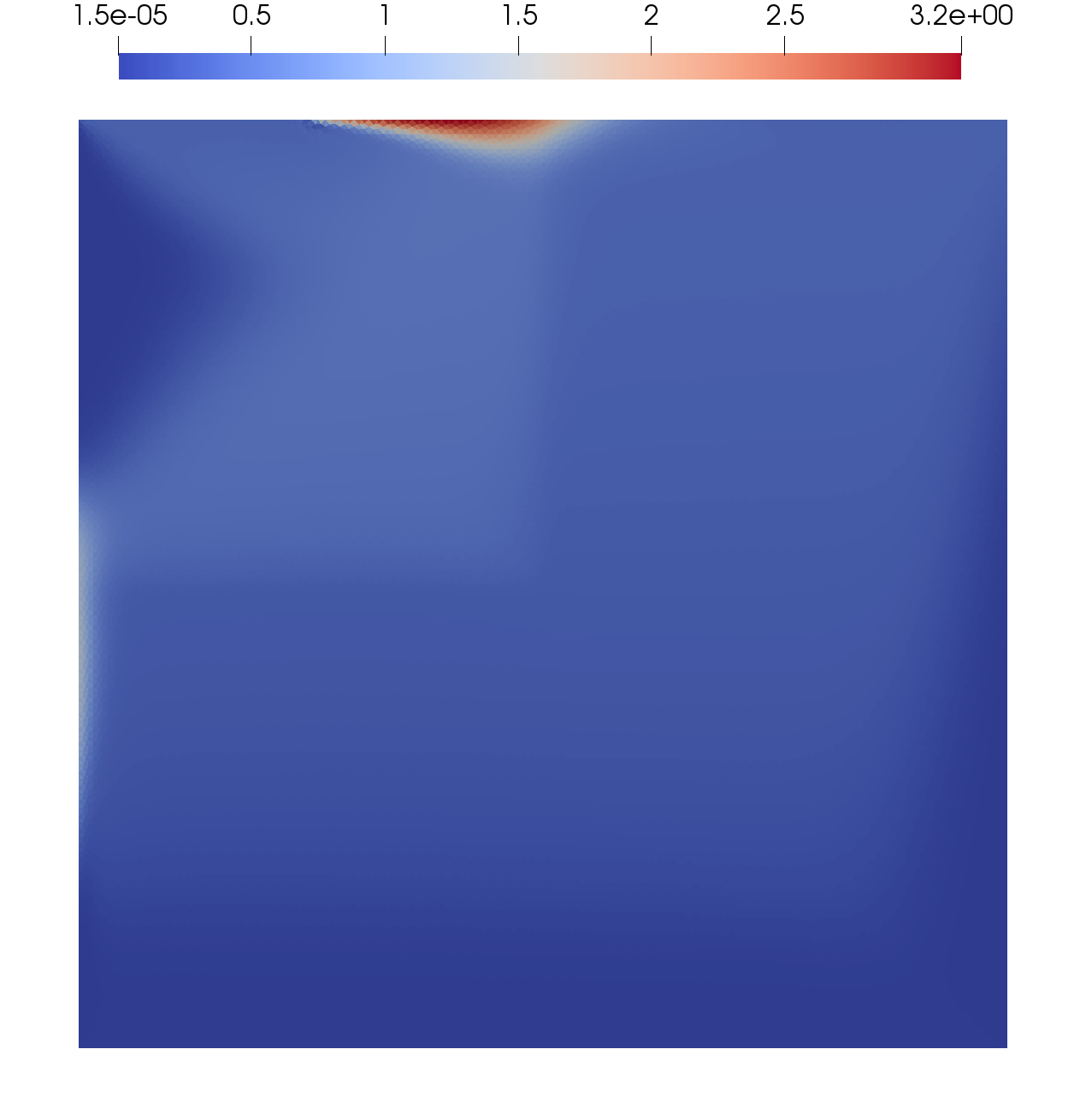}
\end{subfigure}\hfill
\begin{subfigure}{0.32\textwidth}  \caption{$P_\M^n$ with $t^n = 0.004$}  \label{subfig:profil:2}
  \includegraphics[width=\linewidth]{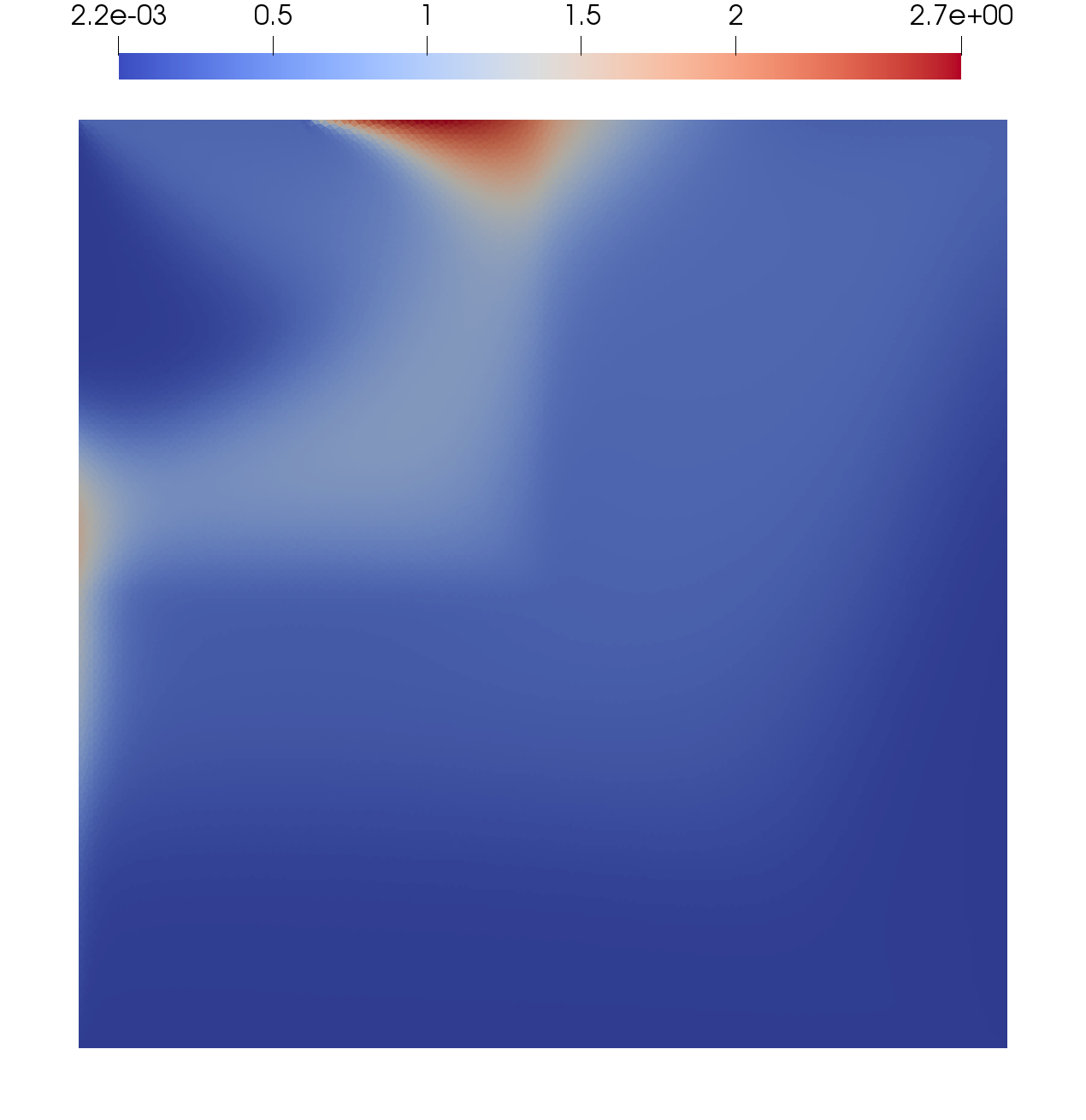}
\end{subfigure}
\begin{subfigure}{0.32\textwidth}  \caption{$P_\M^n$ with $t^n = 0.01$}  \label{subfig:profil:3}
  \includegraphics[width=\linewidth]{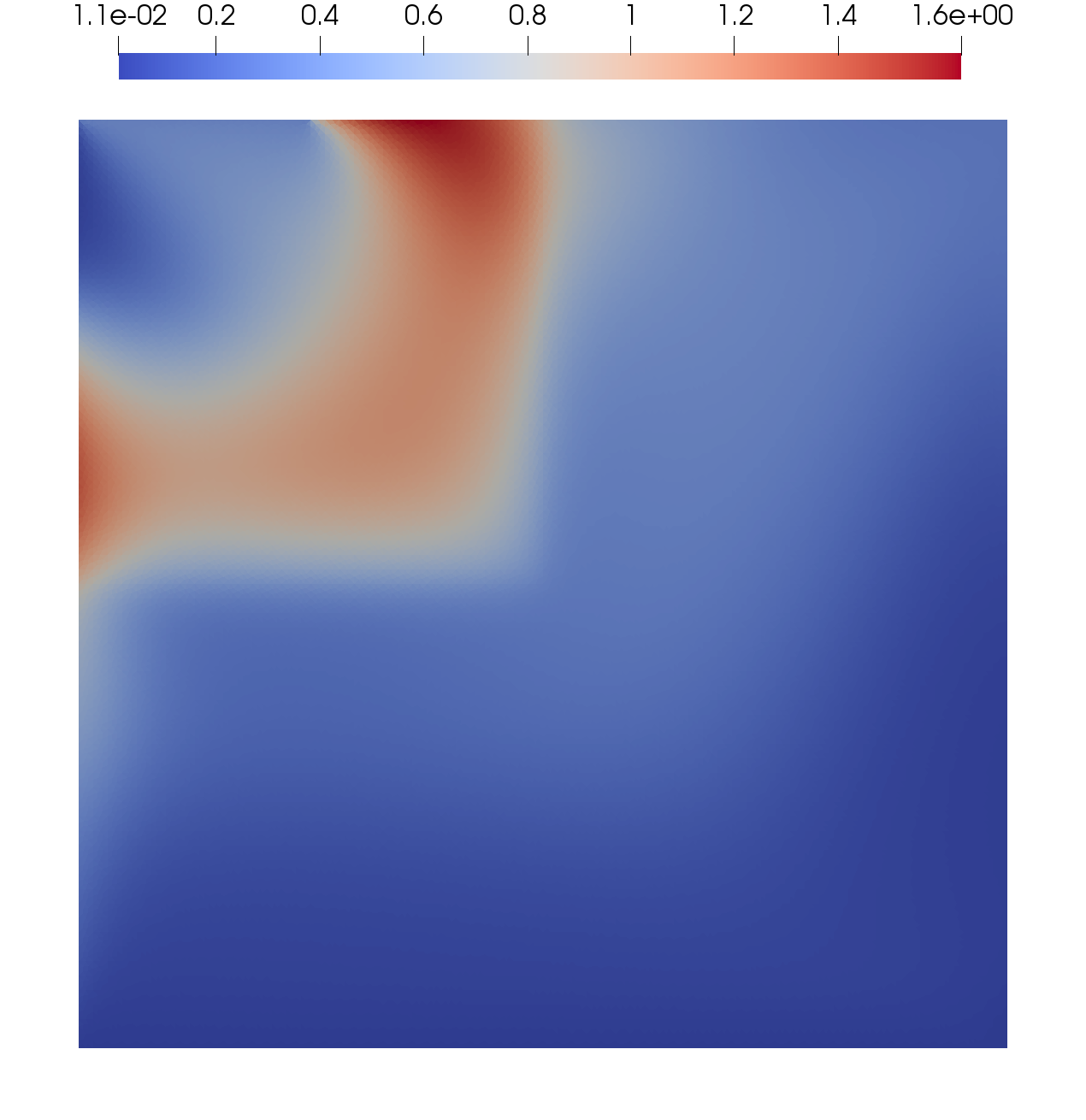}
\end{subfigure}\hfill
\begin{subfigure}{0.32\textwidth}  \caption{$P_\M^n$ with $t^n = 2$}  \label{subfig:profil:4}
  \includegraphics[width=\linewidth]{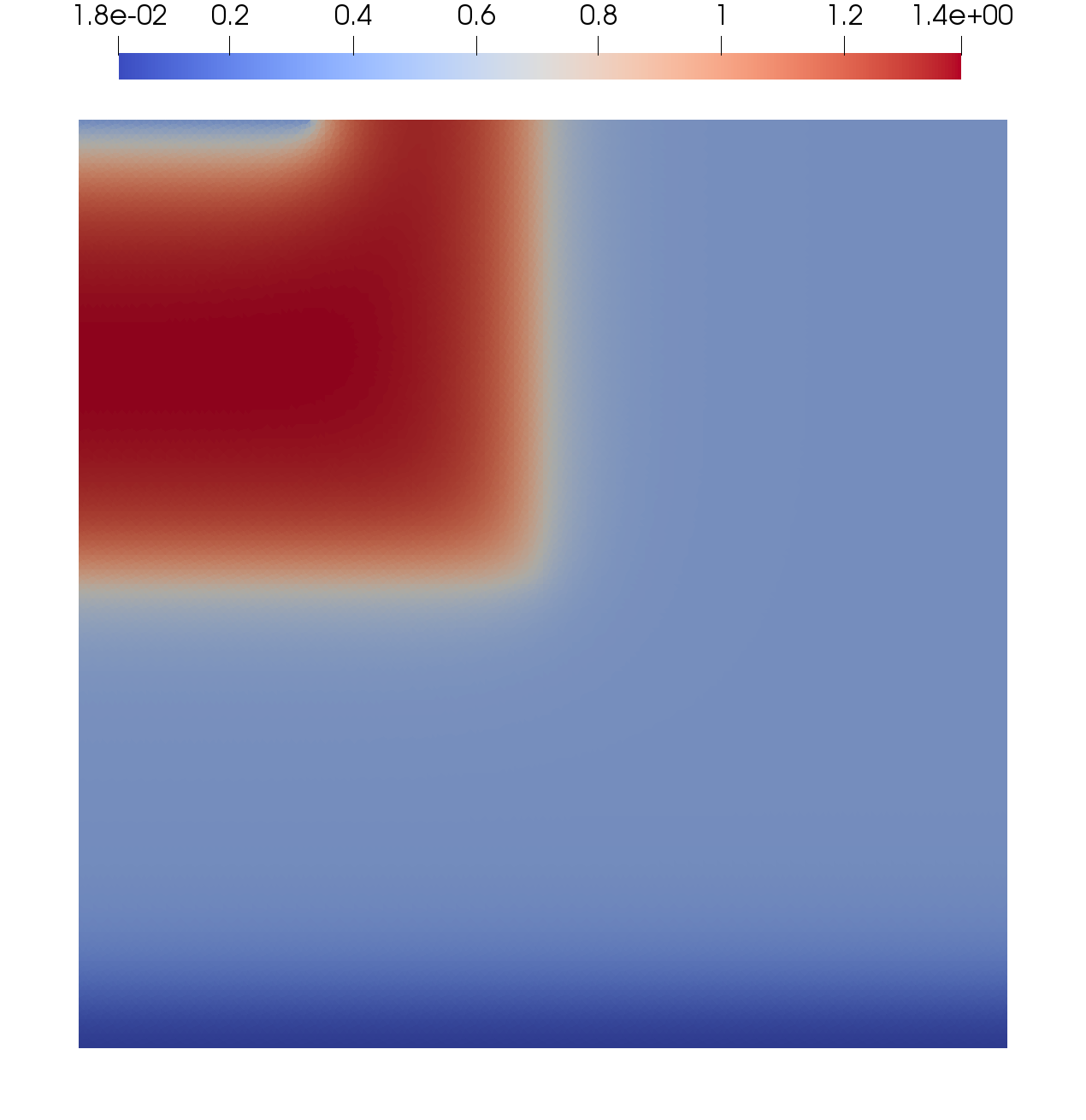}
\end{subfigure}\hfill
\begin{subfigure}{0.32\textwidth}  \caption{$P_\M^e$}  \label{subfig:profil:eq}
  \includegraphics[width=\linewidth]{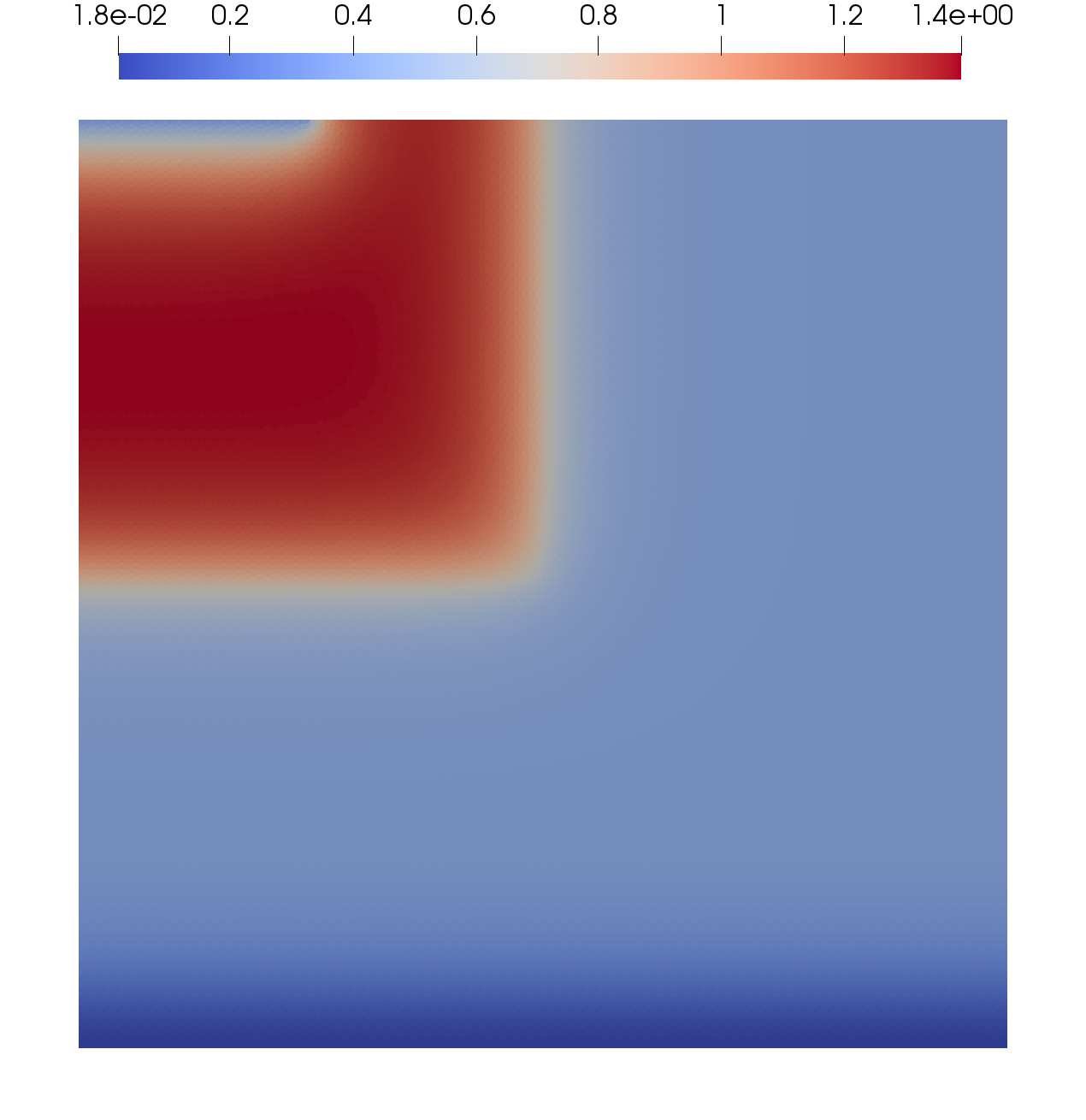}
\end{subfigure}
\caption{\textbf{Test-case \thenumtest} ($b=1$). Evolution of the discrete density of holes (note that the scale varies from a figure to the other)}
\label{fig:profile:Blakemore}
\end{figure}
We consider the test-case~\thenumtest, with Blakemore statistics ($h(s) = \log \left (\frac{s}{1-\gamma s } \right)$, $\gamma = 0.27$), a SRH recombination term ($r(N,P) = \frac{10}{1+N+P}$) and a strong magnetic field $b=1$. We also use a (realistic) small Debye length $\lambda = 0.05$. In order to check the fact that the discrete densities are $I_h$-valued (here, $\dmax = 1/\gamma \approx 3.7$), we consider boundary values close to the maximum admissible densities: $N_0^D = 3.5$ and $N_1^D= 1.5$. Since $r \neq 0$, we let $\alpha_0 = 0$. We perform a simulation on a refined triangular mesh (57 344 cells), with a time step $\Delta t = 0.1$.
We show the profile of $P_\M^n$ for different values of $n$ in Figure~\ref{fig:profile:Blakemore}. Notice that this test-case is subject to boundary layers (essentially because $\lambda$ is small and $b$ is big, see the discussion below and Table~\ref{tab:comp}), therefore the scheme performs numerous adaptations of the time step at the beginning of the simulation (the first admissible time step is $\delta t \approx 10^{-8}$). Moreover, one can see a rotation movement for the density of the holes, which is in accordance with the expected physical effect of the magnetic field. 
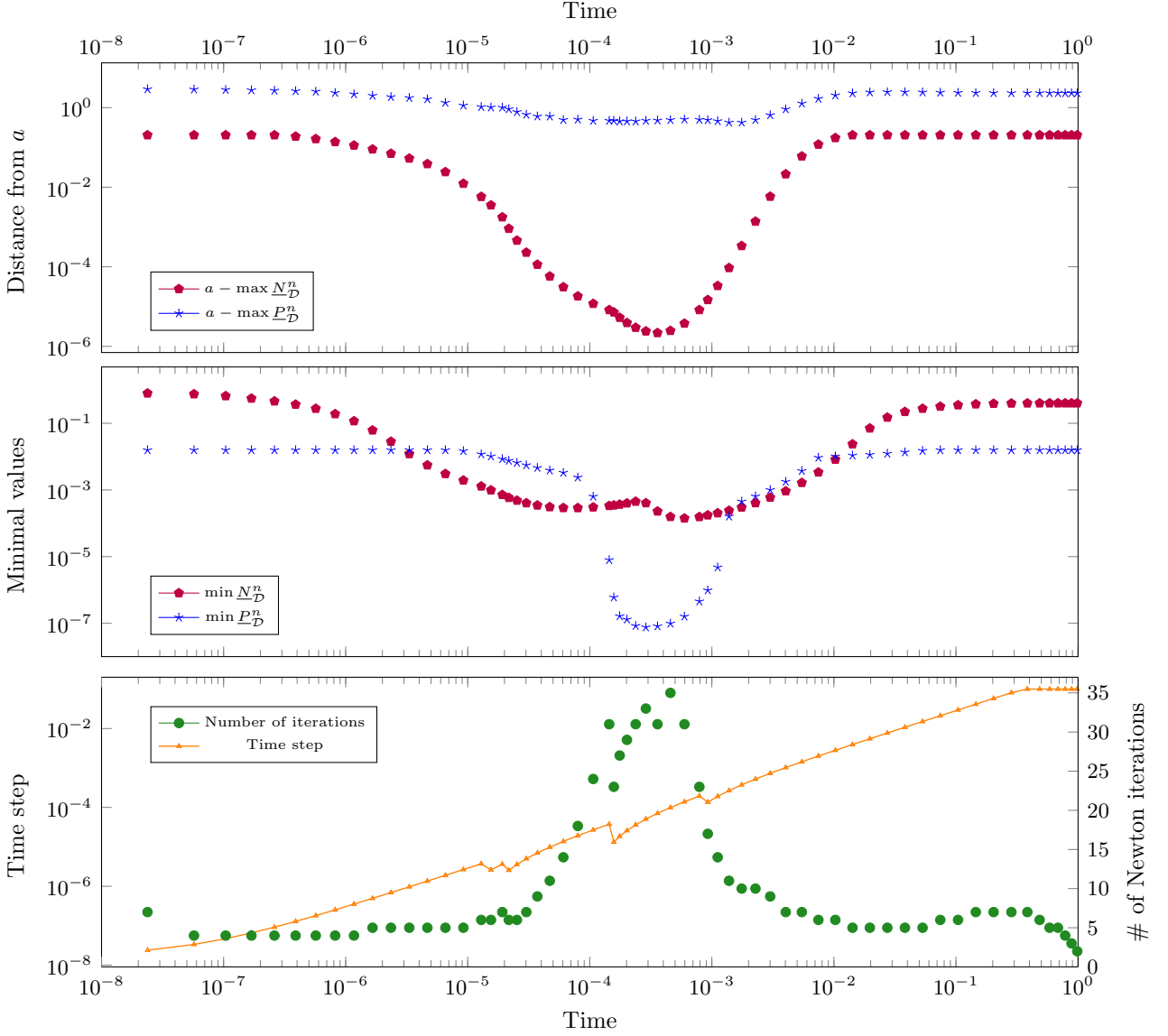
\begin{figure}[h]
\pgfplotsset{width=0.95\linewidth,height=0.35\linewidth,compat=newest}
\begin{minipage}[c]{1\linewidth}
\raggedright
	\begin{tikzpicture}[scale= 0.99] 
        \begin{loglogaxis}[
        		xmin=1.0e-8,
        		xmax=1,
        		ymin= 7.0e-7,
            	legend style = { 
              	at={(0.05,0.075)},
              	anchor = south west,
              	tick label style={font=\footnotesize},
              	legend columns= 1
            			},
            	ylabel=\small{Distance from $a$ \vphantom{Time step Minimal values}},
            	xlabel=\small{Time},
            	xticklabel pos=top,
          ]
          \addplot[mark=pentagon*, draw=none, purple]	table[x=Temps,y=Diff_max_N] {tps_full};
          \addplot[mark=star, draw=none, blue]	table[x=Temps,y=Diff_max_P] {tps_full};
          \legend{\tiny $a - \max \Nd_\D^n $,\tiny $a - \max \Pd_\D^n$, } 
	      \end{loglogaxis}
	\end{tikzpicture} 
%
      \begin{tikzpicture}[scale= 0.99]
        \begin{loglogaxis}[
        		xmin=1.0e-8,
        		xmax=1,        		
        		ymin= 1.0e-8,
        		max space between ticks=25,
            	legend style = { 
             	at={(0.05,0.075)},
              	anchor = south west,
              	tick label style={font=\footnotesize},
              	legend columns=1
            			},
            	ylabel=\small{Minimal values \vphantom{Time step Distance from $a$}},
            	xticklabels= \empty, 
          	]
          \addplot[mark=pentagon*, draw=none, purple] table[x=Temps,y=min_N] {tps_full};
          \addplot[mark=star, draw=none, blue]	table[x=Temps,y=min_P] {tps_full};
          \legend{\tiny $\min \Nd_\D^n$,\tiny $\min \Pd_\D^n$ } 
	      \end{loglogaxis}
      \end{tikzpicture} 
      \begin{tikzpicture}[scale= 0.99]  
		\pgfplotsset{
    			xmin=1.0e-8,
        		xmax=1,
        		legend style = { 
              	at={(0.05,0.9)},
              	anchor = north west,
              	tick label style={font=\footnotesize},
              	legend columns=1
            			}
			}
      	\begin{loglogaxis}[
        		ymax=0.2,
        		ymin=9e-9,
        		axis y line*=left,
			axis x line=none,
            	ylabel=\small{Time step \vphantom{Distance from $a$ Minimal values}},
          	]
          	\addplot[mark=triangle, mark size=1.,  line width=0.5pt, draw=orange] table[x=Temps,y=time_step] {tps_full};
          	\label{step}
	   	\end{loglogaxis}
        	\begin{semilogxaxis}[
        		ymin = 0.,
        		axis y line*=right,
        		ymax=37,
        		max space between ticks= 20,
            	ylabel=\small{\#  of Newton iterations},
            	xlabel=\small{Time}
          	]
          	
          	\addplot[mark=*, draw=none, forestgreen] table[x=Temps,y=Nb_iter] {tps_full};
          	\addlegendimage{/pgfplots/refstyle=step};
          	\addlegendentry{\tiny Number of iterations};
          	\addlegendentry{\tiny Time step}        	
	      	\end{semilogxaxis}
      \end{tikzpicture} 
\end{minipage}
	\caption{\textbf{Test-case \thenumtest} ($b=1$). Evolution of the discrete extremal values, time step and cost}
	\label{fig:bounds}
\end{figure}
\newline 
To give more quantitative informations, we show in Figure~\ref{fig:bounds} the evolution of the minimal and maximal values, along with the time step and the number of Newton's iterations needed to compute the solutions at a given time. First, note that we display each of the steps (from $t^n$ to $t^{n+1}$) used to compute the solutions on the time interval $[0,1]$, there are 63 of them for this simulation. 
The extremal values account for the cell and edge unknowns: we let
\[
	\min \Nd_\D^n =  \min \left ( \min_{K \in \M} N_K^n , \min_{\s \in \E} N_\s^n \right )
		\text{ and }
		\max \Nd_\D^n =  \max \left ( \max_{K \in \M} N_K^n , \max_{\s \in \E} N_\s^n \right ), 
\]
as well as analogous definitions for the holes density. In order to observe the extremal values on relevant scales, we look at the minimum and the distance between $a$ (the maximal density allowed by the model) and the maximal densities computed, and print them on a log scale.
As expected from Theorem~\ref{th:DD}, the values of the discrete densities stay in $I_h = ]0,\dmax[$. In particular, our scheme is not subject to the same lack of positivity as the scheme of~\cite{GaGa:96} when the magnetic field is intense.
Moreover, it is remarkable to note that the scheme seems very robust: the densities are close to the limit values, at a distance reaching 7.56e-8 in the most difficult situation.
Note that this value mean that in practice, we do not perform the projection step described in Section~\ref{sec:Newton} to initialise the Newton's methods, since the projection threshold $\epsilon = 10^{-9}$ is smaller than 7.56e-8.
Using larger values for $\epsilon$ could perhaps improve the convergence of Newton's methods, but the impact of such a modification has not been investigated here.
As previously announced, one can see that the first effective time step needed to compute $(\Nd^{1}_\D,\Pd^{1}_\D, \phid^{1}_\D)$ is relatively small: $0.1 \times 2^{-22} \approx \text{2.38e-8}$. It means that the time step adaptation procedure needed to perform 22 time step reductions before managing to compute a solution. On the other hand, we can see that after this initial reduction, there are only five others time step reductions (one around 1.5e-5, one at 2e-5, two around 1.5e-4, and a last one at 9e-4). At the end of the simulation, around $t = 0.3$, the time step reach its maximal value $\Delta t = 0.1$ and there is no more time step adaptation.
On the third graph, we show the number of Newton iterations needed to compute the discrete solution from one time to another.
Note that we do not take into account the iterations used in non-convergent Newton's methods (i.e. methods that leads to a time step reduction).
For all the time step reductions of this simulation, the Newton's methods do not converge because at least one of the computed discrete densities was not $I_h$-valued.
On average, the convergent Newton's methods converge in 10.5 iterations.
Finally, we can see on Figure~\ref{fig:bounds} a clear correlation between the extremal values reached by the discrete densities and the number of Newton iterations needed to compute the solution.
This can be explained by at least one fact: since $h$ is singular in $0$ and $a$, the values of its derivative near these limit values blow up. Therefore, the Jacobian matrix used in the Newton's method tends to be ill-conditioned, which induces numerical instabilities.
We can give a last remark on this test-case: the most extreme values observed in Figure~\ref{fig:bounds} are located near to the boundary and appear at the beginning of the simulation (see for example the minimal values on Figure~\ref{subfig:profil:1}, on the left and right sides of the square). Hence, the difficulties are essentially due to the presence of the (strong) magnetic field, which induces rotation of the charges and creates boundary layers. 
%
%
\begin{table}[!ht]
\center
\renewcommand{\arraystretch}{1.05}
\begin{tabularx}{0.75\textwidth}{|X|c|c|c|}
\hline 
Magnetic field intensity & $b=0$ & $b=0.5$ &$b=1$ \\ 
\hline 
$\displaystyle \min  \left \{  \min \Nd_\D^n \mid 0 \leq t^n \leq 1 \right \}$  & 3.20e-1 & 5.36e-3  & 1.41e-4  \\ 
\hline 
$\displaystyle \min  \left \{  \min \Pd_\D^n \mid 0 \leq t^n \leq 1 \right \}$  & 1.56e-2 & 2.53e-3  & 7.56e-8  \\ 
\hline 
$\displaystyle \min  \left \{ a - \max \Nd_\D^n \mid 0 \leq t^n \leq 1 \right \}$ & 2.04e-1 & 4.17e-3 & 2.18e-6 \\ 
\hline 
$\displaystyle \min  \left \{ a - \max \Pd_\D^n \mid 0 \leq t^n \leq 1 \right \}$ & 2.31 & 9.39e-1 &4.23e-2 \\ 
\hline 
Number of steps & 18 & 40 & 63 \\ 
\hline 
Minimal time step & 3.13e-3 & 1.53e-6 & 2.38e-8 \\ 
\hline 
Number of initial time step reductions & 5 & 16 &  22 \\ 
\hline 
Total number of time step reductions & 5 & 16 & 27 \\ 
\hline 
Maximal number of Newton iterations & 7  & 7 & 35 \\ 
\hline 
Average number of Newton iterations & 5.05 & 5.45 & 10.54 \\ 
\hline 
Total cost & 96 & 234 & 692 \\ 
\hline 
\end{tabularx} 
\caption{\textbf{Test-case \thenumtest}. Comparison of the extremal values and costs for different magnetic fields}
\label{tab:comp}
\end{table}

To confirm this statement, we perform simulations with the same parameters (still on the time interval $[0,1]$) except that we consider situations with smaller magnetic field intensities ($b=0$ and $b=0.5$). The results are presented in Table~\ref{tab:comp}. As expected, it seems that the extremal values are strongly related to the intensity of the magnetic field. It follows that the number of Newton iterations, and hence the global computational cost of the simulation increase with the intensity of the magnetic field. We can also notice that the time step reductions occur only at the first time step for the moderate magnetic fields ($b=0$ and $b=0.5$): the computational difficulties lie in the boundary layers appearing at small times.
As previously, note that the ``number of iterations'' mentioned on the table do not take into account the iterations of non-convergent Newton's methods.
In the last line, we give the ``Total cost'' of the simulation, that is to say the number of linear systems solved during the simulation, including the iterations of non-convergent Newton's methods.
Taking into account every iterations performed, the simulation with $b=1$ is basically 7 times more time consuming than the one without magnetic field.

\addtocounter{numtest}{1}
In order to assert the robustness of the method with respect to the mesh, we consider the test case \thenumtest, which is characterised by the same physical parameters as the previous one: $h(s) = \log \left (\frac{s}{1-\gamma s } \right)$, $\gamma = 0.27$), $r(N,P) = \frac{10}{1+N+P}$, $b=1$, $\lambda = 0.05$, $N_0^D = 3.5$, $N_1^D= 1.5$ and $\alpha_0 = 0$.
Contrarily to the previous test, we perform a simulation on a tilted hexagonal mesh, constituted of 4192 cells, with a time step $\Delta t = 0.1$. One has to notice that the spatial mesh used here is a ``general polygonal" one, in the sense that it is not a admissible mesh for the TPFA scheme. Moreover, the geometry of this mesh is not well-suited with respect to the geometry of the device, since the junction of the PN-diode (which corresponds to the discontinuity of the doping profile $C$) crosses some cells.
\begin{figure}[h]
\begin{subfigure}{0.32\textwidth}\caption{$P_\M^n$ with $t^n = 0.00082$} 
  \includegraphics[width=\linewidth]{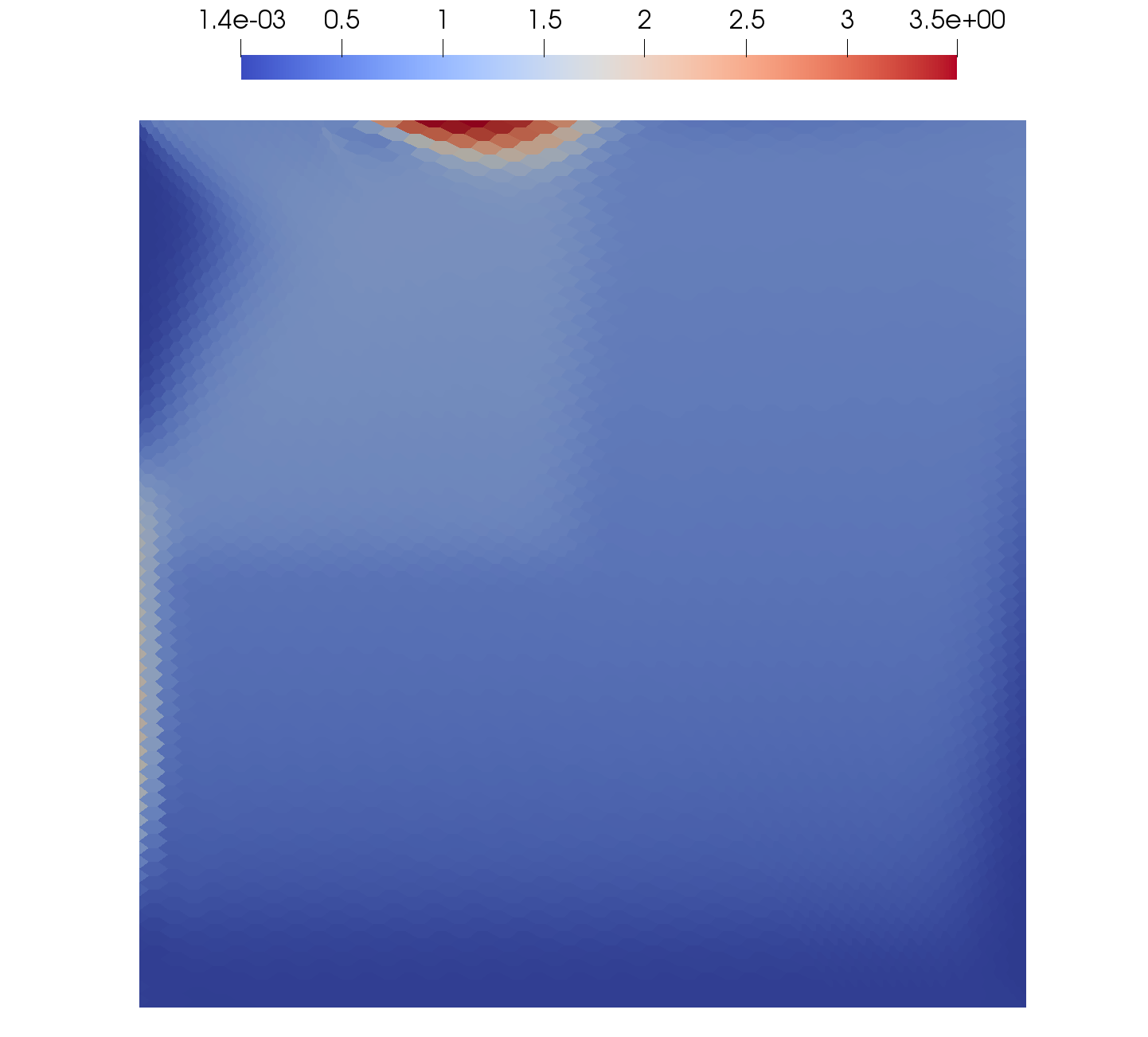}
\end{subfigure}\hfill
\begin{subfigure}{0.32\textwidth}  \caption{$P_\M^n$ with $t^n = 0.0044$} 
  \includegraphics[width=\linewidth]{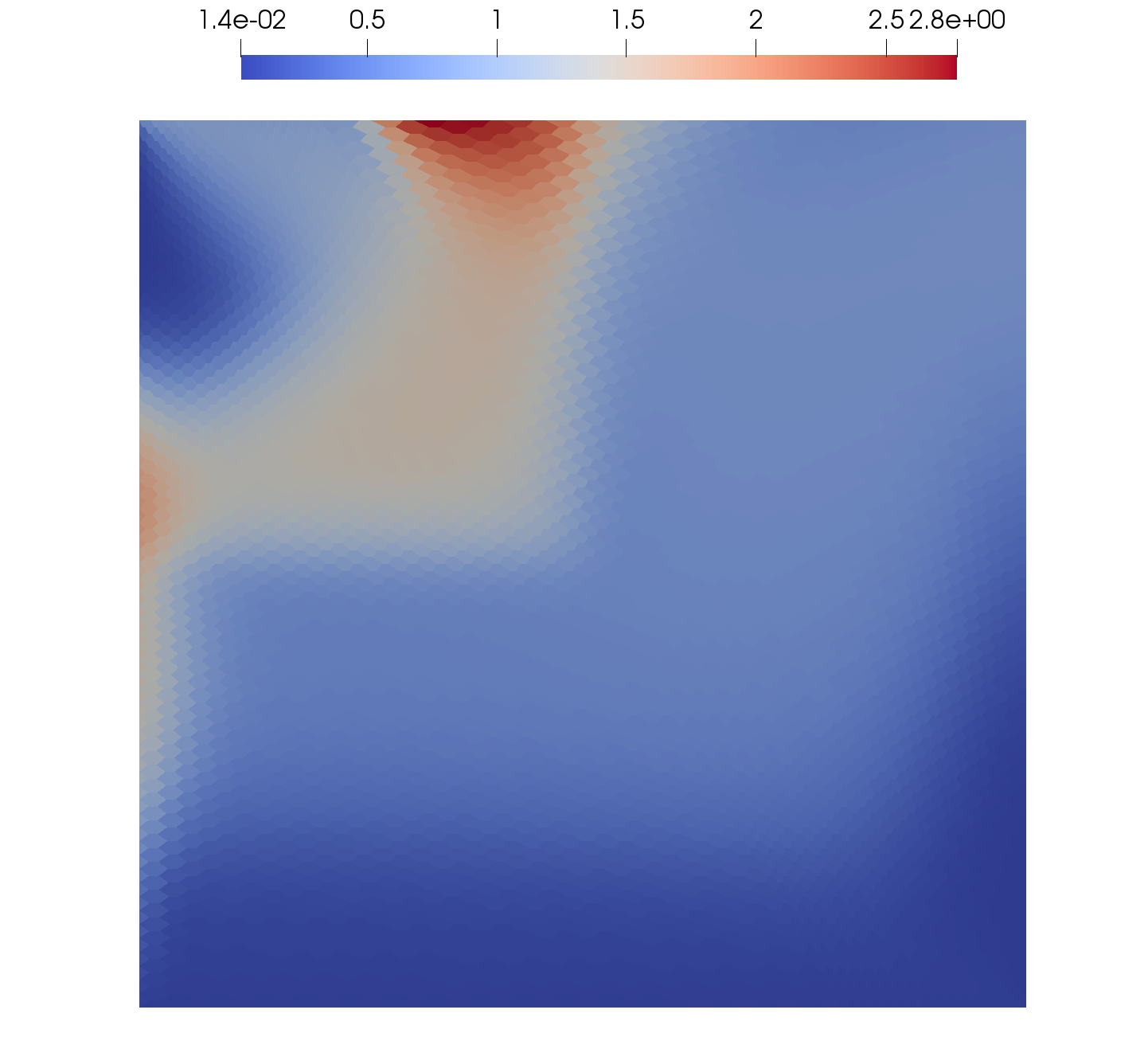}
\end{subfigure}\hfill
\begin{subfigure}{0.32\textwidth}  \caption{$P_\M^n$ with $t^n = 0.0087$} 
  \includegraphics[width=\linewidth]{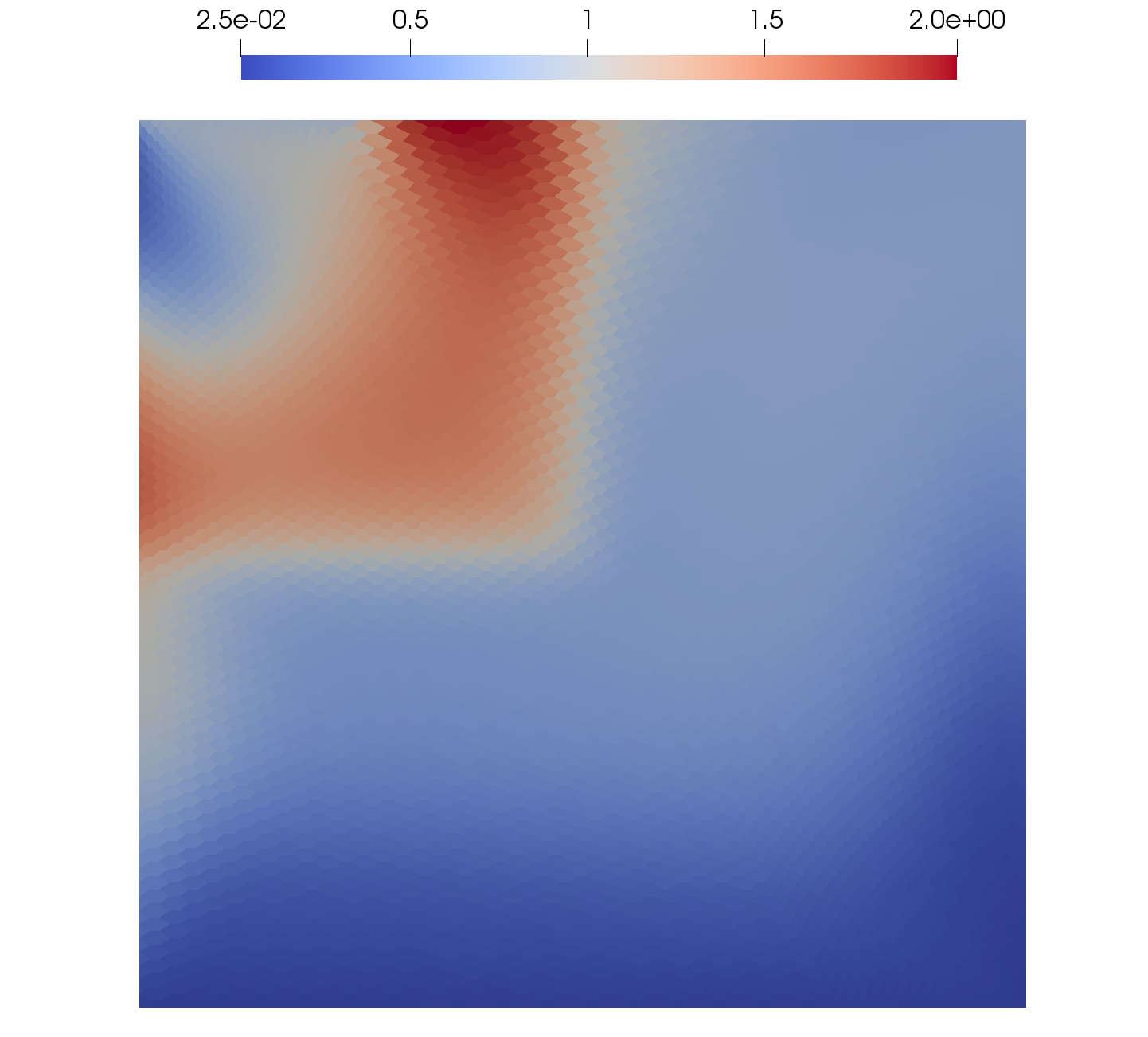}
\end{subfigure}
\caption{\textbf{Test-case \thenumtest}. Evolution of the discrete density of holes, on a tilted hexagonal mesh.}
\label{fig:profile:Blakemore:hexa}
\end{figure}
On Figure \ref{fig:profile:Blakemore:hexa}, we show the profile of the density of holes $P_\M$ computed at different times. These profiles are to be compared with these of the Figure~\ref{fig:profile:Blakemore}: even if the tilted hexagonal does not fit the geometry of the problem perfectly and is much coarser than the triangular mesh used for the Figure~\ref{fig:profile:Blakemore}, the profiles look quite the same.\newline
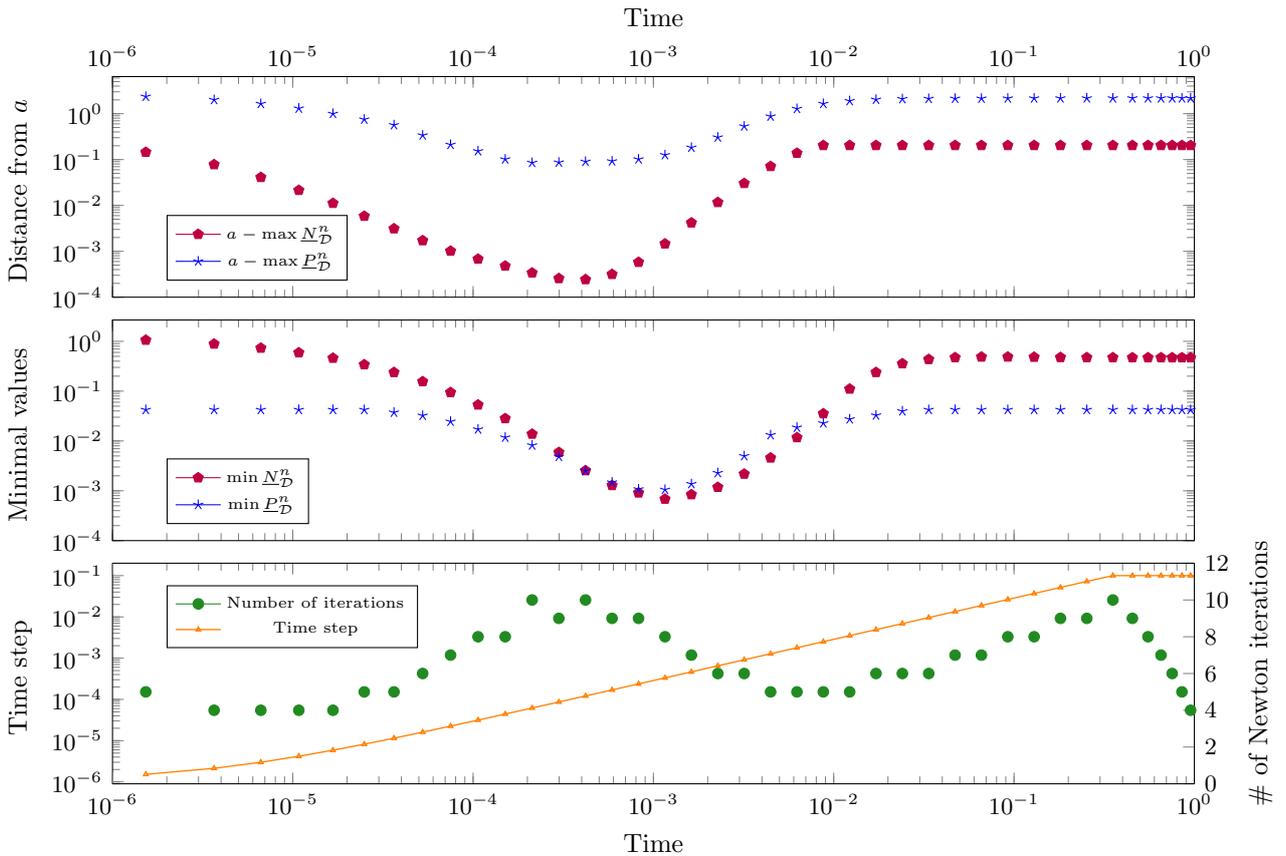
\begin{figure}[!h]
\pgfplotsset{width=0.95\linewidth,height=0.27\linewidth,compat=newest}
\begin{minipage}[c]{1\linewidth}
\raggedright
	\begin{tikzpicture}[scale= 0.99] 
        \begin{loglogaxis}[
        		xmin=1.0e-6,
        		xmax=1,
        		ymin= 1.0e-4,        		
        		max space between ticks=25,
            	legend style = { 
              	at={(0.05,0.075)},
              	anchor = south west,
              	tick label style={font=\footnotesize},
              	legend columns= 1
            			},
            	ylabel=\small{Distance from $a$ \vphantom{Time step Minimal values}},
            	xlabel=\small{Time},
            	xticklabel pos=top,
          ]
          \addplot[mark=pentagon*, draw=none, purple]	table[x=Temps,y=Diff_max_N] {tps_full_hexa};
          \addplot[mark=star, draw=none, blue]	table[x=Temps,y=Diff_max_P] {tps_full_hexa};
          \legend{\tiny $a - \max \Nd_\D^n $,\tiny $a - \max \Pd_\D^n$, } 
	      \end{loglogaxis}
	\end{tikzpicture} 
\vspace{-12pt}
      \begin{tikzpicture}[scale= 0.99]
        \begin{loglogaxis}[
        		xmin=1.0e-6,
        		xmax=1,        		
        		ymin= 1.0e-4,
        		max space between ticks=25,
            	legend style = { 
             	at={(0.05,0.075)},
              	anchor = south west,
              	tick label style={font=\footnotesize},
              	legend columns=1
            			},
            	ylabel=\small{Minimal values \vphantom{Time step Distance from $a$}},
            	xticklabels= \empty, 
          	]
          \addplot[mark=pentagon*, draw=none, purple] table[x=Temps,y=min_N] {tps_full_hexa};
          \addplot[mark=star, draw=none, blue]	table[x=Temps,y=min_P] {tps_full_hexa};
          \legend{\tiny $\min \Nd_\D^n$,\tiny $\min \Pd_\D^n$ } 
	      \end{loglogaxis}
      \end{tikzpicture} 
      \begin{tikzpicture}[scale= 0.99]  
		\pgfplotsset{
    			xmin=1.0e-6,
        		xmax=1,
        		legend style = { 
              	at={(0.05,0.9)},
              	anchor = north west,
              	tick label style={font=\footnotesize},
              	legend columns=1
            			}
			}
      	\begin{loglogaxis}[
        		ymax=0.2,
        		ymin=9e-7,
        		max space between ticks= 20,
        		axis y line*=left,
			axis x line=none,
            	ylabel=\small{Time step \vphantom{Distance from $a$ Minimal values}},
          	]
          	\addplot[mark=triangle, mark size=1.,  line width=0.5pt, draw=orange] table[x=Temps,y=time_step] {tps_full_hexa};
          	\label{stephexa}
	   	\end{loglogaxis}
        	\begin{semilogxaxis}[
        		ymin = 0.,
        		axis y line*=right,
        		ymax=12,
        		max space between ticks= 20,
            	ylabel=\small{\#  of Newton iterations},
            	xlabel=\small{Time}
          	]
          	
          	\addplot[mark=*, draw=none, forestgreen] table[x=Temps,y=Nb_iter] {tps_full_hexa};
          	\addlegendimage{/pgfplots/refstyle=step};
          	\addlegendentry{\tiny Number of iterations};
          	\addlegendentry{\tiny Time step}        	
	      	\end{semilogxaxis}
      \end{tikzpicture} 
\end{minipage}
	\caption{\textbf{Test-case \thenumtest} ($b=1$, tilted hexagonal mesh). Evolution of the discrete extremal values, time step and cost.}
	\label{fig:bounds:hexa}
\end{figure}
We give quantitative information for this test case on Figure~\ref{fig:bounds:hexa}, which is the counterpart of Figure~\ref{fig:bounds}: we show the evolution of the bounds of the discrete densities, as well as the time step and the number of Newton’s iterations used for a given time step.
First, note that the first admissible time step ($\delta t \approx 10^{-6}$) is bigger than the one on the refined triangular mesh. Moreover, the extremal values are much farther from the bounds of $I_h$ than on the triangular mesh.
This behaviour can be explained by the fineness of the triangular mesh, which enable the scheme to capture in a very accurate way the boundary layers: the continuous solution take very small/big values near to the boundary of the domain, but if one average these values on relatively big cells (such as the  cells of the tilted hexagonal mesh), then the mean values are not so extreme.  
We can also remark that the number of iterations needed to compute one time step peaks around $t \approx 3.10^{-4}$, as in Figure~\ref{fig:bounds}: it corresponds to the time when the densities are closest to the bounds of $I_h$. Such fact enforces the hypothesis formulated before: the values of the densities have an important impact on the Jacobian entries and hence on the convergence of the Newton's method.

Overall, the behaviour of the scheme does not depend on the geometry of the mesh, and the fact that the PN-junction crosses some cells does not have a noticeable impact.

\subsection{Long-time behaviour of the discrete solutions} \label{sec:num:longtime}\addtocounter{numtest}{1}
We are now interested in the long-time behaviour of the solutions computed with the scheme.
For the test-case~\thenumtest, we consider a test-case from~\cite{BCCH:17} with Boltzmann statistics ($h = \log$), no recombination ($r=0$), no magnetic field ($b=0$), $\lambda = 1$ and boundary values $N^D_0 = \e$, $N^D_1 = 1$ and $\alpha_0 = 1$.
We are interested in the the evolution of the discrete relative entropy and the $L^2$ distance to the equilibrium (namely, $\sqrt{\|N^n_\M -N^e_\M  \|_{L^2(\Omega)}^2 + \| P^n_\M -P^e_\M\|_{L^2(\Omega)}^2 + \| \phi^n_\M -\phi^e_\M\|_{L^2(\Omega)}^2 }$) with respect to time.

In Figure~\ref{fig:tps:varsize}, we show the evolution of these quantities for simulations performed on the family of triangular meshes (the coarsest mesh has 56 cells, and a size $h_0$), with $\Delta t = 0.01$.
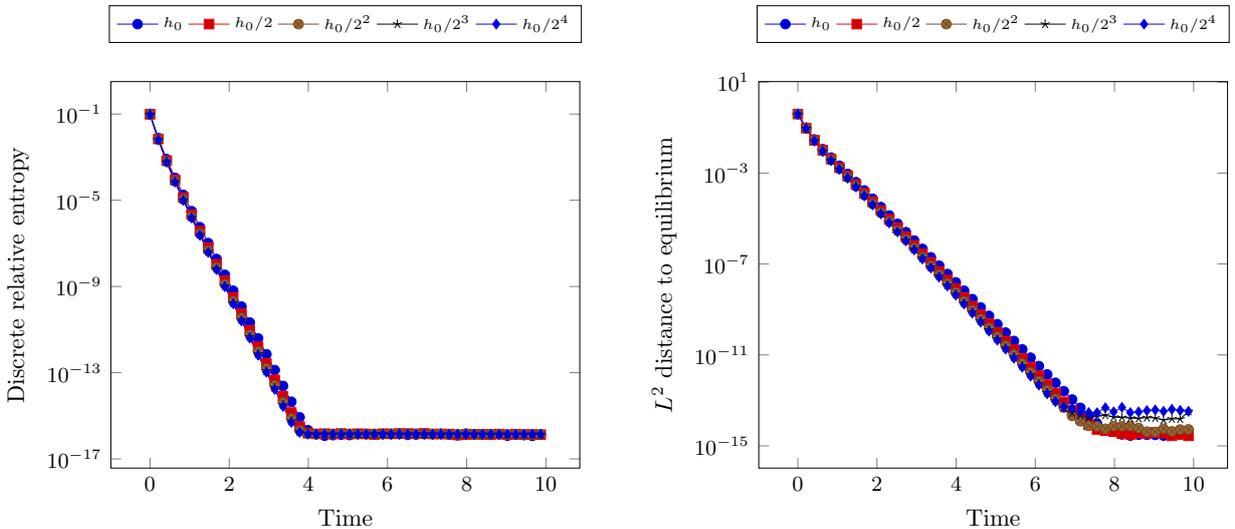
\begin{figure}[ht]
\begin{minipage}[c]{0.5\linewidth}
\begin{tikzpicture}[scale= 0.9]
        \begin{semilogyaxis}[
            	 legend style = { 
              at={(0.5,1.1)},
              anchor = south,
              tick label style={font=\footnotesize},
              legend columns=-1
            },ylabel=\small{Discrete relative entropy},xlabel=\small{Time}
          ]
          \addplot table[x=Temps,y=Entro] {graph_data/mesh_var(dt=0.01)/mesh1_1/tps_impli};
          \addplot table[x=Temps,y=Entro] {graph_data/mesh_var(dt=0.01)/mesh1_2/tps_impli};
          \addplot table[x=Temps,y=Entro] {graph_data/mesh_var(dt=0.01)/mesh1_3/tps_impli};
          \addplot table[x=Temps,y=Entro] {graph_data/mesh_var(dt=0.01)/mesh1_4/tps_impli};
          \addplot table[x=Temps,y=Entro] {graph_data/mesh_var(dt=0.01)/mesh1_5/tps_impli};
          \legend{\tiny $h_0$, \tiny $h_0/2$ ,\tiny $h_0/2^2$,  \tiny $h_0/2^3$, \tiny $h_0/2^4$};
	      \end{semilogyaxis}
      \end{tikzpicture}    
\end{minipage}
\begin{minipage}[c]{.5\linewidth}
\begin{tikzpicture}[scale= 0.9]
        \begin{semilogyaxis}[
            legend style = { 
              at={(0.5,1.1)},
              anchor = south,
              tick label style={font=\footnotesize},
              legend columns=-1
            },ylabel=\small{$L^2$ distance to equilibrium},xlabel=\small{Time}
          ] 
          \addplot table[x=Temps,y expr=sqrt(\thisrow{Diff_eq_L2})] {graph_data/mesh_var(dt=0.01)/mesh1_1/tps_impli};
          \addplot table[x=Temps,y expr=sqrt(\thisrow{Diff_eq_L2})] {graph_data/mesh_var(dt=0.01)/mesh1_2/tps_impli};
          \addplot table[x=Temps,y expr=sqrt(\thisrow{Diff_eq_L2})] {graph_data/mesh_var(dt=0.01)/mesh1_3/tps_impli};
          \addplot table[x=Temps,y expr=sqrt(\thisrow{Diff_eq_L2})] {graph_data/mesh_var(dt=0.01)/mesh1_4/tps_impli};
          \addplot table[x=Temps,y expr=sqrt(\thisrow{Diff_eq_L2})] {graph_data/mesh_var(dt=0.01)/mesh1_5/tps_impli};					  \legend{\tiny $h_0$, \tiny $h_0/2$ ,\tiny $h_0/2^2$,  \tiny $h_0/2^3$, \tiny $h_0/2^4$};
        \end{semilogyaxis}
      \end{tikzpicture}
\end{minipage}
\caption{\textbf{Test-case \thenumtest}. Influence of the meshsize: entropy and $L^2$ distance to equilibrium}
\label{fig:tps:varsize}
\end{figure}
In Figure~\ref{fig:tps:varmesh}, we show the evolution of the discrete relative entropy and the $L^2$ distance to the equilibrium (namely, $\sqrt{\|N^n_\M -N^e_\M  \|_{L^2(\Omega)}^2 + \| P^n_\M -P^e_\M\|_{L^2(\Omega)}^2 + \| \phi^n_\M -\phi^e_\M\|_{L^2(\Omega)}^2 }$) with respect to time. First, one can note that the evolutions are  exponentially fast, as expected from Theorem~\ref{th:longtime}. One can also see a saturation phenomenon when the machine precision is reached. Moreover, as announced in Remark~\ref{rem:nonuniexp}, the decay rate is not strongly impacted by the meshsize. Last but not least, the quantitative values of the decay rate are in accordance with those obtained in~\cite[Figure 2]{BCCH:17} with a \SG TPFA scheme.

\addtocounter{numtest}{1}
In Figure~\ref{fig:tps:varmesh}, we show the evolution of the discrete relative entropy and the $L^2$ distance to the equilibrium for simulations computed with $\Delta t = 0.01$ on meshes with different geometry: a Cartesian one (64 cells, $h_\D = 3.12 \ 10 ^{-2}$), a triangular one (56 cells, $h_\D = 3.07 \ 10 ^{-2}$) and a tilted hexagonal one (76 cells, $h_\D = 3.42\ 10^{-2}$). Note that the meshes under consideration have a similar meshsize.
As for the previous results, the expected exponential decay of entropy is observed. 
The decay rates computed are almost the same for the three mesh geometries, which reinforces the previous observations: in practice, the decay rate is not impacted by the mesh used for the simulation (either by its geometry or its size).
\begin{figure}[ht]
\begin{minipage}[c]{0.49\linewidth}
\begin{tikzpicture}[scale= 0.9]
        \begin{semilogyaxis}[
            	 legend style = { 
              at={(0.5,1.1)},
              anchor = south,
              tick label style={font=\footnotesize},
              legend columns=-1
            },ylabel=\small{Discrete relative entropy},xlabel=\small{Time}
          ]
          \addplot[mark=square*, color=blue] table[x=Temps,y=Entro] {graph_data/mesh_type/mesh2_2/tps_impli};          
          \addplot[mark=triangle*, color=red] table[x=Temps,y=Entro] {graph_data/mesh_type/mesh1_1/tps_impli};
          \addplot[mark=star, color=green] table[x=Temps,y=Entro] {graph_data/mesh_type/pi6_tiltedhexagonal_2/tps_impli};
          \legend{ \tiny Cartesian ,\tiny Triangular, \tiny Tilted Hexagonal };
	      \end{semilogyaxis}
      \end{tikzpicture}    
\end{minipage}
\begin{minipage}[c]{.49\linewidth}
\begin{tikzpicture}[scale= 0.9]
        \begin{semilogyaxis}[
            legend style = { 
              at={(0.5,1.1)},
              anchor = south,
              tick label style={font=\footnotesize},
              legend columns=-1
            },ylabel=\small{$L^2$ distance to equilibrium},xlabel=\small{Time}
          ]           
          \addplot[mark=square*, color=blue] table[x=Temps,y expr=sqrt(\thisrow{Diff_eq_L2})] {graph_data/mesh_type/mesh2_2/tps_impli};
          \addplot[mark=triangle*, color=red] table[x=Temps,y expr=sqrt(\thisrow{Diff_eq_L2})] {graph_data/mesh_type/mesh1_1/tps_impli};
          \addplot[mark=star, color=green] table[x=Temps,y expr=sqrt(\thisrow{Diff_eq_L2})] {graph_data/mesh_type/pi6_tiltedhexagonal_2/tps_impli};
          \legend{ \tiny Cartesian ,\tiny Triangular, \tiny Tilted Hexagonal };
        \end{semilogyaxis}
      \end{tikzpicture}
\end{minipage}
\caption{\textbf{Test-case \thenumtest}. Influence of the mesh geometry on the long-time behaviour: entropy and $L^2$ distance to equilibrium.}
\label{fig:tps:varmesh}
\end{figure}
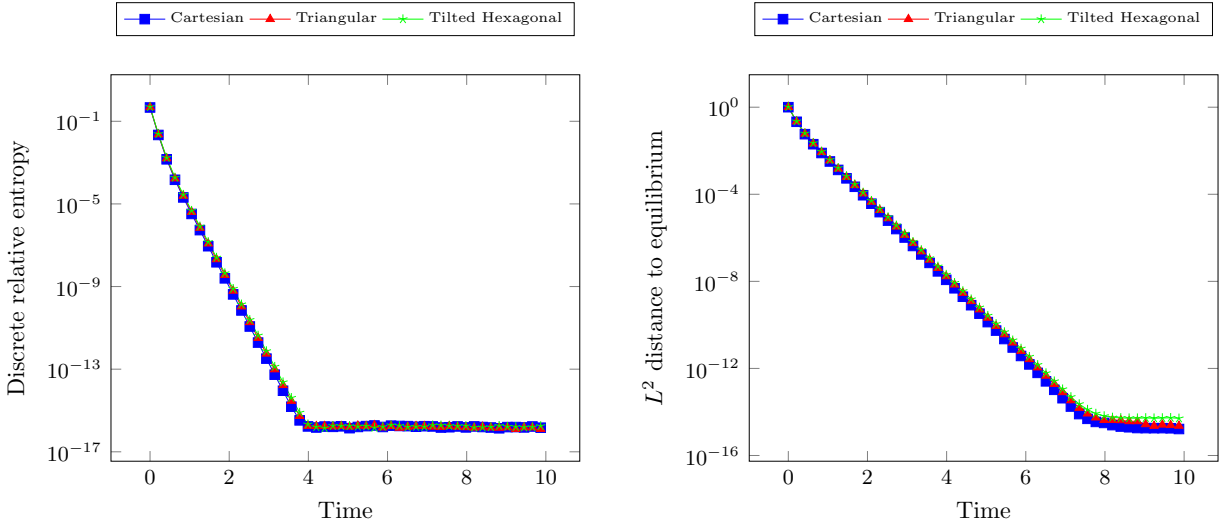
It is also remarkable to note that on the tilted hexagonal mesh, which is not adapted to the geometry of the semiconductor device, the long-time behaviour is essentially similar to the one on meshes with adapted geometry. Again, this indicates the robustness of the scheme with respect to the mesh used.

\addtocounter{numtest}{1}
Last, we investigate the influence of the magnetic field over the long-time behaviour.
\begin{figure}[ht]
\begin{subfigure}{0.49\linewidth}
\center
\begin{tikzpicture}[scale= 0.9,baseline]
        \begin{semilogyaxis}[
            legend style = { 
              at={(0.5,1.1)},
              anchor = south,
              tick label style={font=\footnotesize},
              legend columns=-1
            },ylabel=\small{Relative entropy},xlabel=\small{Time}
          ]
          \addplot table[x=Temps,y=Entro] {graph_data/Blakemore/b=0/tps_impli};
          \addplot table[x=Temps,y=Entro] {graph_data/Blakemore/b=1/tps_impli};
          \addplot table[x=Temps,y=Entro] {graph_data/Blakemore/b=2/tps_impli};
          \addplot table[x=Temps,y=Entro] {graph_data/Blakemore/b=3/tps_impli}; 
          \addplot table[x=Temps,y=Entro] {graph_data/Blakemore/b=5/tps_impli}; 
          \legend{\tiny $b = 0$,\tiny $b = 1$, \tiny $b = 2$, \tiny $b = 3$, \tiny $b = 5$} 
	      \end{semilogyaxis}
      \end{tikzpicture} 
\caption{Evolution of the relative entropy}
\label{fig:tps:varb:entro}
\end{subfigure}
\begin{subfigure}{.49\linewidth}
\center
\begin{tikzpicture}[scale= 0.9,baseline]
        \begin{semilogxaxis}[
            legend style = { 
              at={(0.5,1.1)},
              anchor = south,
              tick label style={font=\footnotesize},
              legend columns=-1
            },ylabel=\small{Decay rate},xlabel=\small{$b$}
          ] 
          \addplot[mark=pentagon*, draw=none] table[x=b,y=decay_rate] {graph_data/Blakemore/b_ev/rate};
          \addplot[mark=none, draw=red, line width=0.5pt] table[x=b,y=ref] {graph_data/Blakemore/b_ev/ref};
  		  \legend{\tiny Computed decay rate, \tiny $\nu_0 /{\sqrt{1+b^2}}$ }     
        \end{semilogxaxis}
      \end{tikzpicture}
\caption{Evolution of the decay rate with respect to $b$}
\label{fig:tps:varb:rate}
\end{subfigure}   
\caption{\textbf{Test-case \thenumtest}. Long-time behaviour for the Blakemore statistics and influence of the magnetic field}
\label{fig:tps:varb}
\end{figure}
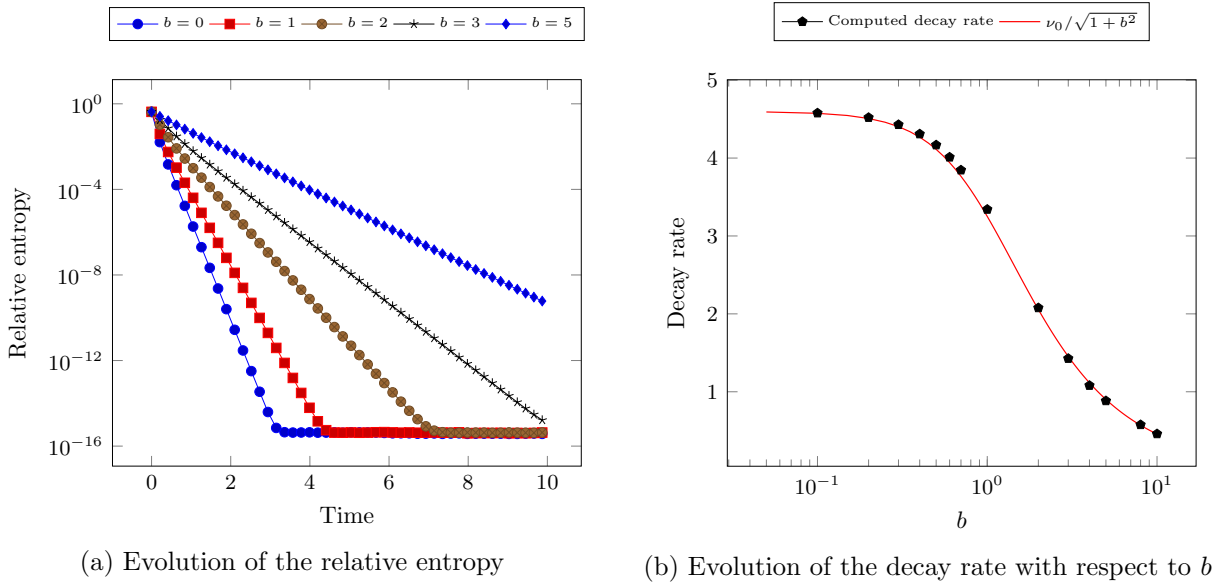
We consider the test-case~\thenumtest, with Blakemore statistics ($h(s) = \log \left (\frac{s}{1-\gamma s } \right)$), no recombination ($r=0$), $\lambda = 1$ and boundary values $N^D_0 = \e$, $N^D_1 = 1$ and $\alpha_0 = 1$. 
We perform our simulations on a Cartesian mesh, with a time step $\Delta t = 0.01$, with different values of $b$. The results are presented in Figure~\ref{fig:tps:varb:entro}. One can see the relative entropy decreases exponentially fast in time. Moreover, the presence of a magnetic field tends to attenuate the dissipative effects, and to slow down the evolution. This was expected from a physical point of view, since the magnetic field induces a rotation of the charge carriers which delays the natural relaxation towards equilibrium.
On Figure~\ref{fig:tps:varb:rate}, we show the evolution of the entropy decay rate with respect to the intensity of the magnetic field $b$. We plot the values of the decay rate for different values of $b$, as well as a reference rate $\nu(b) = \frac{\nu_0}{\sqrt{1+b^2}}$, where $\nu_0$ is the numerical value of the decay rate computed with no magnetic field ($b=0$). The computed values seem to fit very well the reference rate. Note that the decay rate seems therefore to be proportional to $\frac{1}{\sqrt{1+b^2}}$, which happens to be the modulus of the eigenvalues of the diffusion tensors $\Lambda_N$ and $\Lambda_P$.

\section{Conclusion} \label{con}
In this article, we have designed and analysed a scheme for general anisotropic drift-diffusion systems on general polytopal meshes. The scheme is based on the Hybrid Finite Volume method. We have proved that the scheme has a discrete entropic structure, and used it to show the existence of solutions (with $I_h$-valued densities). As a by-product of this structure, we have also proved that the solutions to the scheme converge exponentially fast in time towards the associated discrete thermal equilibrium.
The results are established for general statistics functions $h$, general diffusion tensors (potentially anisotropic and nonsymmetric) and general recombination terms, under the hypothesis that the boundary data are compatible with the thermal equilibrium (conditions~\eqref{eq:compcond} and~\eqref{eq:alpha}). 
Finally, we have validated our scheme on different numerical tests, highlighting the $I_h$-valuation of the discrete densities, the ability to withstand intense magnetic fields and the long-time behaviour. 
Numerical experiments (not presented in this paper) suggest that the scheme introduced here also works in situations where the compatibility condition with the thermal equilibrium does not hold. Hence, a future direction would be to analyse the scheme in such a situation.
\section*{Acknowledgements}      
The author would like to thank the anonymous reviewers for their remarks and suggestions which helped
improving the quality of this article, as well as Claire Chainais-Hillairet, Maxime Herda and Simon Lemaire for fruitful discussions and comments on this manuscript.
The author acknowledges support by the Labex CEMPI (ANR-11-LABX-0007-01).
\appendix
\section{Discrete boundedness by entropy and dissipation} \label{ap:lemma}
In this appendix, we prove Lemma~\ref{lem:bounds}. To do so, we need a technical result stated below.
\begin{lemma} \label{lem:propa}
Let $E :  (x,y) \mapsto \left (G(x) - G(y) \right ) (x-y)$, and $(x,y) \in \R^2$. \newline 
If there exist two positive constants $M$ and $C_1$ such that 
$\displaystyle
	| y| \leq M \text{ and } 0 \leq E(x, y) \leq C_1,
$
then, there exists a constant $C_{M,C_1}$ only depending on $M$ and $C_1$ such that $|x| \leq C_{M,C_1}$.
\end{lemma}
\begin{proof}
First, we define $E_ y : \delta \mapsto \left( G(y+\delta) - G( y) \right ) \delta$, therefore letting $\delta = x- y$, one has 
$\displaystyle
E(x, y) = E_ y(\delta) \leq C_1.
$
The inequality $|x| \leq M + |\delta|$ holds, so it suffices to get a bound on $\delta$ to conclude:
\begin{itemize}
	\item[(i)] if $\delta \geq 2M+1>0$, then since $G$ is increasing, one has  
\[
	G(y+\delta) \geq G(y + 2M +1)\geq G(-M + 2M +1)= G(M+1) \text{ and } G(y) \leq G(M), 
\]
so $C_1 \geq E_y(\delta) \geq \left (G(M+1) - G(M)  \right ) \delta$, and we get $ 2M+1 \leq \delta \leq C_1 \left ( G(M+1) - G(M) \right )  ^{-1}$; 
	\item[(ii)] if $\delta\leq-2M-1$, a similar computation shows that $ -2M-1 \geq \delta \geq C_1 \left ( G(-M-1) - G(-M) \right )  ^{-1}$.
\end{itemize}
Hence, one has 
$\displaystyle
	|\delta| \leq \max \left ( 2M+1, \frac{C_1}{G(M+1)-G(M)}, \frac{C_1}{G(-M)-G(-M-1)} \right )
$, which concludes the proof.
\end{proof}
We can now prove Lemma~\ref{lem:bounds}. The proof is similar to \cite[Lemma 2]{CHHLM:22}, except for the use of the unknowns on $\E_{ext}^D$.
\lem*
\begin{proof}
The proof is divided into two steps. First, we prove bounds on the discrete electrostatic potential $\phid_\D$ thanks to the bound on the entropy. Then, we use these bounds to estimate $(\w_\D^N , \w_\D^P)$.

Since $\N(\w_\D^N, \w_\D^P)  \leq B_\sharp$, by~\eqref{def:globalcoercivity} one has
$\displaystyle 
	\frac{\alpha_\flat \lambda_\flat ^\phi}{2} |\phid_\D - \phid_\D^e|_{1,\D}^2 
		\leq \frac{1}{2}a_\D^\phi(\phid_\D - \phid_\D^e, \phid_\D - \phid_\D^e) \leq B_\sharp$.
Therefore, letting $c_1 = |\phid_\D^e |_{1,\D} + \sqrt{\frac{2 B_\sharp}{\alpha_\flat \lambda_\flat ^\phi}}$, we get that $|\phid_\D|_{1,\D} \leq  c_1$. By definition of $|.|_{1,\D}$, we deduce that, for any $K \in \M$ and any $\s \in \E_K$, 
$\frac{|\s|}{d_{K ,\s}} (\phi_K - \phi_\s)^2 \leq c_1^2$. Letting $c_2 = \displaystyle\max_{K\in \M, \s \in \E_K}{c_1 \sqrt{\frac{d_{K,\s}}{|\s|}}}$, one has 
\[
	\forall K \in \M, \forall \s \in \E_K, \, |\phi_K - \phi_\s | \leq c_2.
\]
On the other hand, there exists $\s_0 \in \E^D_{ext}$ such that $\phi_{\s_0} = \phi^D_{\s_0}$ and so, by \eqref{eq:boundlifting}, we have
\[
	| \phi_{\s_0}|  \leq \| \phi^D \|_{L^\infty(\Omega)}.
\]
Now, one can use the connectivity of the mesh: for any cell $K$ (\emph{resp.} face $\s$) there is a finite sequence of components of $\phid_\D$, denoted $(x_k)_{0 \leq k \leq l}$, starting at $x_0 = \phi_{\s_0}$ and finishing at $x_l = \phi_K$ (\emph{resp.} $x_l = \phi_\s$) such that, for any $k$ in $\lbrace 0, \dots l-1 \rbrace$, one has  $|x_{k+1} - x_{k}| \leq c_2$. Therefore, one concludes that 
\[
	|x_l| \leq l  c_2 + |x_0| \leq 2 |\M| c_2 + \| \phi^D \|_{L^\infty(\Omega)}.
\]
Thus there exists $\phi_\sharp$ positive depending on the data and $\D$ such that 
\[
	-\phi_\sharp \one_\D\leq \phid_\D  \leq  \phi_\sharp \one_\D.
\]

Now, note that since $\Diss(\w_\D^N, \w_\D^P)  \leq B_\sharp$, one has $T_\D^N (\Nd_\D, \w_\D^N, \w_\D^N) \leq B_\sharp$ and $ T_\D^P (\Pd_\D, \w_\D^P, \w_\D^P) \leq B_\sharp$, 
where we recall that the discrete densities are defined by $\Nd_\D= g(\w_\D^N + \phid_\D + \alpha_N \one_\D) $ and
$\Pd_\D= g(\w_\D^P - \phid_\D + \alpha_P \one_\D) $.
For $K\in\M$, using~\eqref{def:TK} alongside with the positivity of $r_K(\Nd_K)$ and the local coercivity~\eqref{def:localcoercivity} of $a^N_K$, one gets that 
\begin{equation} \label{eq:bounddiss}
	B_\sharp \geq r_K(\Nd_K) a_K^N(\w_\D^N, \w_\D^N) 
	\geq \alpha_\flat \lambda_\flat  \sum_{\s \in \E_K} r_K(\Nd_K) \frac{|\s|}{d_{K,\s}}(w_K^N - w_\s^N)^2.
\end{equation}
Given $\s \in \E_K$, by definition of $r_K$ and $\Nd_K$, using~\eqref{hyp:m} and the positivity of $m$, we obtain 
\[
	r_K(\Nd_K) 
		= \frac{1}{|\E_K|}\sum_{\s'\in \E_K} m \left (N_K, N_{\s'} \right )
		\geq \frac{1}{|\E_K|} m_h \left (g(w_K^N + \phi_K + \alpha_N), g(w_\s^N + \phi_\s + \alpha_N) \right ).
\]
Recall that formula~\eqref{eq:mhg} holds, so 
\[
	r_K(\Nd_K) \geq \frac{1}{|\E_K|} \frac{G(w_K^N + \phi_K + \alpha_N)- G(w_\s^N + \phi_\s + \alpha_N) }
		{(w_K^N + \phi_K + \alpha_N) - (w_\s^N + \phi_\s + \alpha_N)}.
\]
Moreover, since $G$ is convex, $(x,y) \mapsto\frac{G(x)-G(y)}{x-y}$ is non-decreasing with respect to both its variables,
therefore using the bound on $\phid_\D$ proved previously we get
\begin{align*}
	r_K(\Nd_K) &\geq \frac{1}{|\E_K|} \frac{G(w_K^N - \phi_\sharp + \alpha_N)- G(w_\s^N - \phi_\sharp + \alpha_N) }
		{(w_K^N - \phi_\sharp + \alpha_N) - (w_\s^N - \phi_\sharp + \alpha_N)} \\
		 &= \frac{1}{|\E_K|} \frac{G(w_K^N - \phi_\sharp + \alpha_N)- G(w_\s^N - \phi_\sharp + \alpha_N) }
		{w_K^N - w_\s^N } .
\end{align*}
%
%
Using \eqref{eq:bounddiss}, one gets that for any $K \in \M$ and for any $\s \in \E_K$, 
\[
	0\leq \left ( G(w_K^N - \phi_\sharp + \alpha_N)-G(w_\s^N - \phi_\sharp + \alpha_N) \right ) ( w_K^N -w_\s^N) 
		\leq\zeta
\]
where $\zeta = \frac{B_\sharp}{\alpha_\flat \lambda_\flat}  \displaystyle\max_{K \in \M,\ \s \in \E_K} \frac{d_{K,\s}|\E_K|}{ |\s|} $. Now, let $\u_\D = \w_\D^N + (-\phi_\sharp + \alpha_N) \one_\D \in \V_\D$. It is clear that it suffices to get bounds on $\u_\D$ to bound $\w_\D^N$. The previous relation writes 
\begin{equation} \label{eq:propa}
	\forall K \in \M, \ \forall \s \in \E_K, \ 0 \leq  E(u_K,u_\s) \leq \zeta, 
\end{equation}
where $E:\R^2 \to \R_+$ is the function defined in Lemma \ref{lem:propa}.
Now, notice that since $\w_\D^N \in \V_{\D,0}$ there exists $\s_{0} \in \E^D_{ext}$ such that $u_{\s_{0}} = -\phi_\sharp + \alpha_N$. One can then use the connectivity of the mesh to conclude: 
for any cell $K$ (\emph{resp.} face $\s$) there is a finite sequence of components of $\u_\D$, denoted $(x_k)_{0 \leq k \leq l}$, starting at $x_0 = u_{\s_0}$ and finishing at $x_l = u_K$ (\emph{resp.} $x_l = u_\s$) such that for any $k$ in $\lbrace 0, \dots , l-1 \rbrace$, one has $E(x_{k+1}, x_{k}) \leq \zeta$. Therefore, by $l$ successive applications of Lemma \ref{lem:propa}, we get the existence of some $c^N>0$, depending on $\zeta$, $l$, and $x_0 =  -\phi_\sharp + \alpha_N$ such that $|x_l| \leq c^N$. Thus, there exists $C^N_\sharp >0$ depending on the data, $\D$ and $B_\sharp$ such that 
\[
	\forall K \in \M, \ \forall \s \in \E, \quad |w^N_K| \leq C^N_\sharp \text{ and } |w^N_\s| \leq C^N_\sharp .
\]
Using the same strategy, one gets bounds for $\w_\D^P$, which concludes the proof. 
\end{proof}

\bibliographystyle{siam} 
\bibliography{Bib_SC}
\end{document}